\documentclass[10pt]{amsart}
\usepackage[utf8]{inputenc}
\usepackage{hyperref}
\usepackage[letterpaper,centering,text={5.5in,8.5in}]{geometry}
\usepackage[all]{xy}
\usepackage{log-cy}
\usepackage{graphicx}

\usepackage{setspace}

\title{On the Symplectic Cohomology of Log Calabi--Yau Surfaces}
\author{James Pascaleff} 
\thanks{The author was partially supported by NSF grant DMS-0636557.}
\address{Department of Mathematics, University of Texas at Austin}
\email{jpascale@illinois.edu}
\curraddr{University of Illinois at Urbana-Champaign}

\begin{document}

\begin{abstract}
This article studies the symplectic cohomology of affine algebraic surfaces that admit a compactification by a normal crossings anticanonical divisor. Using a toroidal structure near the compactification divisor, we describe the complex computing symplectic cohomology, and compute enough differentials to identify a basis for the degree--zero part of the symplectic cohomology. This basis is indexed by integral points in a certain integral affine manifold, providing a relationship to the theta functions of Gross--Hacking--Keel. Included is a discussion of wrapped Floer cohomology of Lagrangian submanifolds and a description of the product structure in a special case. We also show that, after enhancing the coefficient ring, the degree--zero symplectic cohomology defines a family degenerating to a singular surface obtained by gluing together several affine planes.
\end{abstract}

\maketitle
\tableofcontents

\section{Introduction}
\label{sec:intro}

Log Calabi-Yau manifolds (see Definition \ref{def:intro-log-cy} below) have a rich symplectic geometry. They are the
subject of a mirror symmetry conjecture: if $U$ is a log Calabi-Yau manifold, then in favorable situations,
there is a mirror variety $U^\vee$ such that the Floer theory of $U$ is reflected in the algebraic geometry of $U^\vee$. For
instance, as conjectured by Gross--Hacking--Keel \cite{ghk}, one
expects to find the ring of regular functions on $U^\vee$ sitting
inside the symplectic cohomology of $U$. Furthermore, the symplectic
cohomology should come with a natural basis corresponding to a
collection of ``theta functions'' on $U^\vee$
\cite{ghk,gs-theta}. This article contains results towards this
conjecture in complex dimension two.

Symplectic cohomology is a version of Hamiltonian Floer homology that applies to open symplectic manifolds. In particular, given a smooth complex affine variety $U$, we can equip $U$ with an exact symplectic structure coming from a projective compactification $Y$ by an ample divisor $D$, which exists by a result of Hironaka. The symplectic cohomology does not depend on the choices involved \cite[\S 4b]{seidel-biased}, so we may speak of the symplectic cohomology $SH^*(U)$ of the affine variety $U$.

Let us briefly describe the geometry that is involved in the definition of $SH^*(U)$. It is the the cohomology of a cochain complex $SC^*(U)$, which for an appropriate choice of Hamiltonian function, contains both the ordinary cohomology of $U$ as well as generators corresponding to periodic Reeb orbits in a contact hypersurface at infinity. Both the cochain complex $SC^*(U)$ and its cohomology $SH^*(U)$ are in general infinite dimensional, even in a single degree. The differential on the complex counts maps of cylinders into $U$ satisfying an inhomogeneous pseudo-holomorphic map equation.

Symplectic cohomology is interesting from several points of view in symplectic topology. 
One application is the Weinstein conjecture \cite{weinstein-conjecture}: if a certain canonical map $H^*(U) \to SH^*(U)$ is not an isomorphism, then the contact hypersurface at infinity must contain a periodic Reeb orbit. However, the motivating interest in this paper comes from homological mirror symmetry. Symplectic cohomology is the closed--string sector of a two--dimensional field theory that forms the A--side of one version of HMS for an open symplectic manifold $U$, whose open--string sector is the wrapped Fukaya category of Abouzaid--Seidel \cite{as-open-string}; Ganatra's thesis \cite{ganatra-thesis} is a deep study of this relationship between open and closed string sectors. The B--side is an appropriate category of sheaves $\cC$ on a dual space $U^\vee$, for which the closed--string sector is the Hochschild cohomology $HH^*(\cC)$.

This line of thinking leads to interesting predictions about the symplectic cohomology of $U$ in terms of the algebraic geometry of $U^\vee$. We beg the reader to bear with this rather conjectural paragraph. Let us adopt the ansatz that the mirror $U^\vee$ is a smooth affine algebraic variety over a base field $\K$. HMS suggests that the derived wrapped Fukaya category $D^b\cW(U)$ is equivalent to the derived category of coherent sheaves $D^b\Coh U^\vee$, and hence that their Hochschild cohomologies $SH^*(U)$ and $HH^*(U^\vee)$ are isomorphic. But now
\begin{equation}
  HH^p(U^\vee) \cong \bigoplus_{i=0}^pH^i(U^\vee,\Lambda^{p-i}T_{U^\vee}) \cong H^0(U^\vee,\Lambda^pT_{U^\vee}),
\end{equation}
where the first isomorphism is a Hochschild--Kostant--Rosenberg-type isomorphism (\cite[p.~131]{hms} citing \cite{gerstenhaber-schack}), and the second holds because of the ansatz that $U^\vee$ is affine (Cartan--Serre Theorem B, \cite[Theorem III.3.7]{hartshorne}). Setting $p = 0$, 
\begin{equation}
  \label{eq:sh-hms}
  SH^0(U) \cong HH^0(U^\vee) \cong H^0(U^\vee, \cO_{U^\vee}).
\end{equation}
Since $U^\vee$ is affine, $U^\vee = \Spec H^0(U^\vee, \cO_{U^\vee})$. Thus we find
\begin{equation}
  U^\vee \cong \Spec SH^0(U),
\end{equation}
giving us a way to reconstruct the mirror $U^\vee$ from the symplectic cohomology of $U$.

This argument raises some basic questions. First of all, it requires that $SH^*(U)$ be a ring. This ring structure is well-known in Hamiltonian Floer homology as the pair-of-pants product. More surprisingly, it implies that $SH^*(U)$ is $\Z$--graded. The $\Z$--grading may not seem like a central issue from a symplectic topology point of view, but it is closely related to the Calabi--Yau nature of mirror symmetry. The symplectic geometry definition of $SH^*(U)$ yields only a $\Z/2\Z$--grading, and in general this can be lifted to a $\Z$--grading when $c_1(U) = 0$. Even then, the actual $\Z$--grading depends on a trivialization of the canonical bundle of $U$. Therefore, in order to make the isomorphism \eqref{eq:sh-hms} work, the class of symplectic manifolds $U$ under consideration must be restricted and extra data may need to be included on the A--side.

The relevant definition here is that of a log Calabi--Yau manifold. In this paper we consider the case of surfaces only. 
\begin{definition}
\label{def:intro-log-cy}
  Let $U$ be a smooth complex quasi-projective surface. $U$ is \emph{log Calabi--Yau} if there is a smooth projective surface $Y$ containing an at worst nodal anticanonical divisor $D$ such that $U \cong Y \setminus D$. If $D$ is actually nodal, $(Y,D)$ is known as a \emph{Looijenga pair}, or is said to have \emph{maximal boundary}.
\end{definition}
In this definition, by ``at worst nodal divisor'' we include the condition that $D$ is effective and reduced (each component has multiplicity equal to one); see Section \ref{sec:log-cy-surfaces} for a fuller explanation of these conditions.
It turns out that the divisor $D$ in this situation must either be a smooth genus one curve, a nodal genus one curve, or a cycle of rational curves meeting at nodes (Lemma \ref{lem:divisor-topology}).

Since $D$ is anticanonical there is an isomorphism $\omega_Y(D) \cong \cO_Y$, where $\omega_Y$ is the sheaf of holomorphic $2$-forms. Thus there is a unique-up-to-scalar section $\Omega \in H^0(Y,\omega_Y(D))$, which is a meromorphic $2$-form whose divisor of poles and zeros is exactly $-D$. Since $D$ is effective and reduced, $\Omega$ is a meromorphic volume form on $Y$ with simple poles along $D$ (and no other zeros or poles). Thus the restriction $\Omega|_U$ is a nowhere-vanishing complex volume form on $U$ that gives a specific trivialization of the canonical bundle, showing $c_1(U) = 0$. It turns out that the homotopy class of this trivialization does not depend on the choice of compactification $(Y,D)$ (Lemma \ref{lem:volume-form}).

To illustrate the sort of restriction this places on $U$, consider log Calabi--Yau manifolds of complex dimension one. Besides elliptic curves, which satisfy the definition with $U = Y$ and $D = \emptyset$, the condition that $D$ must be effective and anticanonical implies $Y \cong \PP^1$, and that the support of $D$ is at most two points. Thus $U$ is either $\C$ or $\C^\times$. Both of these satisfy $c_1(U) = 0$, as does any punctured curve, but only $U = \C^\times$ (with $D = 0 + \infty$) is log Calabi--Yau in our definition, since the divisor $D = 2\infty$ is not reduced. In terms of symplectic cohomology, the contrast is stark, as $SH^*(\C) = 0$ (which is the cohomology of the empty set), while $SH^*(\C^\times) \cong \K[x,x^{-1},\theta]$, where $x$ is in degree zero, $\theta$ is in degree one, and $\theta^{2}=0$. Observe that $SH^0(\C^\times)$ is isomorphic to $\K[x,x^{-1}]$, the coordinate ring of the expected mirror $\GG_m$.

In this paper, we make two other assumptions, namely that $U$ is affine, and so is an exact symplectic manifold with a well-defined symplectic cohomology, and also that the complex dimension is two, which is the first interesting low--dimensional case. This case was also studied extensively in an algebro-geometric context by Gross--Hacking--Keel \cite{ghk}, who proposed that the degree zero symplectic cohomology should be the coordinate ring of the mirror, in favorable cases. They approach the construction of the mirror from another direction, and actually define the coordinate ring of the mirror using tropical geometry (an essentially combinatorial theory). Our goal is to verify some of their predictions in terms of symplectic geometry and Floer theory, so we actually work with maps of Riemann surfaces into the symplectic manifold. 

The first step towards verifying the Gross--Hacking--Keel prediction is to compute the symplectic cohomology additively. In their work \cite{ghk} they find that $\Gamma(U^\vee,\cO_{U^\vee})$ is additively generated by a basis $\{\theta_p \mid p \in U^\trop(\Z)\}$, where the elements are indexed by integral points of a certain integral affine manifold $U^\trop$ associated to $U$ (we also describe a construction of this manifold). These elements $\theta_p$ are the so-called theta functions on the mirror $U^\vee$ \cite{ghk,gs-theta}.

By actually looking at the periodic Hamiltonian orbits used to define symplectic cohomology, we also find an additive basis for the degree zero part of symplectic cohomology whose indexing set is in correspondence with $U^\trop(\Z)$, via a geometrically natural bijection.

\begin{theorem}
\label{thm:main}
Let $U$ be an affine log Calabi--Yau surface with maximal boundary. We
construct a set of cochains $\{\theta_p \mid p \in
U^\trop(\Z)\}$ having degree zero in the $\Z$--grading on
symplectic cohomology induced by a log Calabi--Yau compactification
$(Y,D)$. They are closed and form a basis of the degree zero cohomology:
  \begin{equation}
    SH^0(U) = \Span \{\theta_p \mid p \in U^\trop(\Z)\}
  \end{equation}
\end{theorem}

  The element $\theta_p$ is defined with reference to a particular Liouville structure on $U$, which is deformation equivalent to the Stein structure, and a particular Hamiltonian $H: U \to \R$. With this in mind, $\theta_p$ a degree zero element arising from perturbation of a periodic torus of the Hamiltonian flow. Such tori are in a geometrically natural bijection with $U^\trop(\Z)$.

Let us outline briefly how this computation goes. In section \ref{sec:construction}, we construct a Liouville structure on $U$, which admits a Lagrangian torus fibration over the cylindrical end of $U$. This construction is based largely on \cite[\S 4]{seidel-biased}, with some tweaks using the extra symmetry of our situation. This structure has another convenient property, namely the contact manifold on which the cylindrical end is modeled satisfies a certain ``convexity'' condition.\footnote{This condition is independent of the contact condition, which may also be regarded as a convexity condition.} With an appropriate Hamiltonian, the periodic orbits can be explicitly described. In fact, there are entire tori that are periodic for the Hamiltonian flow, so this is a Morse--Bott situation. After small time-dependent perturbation of the Hamiltonian, these tori break up into several non-degenerate orbits of various degrees. There is a Morse-Bott spectral sequence converging to the symplectic cohomology whose $E_1$ page reproduces the cohomologies of the periodic tori. We show that the degree zero element for each torus is closed for all higher differentials in the spectral sequence, and is never exact. The corresponding cochain is what we call $\theta_p$.

The crucial point is to understand the differential, and in particular to show that $\theta_p$ is closed, so that it is actually a cocycle rather than just a cochain element. For the definitions see section \ref{sec:sh-hol-curves}. The convexity condition is important for understanding the holomorphic curves that contribute to the differential, as it actually allows us to show that certain moduli spaces are empty for energy reasons. At this point in the argument we use some ideas coming from symplectic field theory, adapted to the framework of Hamiltonian Floer homology, namely a neck-stretching argument due to Bourgeois and Oancea \cite{bo-exact-sequence}, and an adaptation of a technique developed by Bourgeois and Colin \cite{bourgeois-colin} to compute contact homology of toroidal manifolds. To get from these techniques to the algebraic fact that $\theta_p$ is closed, we use the way the differential interacts with the Batalin-Vilkovisky operator and the pair-of-pants product, so we also obtain some limited information about these operations as well. This is done in section \ref{sec:differential}.

Above we described how HMS applied to Hochschild cohomology leads to the expectation that $SH^0(U)$ is isomorphic the ring of global functions on the mirror $U^\vee$. However, there is another, more obvious way to extract this ring from the category of coherent sheaves on $U^\vee$, namely as $\Hom(\cO_{U^\vee},\cO_{U^\vee})$. Under the ansatz that $U^\vee$ is affine, there are no higher $\Ext$--groups, so we should expect to be able to find a Lagrangian submanifold $L$ in $U$ corresponding to $\cO_{U^\vee}$, having wrapped Floer cohomology $HW^*(L,L)$ concentrated in degree zero, and such that $SH^0(U)$ is isomorphic to $HW^0(L,L)$. In section \ref{sec:wrapped} we consider Lagrangian submanifolds with these properties, which are candidates for the mirror to $\cO_{U^\vee}$. In the case where $U$ is the complement of a smooth conic in $\C^2$, we can combine the results of this paper with those of \cite{binodal} to understand the ring structure and show that
\begin{equation}
  SH^0(U) \cong \K[x,y][(xy-1)^{-1}]
\end{equation}

Section \ref{sec:log-cy-surfaces} contains some basic results on log Calabi-Yau surfaces. Section \ref{sec:affine-man} describes the affine manifold $U^\trop$ from a topological viewpoint. 

The paper is organized so that sections \ref{sec:sh-hol-curves}, \ref{sec:construction}, and \ref{sec:differential} are a continuous thread of argument. Sections \ref{sec:log-cy-surfaces} and \ref{sec:affine-man} can be read as interludes describing the algebro-geometric and piecewise-linear context of our study.

We conclude the paper with section \ref{sec:coeff-degeneration} containing some results on the product structure that provide a closer connection to \cite{ghk}. Here, for a chosen compactification $Y$, we define a certain strictly convex cone $P \subset H_2(Y;\Z)$ containing the cone of effective curves in $Y$, and we show how to enhance the coefficient ring of $SH^0(U)$ to the monoid ring $\K[P]$. This makes $\Spec SH^0(U)$ into a family over $\Spec \K[P]$. We show that the central fiber of this family is isomorphic to a singular surface $\mathbb{V}_n$ (called \emph{the vertex}) consisting of $n$ copies of $\Af^2$ glued together along coordinate axes in a cycle. Thus the symplectic cohomology of $U$ provides a deformation of $\mathbb{V}_n$. An analogous result was obtained in \cite{ghk} using theta functions. In fact, along the way to proving this result, we use a symplectic-topological imitation of the \emph{broken lines} of \cite{ghk}, which may be of independent interest. This is done in section \ref{sec:drawing}, where we associate to a holomorphic curve in the cylindrical end of $U$ a graph (or tropical curve) in the affine manifold $U^\trop$.

\subsection{Acknowledgments}
\label{sec:acknowledgments}

The author wishes to thank Sean Keel for many helpful discussions of his joint work with Mark Gross and Paul Hacking, which formed the background of this work. He also thanks Mohammed Abouzaid, Matthew Strom Borman, Vincent Colin, Luis Diogo, Sheel Ganatra, Ailsa Keating, Mark McLean, Tim Perutz, and Paul Seidel for many helpful discussions and suggestions on techniques to compute symplectic cohomology. In particular, the use of the Batalin--Vilkovisky operator that appears in Proposition \ref{prop:prim-closed} was suggested by Paul Seidel. He also thanks the referee for a very careful reading of the manuscript that led to improvements in many places. This work was done while the author held an RTG postdoctoral fellowship (NSF grant DMS-0636557).

\section{Symplectic cohomology and holomorphic curves}
\label{sec:sh-hol-curves}

This section reviews material about pseudo-holomorphic curves in the case of manifolds with contact-type boundary and cylindrical ends, and a description of the versions of Floer cohomology that we use. We hope that this section will make the paper more self-contained for algebro-geometrically minded readers.

\subsection{Conventions}
\label{sec:conventions}


In this section, we set out the conventions for symplectic manifolds and symplectic cohomology. 
In large part our conventions follow \cite{seidel-biased,ritter-tqft}.

\begin{definition}
  An \emph{exact symplectic form} on a manifold $M$ is an exact non-degenerate two-form $\omega$. Thus 
\begin{equation}
  \omega = d\lambda
\end{equation}
for some one-form $\lambda$, which is called a \emph{Liouville one-form}. The corresponding \emph{Liouville vector field} $Z$ is defined by duality with respect to $\omega$:
\begin{equation}
  \iota_Z \omega = \omega(Z,\cdot) = \lambda
\end{equation}
Thus the one-form $\lambda$ determines the two-form $\omega$ and the vector field $Z$. 

A \emph{Liouville domain} \cite{seidel-biased} is a compact manifold with boundary, equipped with a one-form $\lambda$, such that the two-form $\omega = d\lambda$ is symplectic, and the Liouville vector field $Z$ points strictly outward along the boundary.
\end{definition}

\begin{definition}
  Let $L \subset M$ be a Lagrangian submanifold. The Liouville one-form $\lambda$ defines a class $[\lambda|_L] \in H^1(L,\R)$, called the \emph{Liouville class}.
\end{definition}

\begin{definition}
  Let $(M,\omega)$ be a symplectic manifold. An almost complex structure $J$ on $M$ is \emph{compatible} with $\omega$ if the bilinear form $g$ defined by
  \begin{equation}
    g(X,Y) = \omega(X,JY)
  \end{equation}
  is symmetric and positive definite at every point of $M$. Thus $g$ is a Riemannian metric associated to the choice of $J$.
\end{definition}

\begin{definition}
  \label{def:hamiltonian}
  If $H: M \to \R$ is a differentiable function, the \emph{Hamiltonian vector field} $X_H$ associated to $H$ is defined by the relation
  \begin{equation}
    -\iota_{X_H}\omega = \omega(\cdot,X_H) = dH
  \end{equation}
In the presence of a compatible almost complex structure $J$ and associated metric $g$, we may take the gradient $\nabla H$ with respect to $g$. As a consequence of our conventions, this is connected to $X_H$ by
\begin{equation}
  X_H = J\nabla H
\end{equation}
\end{definition}

\begin{example}
  Let $M = \C$ be the complex affine line, coordinatized by $z = x+iy$. We take the Euclidean structures $\omega = \frac{i}{2}dz\wedge d\bar{z} = dx\wedge dy$, $J =$ multiplication by $i$, $g = dx^2 + dy^2$. For $\lambda$ we choose
  \begin{equation}
    \lambda = d^c\left(\frac{1}{4}|z|^2\right) = \frac{i}{4}(z\,d\bar{z} - \bar{z}\,dz) = \frac{1}{2}(x\,dy-y\,dx)
  \end{equation}
  where we have taken advantage of the complex analytic structure of $M$ to write $\lambda$ in terms of a K\"{a}hler potential. Here $d^c = -i(\partial - \bar{\partial})$ as in, for example, \cite{wells}.

The corresponding Liouville vector field is 
\begin{equation}
  Z = \frac{1}{2}\left(z\frac{\partial}{\partial z}+\bar{z}\frac{\partial}{\partial\bar{z}}\right) = \frac{1}{2}\left(x\frac{\partial}{\partial x}+y\frac{\partial}{\partial y}\right)
\end{equation}
Since the vector field $Z$ points radially outward, we find that any disk $\{|z| \leq R\}$ becomes a Liouville domain when equipped with the restrictions of these structures.

Now consider the function $H = \frac{1}{2}(x^2+y^2)$. We have
\begin{align}
\nabla H &= x\frac{\partial}{\partial x} + y\frac{\partial}{\partial y}\\
X_H &= J\nabla H  = -y\frac{\partial}{\partial x} + x\frac{\partial}{\partial y}
\end{align}
Thus the flow of $X_H$ rotates the plane about the origin in the counterclockwise direction with period $2\pi$.
\end{example}

\begin{lemma}
  If $(M, \lambda)$ is a Liouville domain of real dimension $2n$, then $\alpha \in \Omega^1(\partial M)$ defined by
  \begin{equation}
    \alpha = \lambda |_{\partial M}
  \end{equation}
is a contact one-form on the boundary $\partial M$, which is to say $\alpha\wedge(d\alpha)^{2n-2}$ is a volume form on $\partial M$.
\end{lemma}


\begin{definition}
  Let $(N,\alpha)$ be a contact manifold with contact one-form $\alpha$. We have the contact distribution $\xi = \ker \alpha$. The \emph{Reeb vector field} is defined by the conditions $\iota_Rd\alpha = 0$, $\alpha(R)= 1$. The \emph{symplectization} of $(N,\alpha)$ is an exact symplectic structure on $N \times \R$. Letting $\rho$ denote the coordinate on the $\R$ factor, the Liouville one-form, symplectic form, and Liouville vector field are
  \begin{align}
    \lambda &= e^\rho\alpha\\
    \omega &= e^\rho(d\rho\wedge \alpha + d\alpha)\\
    Z &= \frac{\partial}{\partial \rho}
  \end{align}
Symplectic manifolds of the form $N\times \R$ are also called \emph{cylindrical symplectic manifolds}.

An almost complex structure $J$ on $N\times \R$ is \emph{cylindrical} if it is invariant under translation in the $\rho$-direction and it respects the product structure in the following way: with respect to the decomposition
\begin{equation}
\label{eq:tangent-splitting}
  T(N\times \R) = TN \oplus \langle Z\rangle = \xi \oplus \langle R \rangle \oplus \langle Z \rangle
\end{equation}
we require that $J$ preserves $\xi$ and $J(Z) = R$. Thus $J|_\xi$ is an almost complex structure on $\xi$. The full structure $J$ is compatible with $\omega$ if and only if $J|_\xi$ is compatible with $d\alpha|_\xi$, which is a symplectic form on $\xi$.

 A non-compact exact symplectic manifold $M$ is said to have a \emph{cylindrical end} if there is a compact set $K$ such that $M \setminus K$ is isomorphic to the positive part $N \times \R_+$ of a cylindrical symplectic manifold. An almost complex structure on such an $M$ is also called \emph{cylindrical} if it satisfies the above conditions on the end only.
\end{definition}

A basic fact is that any Liouville domain $M$ may be converted into a manifold with a cylindrical end by attaching a copy of the positive part of the symplectization of the contact boundary $\partial M \times \R_+$ \cite[(2a)]{seidel-biased}.

\begin{definition}
  A $(n-1)$--dimensional submanifold $\Lambda \subset N$ is called
  \emph{Legendrian} if $T\Lambda \subset \xi = \ker \alpha$. A
  $n$--dimensional submanifold $L \subset N$ is called
  \emph{pre-Lagrangian} if $d\alpha|_L = 0$. The one-form $\alpha$
  then defines a class $[\alpha|_L] \in H^1(L,\R)$ also called the
  \emph{Liouville class}. Observe that in the symplectization $L$
  lifts to $L \times \{0\} \subset N \times \R$ which is a Lagrangian
  submanifold with the same Liouville class.
\end{definition}

\begin{lemma}
  Let $N \times \R$ be a cylindrical symplectic manifold, and let $H : N \times \R \to \R$ be a function which depends on $\rho$ only. Thus $H(x,\rho) = h(e^\rho)$ for some function $h : (0,\infty) \to \R$. Then the Hamiltonian vector field $X_H$ is tangent to each slice $N\times \{\rho\}$, and is proportional to the Reeb vector field $R$:
  \begin{equation}
    X_H = h'(e^\rho)R
  \end{equation}
If $J$ is a compatible cylindrical almost complex structure, with corresponding metric $g$, then $X_H = J\nabla H$, where 
\begin{equation}
  \nabla H = h'(e^\rho)\frac{\partial}{\partial \rho}
\end{equation}
\end{lemma}


\begin{definition}
  Let $(M,\lambda)$ be an exact symplectic manifold, and $H: M \to \R$ a Hamiltonian function. The symplectic action of a loop $\gamma: \R/\Z \to M$, whose domain is parametrized by $t \in [0,1)$, is given by
  \begin{equation}
    \mathcal{A}(\gamma) = -\int_{\R/\Z}\gamma^*\lambda + \int_0^1 H(\gamma(t))\,dt
  \end{equation}
  The critical points of this action functional are those loops $\gamma$ such that $-\iota_{\dot{\gamma}}d\lambda = dH$, which in light of definition \ref{def:hamiltonian} means $\dot{\gamma} = X_H$. In other words, the critical points are 1-periodic orbits of $X_H$.
\end{definition}
\begin{remark}
  More generally we may consider a time-dependent Hamiltonian function $H : \R/\Z \times M \to \R$, which can be thought of as a family of Hamiltonians $H_{t} : M \to \R$ depending on $t \in \R/\Z$. Then the Hamiltonian vector field $X_{H}$ is a time-dependent vector field whose value at time $t$ is $X_{H_{t}}$. There is also a version of the action functional $\mathcal{A}$ in this context, where the second term has an explicit $t$-dependence. The correspondence between critical points of $\mathcal{A}$ and time 1 periodic orbits of $X_{H}$ still holds.
\end{remark}

\subsection{Holomorphic curves}
\label{sec:hol-curves}

In this section we recall some elementary facts about inhomogeneous pseudo-holomorphic maps that will be used in the paper. Throughout, let $C$ be a Riemann surface with complex structure $j$. 
\begin{definition}
  Let $(M,J)$ be an almost complex manifold. A map $u: C \to M$ is \emph{pseudo-holomorphic} if $J \circ du = du \circ j$.
\end{definition}

\begin{definition}
  Let $(M,\omega,J)$ be a symplectic manifold with compatible almost complex structure. Let $H: M\to \R$ be a Hamiltonian function, with Hamiltonian vector field $X_H = J\nabla H$. Let $\beta \in \Omega^1(C)$ be a one-form. A map $u: C \to M$ is an \emph{inhomogeneous pseudo-holomorphic map} if 
  \begin{equation}
    \label{eq:floers-eqn}
    J \circ (du - X_H \otimes \beta) =  (du - X_H \otimes \beta) \circ j
  \end{equation}
\end{definition}


Let $M = N\times \R$ be a cylindrical symplectic manifold and $J$ a cylindrical almost complex structure. Then we may write any map $u : C \to N\times \R$ as $u = (f,a)$, where $f: C \to N$ and $a : C \to \R$. Let $\pi_\xi : TN \to \xi$ denote the projection whose kernel is the Reeb field $R$, and let $\pi_R : TN \to \langle R \rangle$ denote the complementary projection. Observe that $\pi_R(X) = \alpha(X)R$ Let $J_\xi$ denote the $\xi$ component of $J$. 

\begin{proposition}
\label{prop:hol-map-decomp}
  The map $u = (f,a): C \to N \times \R$ is pseudo-holomorphic iff 
  \begin{align}
    J_\xi \circ \pi_\xi \circ df &= \pi_\xi \circ df \circ j\\
    J \circ da &= \pi_R \circ df \circ j
    \end{align}
    If we identify $da$ with a one-form on $C$, the second equation may be expressed as
    \begin{equation}
      da = \alpha\circ df \circ j = (f^*\alpha)\circ j
    \end{equation}
    The system says that $\pi_\xi \circ df : TC \to \xi$ is complex linear, and the one-form $(f^*\alpha)\circ j$ is exact, with $a$ being an antiderivative.
  \end{proposition}

  \begin{proof}
    Clear by decomposing the tangent space to the target as in \eqref{eq:tangent-splitting}.
  \end{proof}

  \begin{proposition}
    \label{prop:floer-map-decomp}
    Let $H = h(e^\rho)$ be a Hamiltonian function on $N\times \R$ that depends only on the $\R$-coordinate $\rho$, and let $\beta \in \Omega^1(C)$. Then  $u = (f,a) : C \to N\times \R$ is a solution of \eqref{eq:floers-eqn} iff
    \begin{align}
      J_\xi \circ \pi_\xi \circ df &= \pi_\xi \circ df \circ j \\
      J\circ da &= (\pi_R \circ df - X_H\otimes \beta) \circ j
    \end{align}
    If we identify $da$ with a one-form on $C$, the second equation may be expressed as
    \begin{equation}
      da = (\alpha \circ df - \alpha(X_H)\beta)\circ j
    \end{equation}
    Note that the expression $\alpha(X_H) = h'(e^a)$ depends functionally on $a$ but not on $f$.
  \end{proposition}
  \begin{proof}
    First observe that the condition on $H$ implies that $X_H$ is proportional to $R$, thus $\pi_\xi(X_H) = 0$, and $\pi_R(X_H) = X_H = \alpha(X_H)R$. By considering the $\xi$-component, we obtain the first equation. The $R$ component of $du-X_H\otimes \beta$ is $\pi_R \circ df - X_H\otimes \beta$, while the $Z$ component is $da$.
  \end{proof}

By comparing these two propositions we see that, in both cases, for a map $u = (f,a)$ to solve the equation it is necessary that $\pi_\xi\circ df$ be complex linear, in which case $a$ can be more or less reconstructed from $f$ if it exists. This motivates the definition of a pseudo-holomorphic curve in a contact manifold.

\begin{definition}
  Let $(N,\alpha)$ be a contact manifold, and $J_\xi$ an almost complex structure on $\xi = \ker \alpha$ compatible with $d\alpha$. A map $f: C \to N$ is called \emph{pseudo-holomorphic} if 
  \begin{equation}
    J_\xi\circ \pi_\xi \circ df = \pi_\xi \circ df \circ j
  \end{equation}
\end{definition}

The next proposition expresses the familiar principle that ``holomorphic curves are symplectic.'' 
\begin{proposition}
\label{prop:positivity-contact-target}
  Let $f: C \to N$ be a pseudo-holomorphic map. Then $f^*d\alpha$ is a non-negative $2$-form (with respect to the complex orientation of $C$). Furthermore $f^*d\alpha$ can only vanish at a point where $\pi_\xi \circ df : TC \to \xi$ vanishes as a linear transformation, or equivalently $df$ maps $TC$ into the line spanned by $R$. Also, for any point $p$, $\pi_\xi \circ df_p$ has rank either zero or two.
\end{proposition}
\begin{proof}
 For any vector $v \in TC$, we have an oriented basis $\langle v , jv\rangle$ of $TC$. We compute
  \begin{equation}
    \begin{split}
      f^*d\alpha(v,jv) = d\alpha(df(v),df\circ j(v)) = d\alpha(\pi_\xi\circ df(v), \pi_\xi \circ df \circ j(v)) \\
    = d\alpha(\pi_\xi\circ df (v), J_\xi \circ \pi_\xi \circ df (v)) = \Vert \pi_\xi \circ df(v) \Vert^2_{g_\xi} \geq 0
    \end{split}
  \end{equation}
and the expression can only vanish if $\pi_\xi \circ df(v) = 0$, which is to say that $df(v)$ is proportional to $R$. For the last assertion, observe that if $\pi_\xi\circ df_p$ has rank less than two, by choosing $v$ in the kernel we have $f^*d\alpha(v,jv) = 0$.
\end{proof}

Following \cite[\S 5.3]{behwz}, we define notions of energy for holomorphic curves in symplectizations.
\begin{definition}
Let $u = (f,a) : C \to N \times \R$ be a map. The $d\alpha$-\emph{energy} of $u$ is
\begin{equation}
  E_{d\alpha}(u) = \int_C f^*d\alpha,
\end{equation}
while the $\alpha$-\emph{energy} is 
\begin{equation}
  E_\alpha(u) = \sup_{\phi \in \cC} \int_C (\phi \circ a) da \wedge f^*\alpha,
\end{equation}
where $\cC$ is the set of functions $\phi : \R \to \R$ that are non-negative, compactly supported, and of integral one. The \emph{energy} of $u$ is the sum $E(u) = E_{d\alpha}(u) + E_\alpha(u)$.
\end{definition}
 Note that the $d\alpha$-energy only depends on $f$, the $N$-component of $u$.
The key property of maps satisfying $E(u) < \infty$ is that they are asymptotic to Reeb orbits at the punctures of $C$ \cite[Proposition 5.6]{behwz}.
\begin{proposition}
  Let $C$ be a Riemann surface with punctures, and let $u = (f,a): C \to N \times \R$ be a map satisfying $E(u) < \infty$. Suppose that the Reeb flow on $N$ has Morse-Bott manifolds of Reeb orbits. Then $u$ is asymptotic to a Reeb orbit at each puncture. Namely, with respect to holomorphic cylindrical coordinates $(s,t) \in [0,\infty) \times S^1$ near a puncture, there is a Reeb orbit $\gamma$ of period $T$ such that
  \begin{equation}
    \lim_{s\to \infty}f(s,t) = \gamma(\pm Tt),
  \end{equation}
  where the sign on the right-hand side is positive if $\lim_{s \to \infty} a = \infty$ and negative if $\lim_{s \to \infty} a = -\infty$.
\end{proposition}

\subsection{Floer cohomology}
\label{sec:floer-cohomology}

The moduli spaces of inhomogeneous pseudo-holomorphic maps to a fixed symplectic target $M$ may be used to setup a TQFT-type structure of which the symplectic cohomology is a part. Fix a base field $\K$. Assume that $M$ comes with a cylindrical end with natural coordinate $\rho$. The relevant references are \cite{seidel-biased,q-intersection,ritter-tqft}.

\begin{definition}
  A Hamiltonian $H^m : M \to \R$ that is of the form $H^m = me^\rho + C$ for large $\rho$ is said to have \emph{linear of slope $m$ at infinity}. A Hamiltonian $H^Q$ that is of the form $H^Q = C(e^\rho)^2+D$ for large $\rho$ is said to be \emph{quadratic at infinity}.
\end{definition}
 
The rough idea is that we can work either with a quadratic Hamiltonian $H^Q$, or with a family of linear Hamiltonians $\{H^m\}$, and take the limit as $m$ goes to infinity to eliminate the dependence on $m$. To explain the latter version, let $m \in \R$ be a number so that all 1-periodic orbits of $X_{H^m}$ lie in a compact subset of $M$. Equivalently, we require that $m$ is not equal to the period of any Reeb orbit in the contact hypersurface $\{\rho = 0\}$. Let $J$ be a compatible almost complex structure that is cylindrical for large $\rho$.

Take a time-dependent perturbation $K :S^1\times M \to \R$ of $H^m$ such that $K(t,x) = H^m(x)$ for $x$ outside a compact subset of $M$, which is such that all the 1-periodic orbits of $X_K$ are non-degenerate. These 1-periodic orbits form a basis of the cochain complex $CF^*(H^m)$. This complex receives a $\Z$-grading by Conley-Zehnder index as soon as we pick a trivialization of the canonical bundle $\Lambda^n_\C TM$, and homotopic trivializations produce the same grading. 

The differential $d$, of degree $1$, is defined as follows. Take a time-dependent family of compatible almost complex structures $J(t,x)$ that are equal to the given cylindrical $J(x)$ outside of a compact subset, which is chosen so as to make the moduli space of Floer trajectories regular. This is the moduli space of inhomogeneous pseudo-holomorphic maps $u(s,t): \R \times S^1\to M$ satisfying the equation
\begin{equation}
\begin{cases}
  \partial_su + J(t,u)\left(\partial_tu - X_K(t,u)\right) = 0\\
  \lim_{s\to\pm\infty} u(s,t) = \gamma_{\pm}(t)
\end{cases}
\end{equation}
where $\gamma_\pm$ are generators of $CF^*(H^m)$. The signed count of solutions to this equation (modulo the $\R$-translation action that shifts the coordinate $s$ on the domain) yields the coefficient of $\gamma_-$ in $d(\gamma_+)$. The cohomology of this cochain complex is the Floer cohomology $HF^*(H^m)$.

The Floer cohomologies $HF^*(H^m)$ for various values of the slope parameter $m$ are not isomorphic, but are related by \emph{continuation maps}, which count solutions to Floer's equation where the inhomogeneous term $X_K(t,u)$ now depends on $s$ as well (breaking the $\R$-translation symmetry) and it interpolates between the corresponding terms used to define $CF^*(H^m)$ (at $s \gg 0$) and $CF^*(H^{m'})$ (at $s \ll 0$). Assuming that $m' \geq m$ and the interpolation satisfies a monotonicity condition, this leads to a chain map (of degree 0)
\begin{equation}
  c_{m,m'}: CF^*(H^m) \to CF^*(H^{m'})
\end{equation}
The continuation maps form a directed system, and by passing to the direct limit (category-theoretical colimit), we get a definition of the symplectic cohomology of $M$.
\begin{equation}
  SH^*(M) = \lim_{m\to \infty} HF^*(H^m)
\end{equation}
By further continuation map arguments, one can show that $SH^*(M)$ (and even $HF^*(H^m)$) is independent of the choices of Hamiltonians and almost complex structures up to canonical isomorphism \cite[(3e)]{seidel-biased}.

We will have use for a few other parts of the structure, namely the Batalin-Vilkovisky (BV) operator and the product. The BV operator $\delta$ also counts cylinders, but where the perturbation data are allowed to vary in a one parameter family, parametrized by $r \in S^1$. Since this operation involves a family of domains, it is not part of the TQFT studied by Ritter \cite{ritter-tqft}, but rather part of larger structure known as a Topological Conformal Field Theory (TCFT). The following discussion of the BV operator is based on Seidel-Solomon \cite[\S 3]{q-intersection}.

 To define the BV operator $\delta: CF^*(H^m) \to CF^{*-1}(H^m)$, we use a perturbation $K_\delta(r,s,t,x)$ and family of almost complex structures $J_\delta(r,s,t,x)$ that also depend on the $s$-coordinate of the domain and an auxiliary parameter $r \in S^1$. On the ends of the domain cylinder these are required to be compatible with the data used to define the differential as follows: for $s \ll 0$, they simply agree, namely, $K_\delta(r,s,t,x) = K(t,x)$ and $J_\delta(r,s,t,x) = J(t,x)$, while for $s \gg 0$, they agree after a shift depending on $r$, namely, $K_\delta(r,s,t,x) = K(t+r,x)$ and $J_\delta(r,s,t,x) = J(t+r,x)$. The asymptotic condition at the $s \gg 0$ end then becomes $\lim_{s\to\infty} u(s,t) = \gamma_+(t+r)$. Counting solutions that are isolated even as the parameter $r$ is allowed to vary yields the degree $-1$ map $\delta$. One finds that $\delta$ is a chain map, and that compatibility of the $\delta$ for various $m$ yields a BV operator $\Delta: SH^*(M) \to SH^{*-1}(M)$. A useful property is that $\Delta$ vanishes on the image of the canonical map $H^*(M)\to SH^*(M)$.

The product is the TQFT operation associated to the pair of pants. It defines a map
\begin{equation}
  HF^*(H^m) \otimes HF^*(H^{m'})\to HF^*(H^{m''})
\end{equation}
as long as $m'' \geq m+m'$. Passing to the limit as $m \to \infty$ and $m'\to \infty$, this induces a product on $SH^*(M)$.

In the setup where a quadratic Hamiltonian is used, the periodic orbits will not in general be contained in any compact set, meaning that in order to achieve non-degeneracy and transversality, the perturbation of the Hamiltonian and complex structure cannot necessarily be compactly supported. This makes the compactness for pseudo-holomorphic curves more subtle. Ritter \cite{ritter-tqft} provides two approaches for overcoming this difficulty and defining the TQFT structure using quadratic Hamiltonians.

\subsection{Example of the complex torus}
\label{sec:complex-torus}

In the rest of the paper we are interested in complex dimension two, but in this section let the dimension be general $n$. Let $N \cong \Z^n$ be a lattice, $M = \Hom(N,\Z)$. Let $T = N \otimes_\Z \C^{\times} \cong (\C^{\times})^n$ be the complex torus. Let $z_i$ be a set of coordinates on $T$ in bijection with a basis of $M$. For the purpose of grading Floer cohomology, we use the complex volume form $\Omega = \prod_{i=1}^n \frac{dz_i}{z_i}$.Then
\begin{equation}
  \label{eq:sh-torus}
  SH^p(T) \cong \Z[N] \otimes \Lambda^p M \cong \Z[x_1^{\pm 1},\dots,x_n^{\pm 1}]\otimes \Lambda^p [x_1^\vee, \dots, x_n^\vee]
\end{equation}
where $x_i$ and $x_i^\vee$ represent dual bases of $N$ and $M$ respectively.

This computation is a special case of the symplectic homology of cotangent bundles \cite{abb-schwarz, salamon-weber, viterbo-2}. Since we are using the cohomological convention, and the convention that the canonical map $H^*(\bullet) \to SH^*(\bullet)$ has degree zero, the isomorphism relating the symplectic cohomology to the loop space of a spin manifold $Q$ is
\begin{equation}
\label{eq:sh-hloop}
  SH^{n-*}(T^*Q) \cong H_*(\cL Q)
\end{equation}
Note that \eqref{eq:sh-hloop} ``implicitly fixes all the
conventions used in the present paper (homology versus cohomology, the grading,
and the inclusion of non-contractible loops)'' \cite{seidel-sh-as-hh}.

The subspace $x_1^{a_1}x_2^{a_2}\cdots x_n^{a_n} \Lambda^* [x_1^\vee,\dots,x_n^\vee]$ is the cohomology of the component of $\cL T^n$ consisting of loops representing a certain class $(a_1,\dots,a_n) \in N = H_1(T^n,\Z)$, since this component is homotopy equivalent to $T^n$. The isomorphism \eqref{eq:sh-hloop} maps this subspace in the manner of Poincar\'{e} duality to the homology of the loop space.

For a manifold $Q$, string topology shows \cite{string-topology} that the space $\cL Q$ has a product given by composing families of loops when they are incident, and a BV operator given by spinning the parametrization of the loops. The isomorphism \eqref{eq:sh-hloop} identifies these structures as well (see \cite{abb-schwarz-product} for the product and \cite{abouzaid-sh-viterbo} for the full BV structure).

We will briefly explain how this computation can be done from a symplectic viewpoint, previewing the method used for general log Calabi--Yau surfaces. The complex torus can in some way serve as a local model for the general computation. We will describe in section \ref{sec:construction} a general method for finding ``nice'' Liouville structures on log Calabi--Yau surfaces. The main feature of such a ``nice'' structure is that it contains a contact-type hypersurface $\Sigma \subset (\C^\times)^n$ fibered by Lagrangian tori, such that the Reeb flow acts preserving the tori, and rotating each by some amount. Let $\Log : (\C^\times)^n \to \R^n$ be the standard torus fibration given in each coordinate by the logarithm of the absolute value. Let $S \subset \R^n$ be some large sphere centered at the origin, and let $\Sigma = \Log^{-1}(S)$ be the union of the torus fibers sitting over $S$. We arrange that $\Sigma$ is contact-type, with the evident torus fibration $\pi: \Sigma \to S$, that the Reeb vector field is tangent to the fibers of $\pi$, and that the Reeb flow acts on each fiber $\pi^{-1}(s)$ as a linear translation on the torus, say translation by the vector $v(s)$. What is important is that the direction of translation depends on the point in the base. Given the base point $s$, represent the torus $\pi^{-1}(s)$ as $\R^n/\Z^n$; if the direction of $v(s)$ is rational in this representation, the Reeb flow on $\pi^{-1}(s)$ is periodic (with some period $T(s)$). Now for each such $s$, and each multiplicity $r \in \N^+$, we have a torus $T_{s,r} \subset \cL(\C^\times)^n$ of periodic orbits lying on $\pi^{-1}(s)$ that wrap a primitive orbit $r$ times. In fact, we can arrange that the pairs $(s,r)$ indexing the tori correspond bijectively to the nonzero elements in $H_1((\C^\times)^n,\Z)$ (under the Hurewicz map $\cL M \to H_1(M,\Z)$).

If we use either a Hamiltonian $H$ with linear or quadratic growth on the cylindrical end, these periodic Reeb orbits correspond to tori of Hamiltonian orbits of period $1$ (with the exception that the Reeb period must be less than the asymptotic slope in the linear case). These orbits are evidently degenerate since they come in continuous families, but a generic time-dependent perturbation of the Hamiltonian near each torus $T_{s,r}$ (possibly different for each $s,r$) breaks this manifold of orbits into several non-degenerate orbits. The differential counts cylinders, hence can only connect orbits that are homologous, and so orbits corresponding to different tori $T_{s,r}$ are not connected by differentials. 

Using the isomorphism with loop-space homology, we see that, within the set of orbits corresponding to a single torus $T_{s,r}$, the Floer cohomology complex computes the cohomology of the component of the loop space containing $T_{s,r}$. As this component is homotopy equivalent to $T_{s,r}$, we identify this cohomology with $H^*(T_{s,r})$. The contractible orbits of $H$ contribute the ordinary cohomology of $(\C^\times)^n$. Another expression for the cohomology of $(\C^\times)^n$ is then
\begin{equation}
  SH^*((\C^\times)^n) \cong H^*((\C^\times)^n) \oplus \bigoplus_{s,r} H^*(T_{s,r})
\end{equation}
This shows us that the degree $0$ generators of symplectic cohomology correspond to the fundamental classes of the iterates of the periodic tori. 

We can use this computation to draw some conclusions about the structure of symplectic cohomology near a periodic torus. Since all the classes in $H^*(T_{s,r})$ are linearly independent in Floer cohomology, we see that the perturbation of Hamiltonian near each periodic torus must create at least $\binom{n}{k}$ orbits of Conley-Zehnder index $k$.

We record now some facts about the BV operator and the product that will be used in the computation of the differential on symplectic cohomology in Section \ref{sec:differential}.

The BV operator on symplectic cohomology is identified with the rotation of loops $H_*(\cL T)\to H_{*+1}(\cL T)$. Thus the action of $\Delta$ on $H^*(T_{s,r})$ is Poincar\'{e} dual to the operation of taking a cycle $T_{s,r}$ and rotating the parametrization of the loops. Under the isomorphism $H^*(T_{s,r}) \cong \Lambda^*\K^n$, this corresponds to contraction with the class of the orbit.

The product structure is also straightforward. There are maps
\begin{equation}
  H^*(T_{s,r})\otimes H^*(T_{s',r'}) \to H^*(T_{s'',r''})
\end{equation}
that can be characterized as follows. If we let $a(s,r) \in \pi_{0}(\cL(\C^{\times})^{2})\cong H_{1}((\C^{\times})^{2},\Z)$ denote the component of the free loop space containing $T_{s,r}$, the map above is nontrivial precisely when $a(s,r)+a(s',r') = a(s'',r'')$, and in this case it is given by a sort of cup product; this follows from the identification of this product with the Chas-Sullivan product on the free loop space homology. In particular, if we consider $s = s' = s''$, and $r'' = r + r'$, then the map is nontrivial, and the degree zero component of the target is in the image. In the symplectic cohomology, this product is represented by a pair of pants that is a small perturbation of an $r''$-to-1 branched covered cylinder (the ramification at one end, corresponding to the inputs, has two components mapping with multiplicities $r$ and $r'$, while at the other end we have one component mapping with multiplicity $r''$.) These cylinders will appear again later as low-energy contributions to the product the general case.

\section{Log Calabi--Yau surfaces}
\label{sec:log-cy-surfaces}

\subsection{Basics}
\label{sec:basics}
Our main objects of study are log Calabi--Yau pairs $(Y,D)$ with
positive, maximal boundary. We define these notions presently. For
the reader who finds this terminology perverse we note that this
combination of conditions is equivalent to saying that $Y$ is a surface and
$D$ is a nodal reduced anticanonical divisor, such that $D$ supports an ample divisor class. Readers who are familiar with these notions may skip this section but should note Lemma \ref{lem:divisor-topology}, which is relied on throughout the paper.

Let $Y$ denote a smooth projective surface over the complex
numbers. The canonical bundle is denoted $\Omega^2_Y$, and the
canonical divisor class is denoted $K_Y$. Let $D$ be an effective
divisor on $Y$.
\begin{definition}
  The pair $(Y,D)$ is a \emph{log Calabi-Yau pair} if $K_Y+D$ is a
  principal divisor, that is, $D$ lies in the anticanonical divisor
  class. Equivalently, there is an isomorphism $\Omega^2_Y(D) \cong
  \cO_Y$.
\end{definition}

In this paper we will usually assume that $D$ is a normal crossings divisor.
\begin{definition}
  An effective divisor $D$ on a surface $Y$ is a \emph{normal crossings divisor} if $D$ is a reduced Cartier divisor, and, writing $D = \sum_i D_i$ with irreducible components $D_i$, each $D_i$ is a smooth or nodal curve intersecting the other components transversely (so that $(D-D_i)|_{D_i}$ is a reduced divisor on $D_i$).
\end{definition}

There is a restriction on the topology of the pair $(Y,D)$ when $D$ is normal crossings. The proof is an exercise in adjunction.
\begin{lemma}
\label{lem:divisor-topology}
  Let $D$ be a connected normal crossings divisor in a smooth projective surface $Y$, such that the pair $(Y,D)$ is log Calabi--Yau. Then either
  \begin{enumerate}
  \item $D$ is a smooth genus one curve,
  \item \label{item:nodal-elliptic} $D$ is a irreducible nodal curve of arithmetic genus one, or
  \item \label{item:cycle-rationals} $D$ is a sum of smooth rational curves, whose intersection graph is a cycle.
  \end{enumerate}
\end{lemma}
\begin{proof}
  The arithmetic genus of the possibly nodal curve $D$ is given by the adjunction formula and the assumption $K_Y+D \sim 0$,
  \begin{equation}
    p_a(D) = \frac{(K_Y + D)\cdot D}{2} + 1 = 1.
  \end{equation}
  If $D$ is irreducible, it falls under one of the first two cases. 

  Suppose that $D = \sum_{i\in I} D_i$ is reducible with components
  $D_i$. Let $\Gamma(D)$ be the intersection graph of $D$. The vertex
  set $I$ is the index set for the components, each vertex is labeled
  by the arithmetic genus of the component, and we draw as many edges
  between two vertices as there are intersections between the
  corresponding components. Our assumption is that $\Gamma(D)$ is
  connected and has at least two vertices.

  We claim that each vertex of $\Gamma(D)$ has valence at least
  two. Suppose $\Gamma(D)$ has a vertex of valence one, say $k$.
 Then $(D-D_k)\cdot D_k = 1$. Using $-K_Y \sim D$, 
  \begin{equation}
    -K_Y\cdot D_k = (D-D_k) \cdot D_k + D_k \cdot D_k = 1 + D_k \cdot D_k.
  \end{equation}
  Reducing modulo $2$,
  \begin{equation}
    K_Y\cdot D_k \not\equiv D_k \cdot D_k \pmod{2},
  \end{equation}
  which is impossible by the adjunction formula. Thus every vertex of $\Gamma(D)$ has valence at least two.

  Since every vertex as valence at least two $\Gamma(D)$ cannot be a
  tree. Therefore it contains a cycle. By genus considerations there
  can only be one cycle, and every component of $D$ is rational.
  Since $\Gamma(D)$ is connected and contains only one homological
  cycle, $\chi(\Gamma(D)) = 0$. Thus the numbers of vertices and edges
  are equal. Since each vertex has valence at least two, and
  $\sum_{i\in I} \frac{1}{2}\text{valence}(i) = \#\text{edges}$, each
  vertex has valence exactly two. Thus $\Gamma(D)$ is a cycle.
\end{proof}

To see that the assumption that $D$ is connected is necessary, consider $Y= E \times \PP^1$ where $E$ is an elliptic curve, with $D = E \times \{0\} \cup E \times \{\infty\}$.

\begin{definition}
  A log Calabi--Yau pair $(Y,D)$ satisfying the hypotheses of Lemma \ref{lem:divisor-topology} is said to have \emph{maximal boundary}, or is called a \emph{Looijenga pair}, if it falls under cases \ref{item:nodal-elliptic} or \ref{item:cycle-rationals} of the conclusion.
\end{definition}

For the grading on symplectic cohomology of $U = Y\setminus D$, it is important to actually specify the trivialization of $\Omega^2_U$. There is a preferred trivialization, given by a meromorphic two-form on $Y$ that is non-vanishing and holomorphic on $U$ with simple poles along $D$. We have the following proposition, that in particular shows that the homotopy class of the trivialization does not depend on the choice of compactification.

\begin{lemma}
  \label{lem:volume-form} Let $(Y,D)$ be a log Calabi--Yau pair with maximal boundary. The complement $U = Y \setminus D$ carries a non-vanishing holomorphic two-form $\Omega$, characterized up to a scalar multiple by the property that $\Omega$ has simple poles along $D$. If $(Y,D)$ and $(Y',D')$ are two log Calabi--Yau compactifications of a given $U$, then the corresponding two-forms $\Omega$ and $\Omega'$ differ by a scalar multiple.
\end{lemma}

\begin{proof}
  The second assertion implies the first, so it suffices to consider two pairs $(Y,D)$ and $(Y',D')$ such that $Y\setminus D = U = Y' \setminus D'$. Let $\Omega$ (resp. $\Omega'$) be any meromorphic form on $Y$ (resp. $Y'$) that is non-vanishing and holomorphic on $U$ and has simple poles along $D$ (resp. $D'$). There is a birational map $p: Y \dashrightarrow Y'$ that is the identity on $U$. The pull-back $p^*\Omega'$ is a meromorphic form on $Y$, that is non-vanishing and holomorphic on $U$. Thus the ratio $f = p^*\Omega'/\Omega$ is a rational function on $Y$, whose divisor of zeros and poles is contained in $D$.

When $Y$ and $Y'$ are the same, the condition that both $\Omega$ and $\Omega'$ have the same divisor of poles implies that $f$ has no zeros or poles, and hence is constant.

If $Y$ and $Y'$ are distinct, then since any birational map factors into blow-ups, it suffices to prove the lemma when $p: Y \dashrightarrow Y'$ is a blow up of $Y'$. The exceptional locus is necessarily contained in $D$, and by Lemma \ref{lem:divisor-topology}, the exceptional curves of $p$ must map to nodes of $D'$.

Now we make use of another property of $\Omega$, namely that it has nonzero residue at any node of $D$. This residue is the integral of $\Omega$ on a small torus linking the node. Indeed, picking local analytic coordinates $(z_1,z_2)$ such that the node takes the form $\{z_1z_2 = 0\}$, the condition that $\Omega$ has simple poles along $D$ is equivalent to the condition that its lowest order term is proportional to $dz_1\wedge dz_2/z_1z_2$. 

 Since $\Omega$ is a closed form on $U$ (being a holomorphic top form), we must obtain the same residue by integrating over any homologous torus, and since the boundary divisor is a cycle of rational curves, the tori at each of the nodes are homologous to each other. Furthermore, since $p: Y \dashrightarrow Y'$ is a blowup at some nodes of $D'$, the homology classes of the linking tori in $Y$ and $Y'$ correspond under $p|_U$.

The form $\Omega'$ on $Y'$ also has a nonzero residue at any linking torus of $D'$. Since this residue is given by an integral inside $U$, the same must be true of $p^*\Omega'$ on $Y$. 

Suppose now that the ratio $f$ is not constant. Then $f$ must have zeros somewhere in $Y$. Since it is non-vanishing in $U$, we conclude that it vanishes along some component $D_1$ of $D$. But then writing $p^*\Omega' = f\Omega$, we see that the zero of $f$ cancels the pole of $\Omega$ on $D_1$, implying that $p^*\Omega'$ would have zero residue at any node of $D$ involving $D_1$, which is a contradiction.
\end{proof}

\begin{definition}
  A reduced divisor $D = \bigcup_{i=1}^n D_i$  is said to \emph{support an ample divisor}, if some linear combination of the irreducible components $A = \sum_{i=1}^n a_iD_i$ is ample. If there is such a combination with all coefficients $a_i$ strictly positive, then $D$ is called \emph{positive}.
\end{definition}

\begin{lemma}
\label{lem:nakai}
  Let $D$ be a connected divisor in a projective surface $Y$. If $D$ supports an ample divisor, then $D$ supports an ample divisor $A = \sum_{i=1}^n a_iD_i$ where all coefficients $a_i$ are strictly positive.
\end{lemma}

\begin{proof}
  This proof is drawn from \cite[\S 2.4]{fujita-noncomplete} via \cite{zaidenberg-exotic}. Consider the set 
  \begin{equation}
    \mathcal{S} = \left\{A = \sum_{i \in I} a_iD_i \mid I \subset \{1,\dots,n\}, (\forall i \in I)(a_i > 0 \text{ and } A\cdot D_i > 0)\right\}
  \end{equation}
  In words, $\mathcal{S}$ is the set of effective divisors, supported
  on $D$, that have positive intersection with any irreducible
  component of their support. The conclusion follows once we know that $\mathcal{S}$ contains a divisor $B$ whose support is all of $D$. By definition, $B$ contains every $D_i$ with a strictly positive coefficient $a_i$. The Nakai--Moishezon criterion implies that $B$ is ample: First, by definition $B\cdot D_i > 0$ for every irreducible component of $D$. Then we also see that $B^2 = \sum_{i=1}^n a_i(B\cdot D_i) > 0$. Lastly, if we consider an irreducible curve $C$ that is not a component of $D$, we have $C\cdot D_i \geq 0$ for all $i$. The fact that $D$ supports an ample divisor implies that $C$ is not disjoint from $D$, so $C \cdot D_i > 0$ for some $i$. Thus $B \cdot C > 0$.

It remains to show that $\mathcal{S}$ contains an element whose support is all of $D$. First we show that $\mathcal{S}$ is not empty. Start with some ample divisor $A = \sum_i a_i D_i$ supported on $D$, and write $A = P - N$, where $P$ and $N$ are effective and have no components in common. Let $D_i$ be contained in the support of $P$. As $A\cdot D_i > 0$, we have $P\cdot D_i > N \cdot D_i$. Since $D_i$ is not contained in the support of $N$, we have $N \cdot D_i \geq 0$, and so $P \cdot D_i > 0$. Thus $P \in \mathcal{S}$.

Now we must add the other components of $D$ while staying in $\mathcal{S}$. Let $D_j$ be an irreducible component of $D$ that is not contained in $P$, but which does intersect $P$ non-trivially. Then $mP + D_j \in \mathcal{S}$ for $m \gg 0$. Indeed, $(mP+D_j)\cdot D_j > 0$ as long as $m > -D_j^2/P\cdot D_j$ (the denominator is greater than zero by assumption that $D_j$ intersects $P$ and is not contained in it). 

Because we assumed that $D$ is connected, we may iterate the previous step to add each time an irreducible component of $D$ that is not contained in the support but which intersects it non-trivially. Thus $\mathcal{S}$ contains an element whose support is all of $D$.
\end{proof}

\subsection{Examples}
\label{sec:examples}

  Observe that a pair $(Y,D)$ is a log Calabi--Yau pair with maximal boundary and $D$ ample if and only if $Y$ is a del Pezzo surface and $D$ is a nodal reduced anticanonical divisor. However, the weaker assumption that $D$ merely supports an ample divisor includes infinitely many more types of surfaces. For example, we can take $Y = \PP^2$ blown up any number of times, as long as these blowups all lie on a single conic. We let $D= Q\cup L$ be the union of the proper transform of that conic $Q$ with some line $L$. This $D$ is anticanonical and nodal, and $mL + Q$ is ample for $m \gg 0$.

Here we list some examples of log Calabi--Yau pairs. 

\subsubsection{The projective plane}
Let $Y = \PP^2$. As $\Omega^n_{\PP^2} \cong \cO_{\PP^2}(-3)$, any cubic curve will serve for $D$. There are essentially four possibilities. In going from each case to the next, we smooth a node of $D$. This changes the complement $U$ by adding a $2$--handle.
\begin{enumerate}
\item $D$ is the union of three lines in general position. Then $U = Y \setminus D \cong (\C^\times)^2$, and the Betti numbers are $b_1 = 2$, $b_2 = 1$.
\item $D$ is the union of a conic and a line in general position. The Betti numbers of $U$ are $b_1 = 1$, $b_2 = 1$. Floer cohomology for Lagrangian submanifolds in $U$ was studied in \cite{binodal}.
\item $D$ is a nodal cubic curve. The Betti numbers of $U$ are $b_1 = 0$, $b_2 = 1$.
\item $D$ is a smooth cubic curve. The pair $(Y,D)$ does not have maximal boundary. The Betti numbers of $U$ are $b_1 = 0$, $b_2 = 2$.
\end{enumerate}


\subsubsection{A cubic surface}
\label{sec:cubic}

Let $Y$ be a smooth cubic surface in $\PP^3$. As is well-known, $Y$ contains 27 lines, and it is possible to choose three of them intersecting in a $3$--cycle so that their sum is an anticanonical divisor. To see this, realize $Y$ as the projective plane blown up in six general points $p_1,\dots,p_6$, giving six exceptional curves $E_1,\dots,E_6$. Let $L_{ij}$ denote the proper transform of the line through $p_i$ and $p_j$ (there are 15 of these). Let $C_k$ denote the proper transform of the conic through five of the points, all except $p_k$ (there are six of these). The curves $E_i, L_{ij}, C_k$ are the 27 $(-1)$--curves that are mapped to lines by the anticanonical embedding $Y \to \PP^3$.

For two indices $a$ and $b$, consider the configuration $L_{ab},E_b,C_a$. This means the line through $p_a,p_b$, the exceptional curve over $p_b$, and the conic that does not contain $p_a$. Clearly $L_{ab}$ and $E_b$ intersect over $p_b$, while $C_a$ and $E_b$ intersect since $C_a$ passes over $p_b$. Also, $L_{ab}$ and $C_a$ intersect since their projections to $\PP^2$ intersect in two points: one point is $p_b$, and the other point is none of the $p_i$ (which are assumed to be in general position), and this latter intersection point persists in the blow-up.

The divisor $D = L_{ab}+E_b+C_a$ is anticanonical and very ample. 


\subsubsection{The degree 5 del Pezzo surface}
\label{sec:M05}

Let $Y$ be the (unique) degree 5 del Pezzo surface, realized as the blow-up of $\PP^2$ at 4 general points $p_1,p_2,p_3,p_4$. There are 10 $(-1)$--curves in $Y$, namely the 4 exceptional curves $E_i$ coming from the blow-ups, and $\binom{4}{2} = 6$ proper transforms of the lines passing through two of the points $L_{ij}$. To get an anticanonical divisor, choose a partition of the set $\{1,2,3,4\}$ into $\{i,j\}$ and $\{k,\ell\}$ (there are 12 such choices). Then take $D = L_{ki} + E_i + L_{ij} + E_j + L_{j\ell}$. Thus $D$ is a $5$--cycle of $(-1)$--curves which is anticanonical and ample. 


\subsubsection{``Punctured'' $A_n$ Milnor fibers}
\label{sec:a-n-milnor-fiber}

Let $V = \{x^2+y^2+z^{n+1} = 1\} \subset \C^3$ be the Milnor fiber of the two-dimensional $A_n$ singularity. It is possible to compactify $V$ by adding two rational curves \cite[\S 7.1]{evans-mcg}. Start with $\PP^2$, with homogeneous coordinates $[x:y:z]$. Blow up the $n+1$ points $[\xi_k : 0 : 1]$ along the $x$--axis, where $\xi_k = \exp(2\pi i k/(n+1))$, and call the result $Y$. Let $P_t$ be pencil on $Y$ that is the preimage of the pencil of lines through $[0:1:0]$. In an affine chart these are depicted as the lines parallel to the $y$--axis. The line at infinity $\{z=0\}$ is a fiber of this pencil, $P_\infty$. The pencil $P_t$ on $Y$ has $n+1$ singular fibers, where the line passes through a blown-up point. Let $C$ denote the proper transform of the $x$--axis; it is a section of the pencil and passes through all of the exceptional curves. The complement $Y \setminus (P_\infty \cup C)$ is isomorphic as a complex manifold to the Milnor fiber $V$ \cite[Lemma 7.1]{evans-mcg}.

Although $V$ satisfies $c_1(V) = 0$, it is not log Calabi--Yau in the sense of this paper. The anticanonical class of $Y$ is
\begin{equation}
  -K_Y \sim 3H - \sum_{i=1}^n E_i \sim 2P_\infty + C 
\end{equation}
where $H$ denotes the pull-back of the hyperplane class on $\PP^2$. The issue is the coefficient of $2$ in front of $P_\infty$, which means that a holomorphic volume form on $V$ will have a pole of order $2$ along $P_\infty$. 

We can get something that falls into our setting by removing another smooth fiber of the pencil, say $P_0$. Write $U = Y \setminus (P_0 \cup P_\infty \cup C)$. We call $U$ the punctured $A_n$ Milnor fiber, since we puncture the line parametrizing the pencil on $V$. As $-K_Y \sim P_0 + P_\infty + C$, this is a log Calabi--Yau surface.

The self-intersections are $P_0^2 = 1, P_\infty^2 = 1, C^2 = 1-n$. The compact surface $Y$ is not Fano unless $n \leq 2$, since $-K_Y \cdot C = 3-n$. Nevertheless, the divisor $P_0\cup P_\infty \cup C$ supports an ample divisor. Indeed, $aP_0 + bP_\infty + C$ is ample as long as $a > 0$, $b > 0$, and $a + b > n - 1$. 


\begin{remark}
   The point of view that $U$ is the form of the Milnor fiber that is ``truly Calabi--Yau'' comes from the Strominger--Yau--Zaslow (SYZ) picture and is discussed in \cite[\S 9.2]{aak-lagrangian-fibrations}. This lines up well with \cite{ghk}, which also used the SYZ picture (in the form of the Gross--Siebert program) as its starting point.
\end{remark}


\section{Construction of the Liouville domain}
\label{sec:construction}

In this section we will construct a Liouville domain associated to a log Calabi--Yau pair $(Y,D)$ where $D$ is positive. In order to obtain symplectic forms, we use Lemma \ref{lem:nakai}, and choose an ample divisor $A = \sum a_iD_i$ supported on $D$, such that each coefficient $a_i$ is strictly positive. The Liouville domain we construct is the symplectic model for the complement $U = Y\setminus D$. Since we are ultimately interested in symplectic cohomology, and symplectic cohomology is an invariant of Liouville deformation, we are free to take a particular representative of the Liouville deformation class that has convenient properties.\footnote{In particular, the symplectic cohomology does not depend on the choice of ample divisor $A = \sum a_{i}D_{i}$, although the coefficients $a_{i}$ will appear in the local expressions for the Liouville class.} This idea was used by Seidel \cite{seidel-biased} and Mark McLean \cite{mclean-growth} to understand the growth rate of symplectic cohomology. In fact, the first four steps of the construction follow \cite[\S 4]{seidel-biased} very closely, though at some points we extract more precise information for our particular cases. The fifth step is new, and highlights an interesting property of contact hypersurfaces in $U$. 

\subsection{Basic Liouville structure on an affine variety}
\label{sec:liouville-basics}
Let $Y$ be a smooth projective variety with a positive divisor $D$, and let $A = \sum a_iD_i$ be a strictly positive combination of components that is ample. Then there is a holomorphic line bundle $\cL \to Y$ and a section $s \in H^0(Y,\cL)$ such that $A = s^{-1}(0)$. The line bundle $\cL$ admits a Hermitian metric $\Vert\cdot\Vert$ such that, if $F$ is the curvature of the unique connection compatible with the metric and the holomorphic structure, then $\omega = 2i F$ is a K\"{a}hler form. On the complement $U = Y\setminus D$, the function $\phi = -\log\Vert s\Vert$ is a K\"{a}hler potential since
\begin{equation}
\label{eq:kahler-potential}
   2iF|_U = dd^c(-\log\Vert s \Vert) 
\end{equation}
Thus the symplectic form $\omega$ on $U$ is exact and $\lambda = d^c\phi$ is a primitive. Here $d^c = -i(\partial - \bar{\partial})$, and for a function $f$ this means $d^cf = -df\circ J$.

The function $\phi$ is clearly proper and bounded below. A simple
lemma \cite[Lemma 4.3]{seidel-biased} shows that the set of critical
points of $\phi$ is compact when $D$ has normal crossings, so by
choosing a sufficiently large regular value $C$, we find that
$\overline{U} = \phi^{-1}(-\infty,C]$ is a compact subset containing
all of the topology of $U$.

The Liouville vector field is defined by the condition $\iota_Z \omega = \lambda$. Thus
\begin{equation}
  0 \leq g(Z,Z) = \omega(Z,JZ) = \lambda(JZ) = -d\phi\circ J (JZ) = d\phi(Z)
\end{equation}
and equality can only hold when $Z = 0$, whence $\lambda = 0$, whence $d\phi = 0$. Thus $Z$ points strictly outwards along $\partial{\overline{U}} = \phi^{-1}(C)$.

Thus $\overline{U}$ equipped with the structures $\omega,\lambda,Z$ is a Liouville domain.

\subsection{Refinements of the basic construction}
\label{sec:refinements}

We assume that, in addition to being positive, $D$ is anticanonical, so that $(Y,D)$ is a log Calabi--Yau pair. We also assume that the pair has maximal boundary, which means that $D$ is normal crossings with at least one node. By Lemma \ref{lem:divisor-topology}, $D$ is either isomorphic to the irreducible nodal genus one curve or it is a cycle of rational curves. While the case of an irreducible nodal curve appears exceptional, it can be subsumed into the other case by blowing up the node; this replaces the irreducible nodal genus one curve with a cycle of two smooth rational curves. Because blowing up the a point on $D$ does not change $U= Y \setminus D$, we may compute the symplectic cohomology after the blowup. After this modification, we can assume that $D$ is a cycle of smooth rational curves.


To begin with we start with a basic Liouville structure $\omega,\lambda,Z$ on $U$ as in the previous section.

\subsubsection{Step 1: constructing local torus actions along the divisor}
\label{sec:step1}

This step is basically the same as in \cite{seidel-biased}, but we get a little more structure along the smooth parts of the divisors.

Let us write $D = \bigcup_{i=1}^r D_i$, where each irreducible component $D_i$ is a smooth rational curve, and the components are ordered cyclically according to some chosen orientation of the intersection graph. So $D_i \cdot D_{i+1} = 1$. The first thing to do is to make consecutive divisors symplectically orthogonal. According to \cite[\S 4, Step 1]{seidel-biased} we may choose the metric on $\cL$ so that in local coordinates $(z_1,z_2)$ near $D_i\cap D_{i+1}$, the divisor is $D=\{z_1z_2 = 0\}$ and the K\"{a}hler form is standard. This neighborhood therefore admits a Hamiltonian $T^2$ action that rotates the complex coordinates $(z_1,z_2)$, with moment map $m(z_1,z_2) = \frac{1}{2}(|z_1^2|,|z_2^2|)$.

\begin{remark} 
  Our strategy is to progressively extend these $T^2$
  actions to larger subsets of $Y$. This will involve constructing a group action on some subset, and saying that it agrees with ones previously constructed on the overlap. To say that two group actions
  ``agree'' really means to say that there is an isomorphism of the acting
  groups that intertwines the actions. Thus when we say ``such and
  such $T^2$ actions agree,'' we should really add ``up to an element
  of $\Aut(T^2) \cong \GL(2,\Z)$''. Alternatively, since all of the
  actions we consider are faithful, we may simply speak of agreement
  of subgroups of the diffeomorphism group.
\end{remark}

The next thing to do is to construct a Hamiltonian $S^1$ actions in a neighborhood of each $D_i$. This is also present in \cite[\S 4, Step 1]{seidel-biased}, but we shall provide full details since we need to extend the argument. Here is the precise claim:

\begin{claim}
  For each $i$, there is a Hamiltonian $S^1$ action in a neighborhood of $D_i$, fixing $D_i$ and rotating its normal bundle, that is furthermore compatible with the previously constructed $T^2$ actions at the nodes in the sense that, in a neighborhood of the intersection $D_i \cap D_{i\pm 1}$, the $S^1$ rotating the normal bundle of $D_i$ agrees with the one-parameter subgroup of $T^2$ that rotates $D_{i\pm 1}$ and fixes $D_i$.
\end{claim}

\begin{proof}
  The tool to achieve this is the symplectic tubular neighborhood theorem, which says that the only local invariants of a symplectic embedding are the symplectic structure on the submanifold itself and the normal bundle as a symplectic vector bundle. Since rank two symplectic vector bundles over surfaces are determined by their degree, it will suffice to consider the model space $X_i$ that is a degree $D_i^2$ complex line bundle over $D_i$. The space $X_i$ admits K\"{a}hler structures that are invariant under the $S^1$ action that rotates the fibers of the bundle projection $p_i : X_i \to D_i$. We want such a K\"{a}hler structure with the property that, if $q \in D_i$ is a nodal point (a point where $D_i$ intersects another component of $D$), there is a neighborhood $V$ of $q$ such that $X_i|_V$ is trivial as a bundle of K\"{a}hler manifolds (so it is metrically the product of $V$ and $p_i^{-1}(q)$). This may be done by choosing any K\"{a}hler metric $g$ on $X_{i}$ that has the desired behavior near the nodal points, and then averaging it with respect to the holomorphic $S^1$ action to obtain $\bar{g}$; since the desired behavior near $q$ forces $g$ to $S^1$-invariant there, $\bar{g} = g$ near the nodal points.

Now we must compare neighborhoods of $D_i$ in $X_i$ and in the actual log Calabi-Yau $Y$. By scaling $\bar{g}$, we can ensure that $D_i$ has the same symplectic area in $X_i$ as it does in $Y$. In order to apply the symplectic tubular neighborhood theorem, we must construct an isomorphism of symplectic normal bundles $\psi : N(D_i/X_i) \to N(D_i/Y)$ covering a symplectomorphism of $D_i$. We take the map on $D_i$ to be the identity near the nodal points, and we also prescribe that near the nodal points the isomorphism of symplectic normal bundles matches the trivializations coming from the local product structures on $X_i$ and $Y$, respectively. Now we consider the composition $\exp_Y \circ \psi \circ \exp_{X_i}^{-1}$, where $\exp_Y$ and $\exp_{X_i}$ are the exponential maps of the chosen K\"{a}hler metrics on $Y$ and $X_i$. By the differentiable tubular neighborhood theorem, this map is a diffeomorphism between some neighborhoods of $D_i$ in the two manifolds, and because of the care we have taken near the nodal points, it matches the local product structures in neighborhoods of the nodal points, and so the $S^1$ actions match as well. Finally, we use the Moser argument to correct this map to be a symplectomorphism everywhere, see \cite[Lemma 3.14]{mcduff-salamon-intro}; an analysis of the Moser argument in this situation shows that it does not change the map in neighborhoods of the nodal points, where the symplectic structures already match. Carrying over the $S^1$ action on $X_i$ through the symplectomorphism, we are done.
\end{proof}

The last thing to do is to show that a whole neighborhood of $D_i$ admits a $T^2$ action that is compatible with all previously constructed actions. Note that this will only work if $D_i$ is a sphere containing exactly two nodal points, whereas the preceding paragraphs in this subsection work for any normal crossings divisor in an algebraic surface.

\begin{claim}
  For each $i$, there is a Hamiltonian $T^2$ action in a neighborhood of $D_i$, such that the previously constructed $S^1$ action agrees with the action of a one-parameter subgroup of $T^2$, and such that near the nodes this $T^2$ action agrees with the one previously constructed. 
\end{claim}

\begin{proof}
  We start from the proof of the previous claim, where we showed that a neighborhood $V_i$ of $D_i$ in $Y$ is symplectomorphic to a neighborhood of $D_i$ in $X_i$. Since $p_i : X_i \to D_i$ is a vector bundle with symplectic fibers, we find that $V_i$ has the structure of a symplectic fibration over $D_i$, we use the same letter to denote the projection $p_i : V_i \to D_i$. The fibers here are disks, and the $S^1$ action preserves the fibers. Let $q_0$ and $q_\infty$ denote the two nodal points in $D_i$. Near $q_0$, we have the $T^2$ action already constructed. Let $\rho_0 : S^1 \to T^2$ be the one parameter subgroup that rotates the \emph{base} direction at $q_0$ and fixes the \emph{fiber} $p_i^{-1}(q_0)$. Let $m_0$ be the local moment map for this action. Because of the local product structure near $q_0$, $m_0$ has the form $h_0 \circ p_i$, where $h_0$ is a function on $D_i$ defined near $q_0$. We may assume that $h_0$ has a local minimum at $q_0$, by inverting the one-parameter subgroup if necessary. Analogously, near $q_\infty$, let $\rho_\infty : S^1 \to T^2$ be the one-parameter subgroup that rotates the base and fixes the fiber $p_i^{-1}(q_\infty)$. It has a local moment map $m_\infty = h_\infty\circ p_i$, where we assume that $h_\infty$ has a local maximum at $q_\infty$.

Now we use crucially the fact that $D_i$ is a two-sphere and there are only two nodal points on it. We claim that there is a constant $c \in \R$ and a function $h : D_i \to \R$ such that $h = h_0$ near $q_0$ and $h = h_\infty + c$ near $q_\infty$, and \emph{which has no other critical points}. That is to say, there is a perfect Morse function on the two-sphere with minimum at $q_0$, maximum at $q_\infty$, and prescribed differential near these points. This is elementary. 

Now consider the function $m = h \circ p_i : V_i \to \R$, and the Hamiltonian vector field $X_m$. There is no reason for $X_m$ to generate a circle action, but we will show that it does generate an $\R$ action on $V_i$. Let $f$ be the Hamiltonian generating the $S^1$ action that rotates the fibers of $p_i : V_i \to D_i$. We claim that $m$ and $f$ Poisson commute. Indeed, we have
\begin{equation}
  \{f,m\} = \omega(X_f,X_m) = dm(X_f) = dh(dp_i(X_f)) = 0
\end{equation}
where $dp_i(X_f) = 0$ because the $S^1$ action preserves the fibers, and so $X_f$ is tangent to the fibers. Since $m$ commutes with $f$, $X_m$ is tangent to the level sets of $f$; since these level sets are compact (they are circle bundles over $D_i$, together with the critical level that is $D_i$ itself), the flow of $X_m$ is complete. Thus $X_m$ generates an $\R$ action on $V_i$.

Thus the pair $(m,f)$ is a maximal collection of Poisson commuting functions on $V_i$, generating an $\R \times S^1$ action. Our $T^2$ action will come from an application of the Arnold-Liouville theorem, so let us study the orbit structure of this action. The orbits are contained in the level sets of the map $(m,f) : V_i \to \R^2$. The zero-dimensional level sets consist of the nodal points $\{q_0\}$ and $\{q_\infty\}$, where both $df$ and $dm$ vanish; these are the zero-dimensional orbits. The one-dimensional level sets are of two kinds: some are the regular level sets of $h$ sitting in $D_i$, where $df = 0$ and $dm \neq 0$, and there are also the $S^1$ orbits in the fibers $p_i^{-1}(q_0)$ and $p_i^{-1}(q_\infty)$, where  $dm = 0$ and $df \neq 0$; these are the one-dimensional orbits. All other level sets are two-dimensional compact tori; indeed, they are the intersection of a regular level of $f$, which is a circle bundle over a the two-sphere, with a set of the form $p_i^{-1}(h^{-1}(C))$, where $h^{-1}(C)$ is some circle on the two-sphere. At a point on such a level set, $X_f$ and $X_m$ are linearly independent, since $X_f$ is tangent to the fiber of $p_i$ and $X_m$ is symplectically orthogonal to it. Thus these level sets are two-dimensional orbits. 

Now we apply the Arnold-Liouville theorem \cite{arnold}. This yields a set of action coordinates $(f_1,f_2) : V_i \to \R^2$ that when taken as Hamiltonians generate a $T^2$ action having the same orbits as the pair $(m,f)$. The action coordinates are uniquely determined up to integral affine transformations. It is possible to arrange that our original $f$ is one of the action coordinates, which shows that the $S^1$ action constructed previously agrees with a one-parameter subgroup of this $T^2$ action. Near the nodal points, the pair $(m,f)$ that generates the previously constructed $T^2$ action is a system of action coordinates, and so they are related to $(f_1,f_2)$ by a integral affine transformation (it will not necessarily be the same transformation at the two nodal points). Thus the $T^2$ actions agree at the nodal points as well.
\end{proof}

Let us summarize the outcome of this construction. In a neighborhood of each $D_i$, there is a Hamiltonian $T^2$ action. At the nodes, where two such neighborhoods overlap, the $T^2$ actions agree in the sense that there is an automorphism of $T^2$ that intertwines them, and moment maps of the two $T^2$ actions are related by an integral affine transformation. The orbits of all the $T^2$ actions are compact isotropic submanifolds; each is either a node of $D$, a circle on some component of $D$, or else a Lagrangian torus. In particular, a neighborhood of $D$, minus $D$ itself, is fibered by Lagrangian tori, and these tori are fibers of the local moment maps. 

Observe at this point that, since the torus fibers near $D$ are Lagrangian, the restriction of $\lambda$ to those fibers is closed. This means that each such torus $L$ has a well-defined \emph{Liouville class} $[\lambda|_L] \in H^1(L;\R)$; we shall study these classes later on.



\begin{figure}
  \centering
  \includegraphics[width=3in]{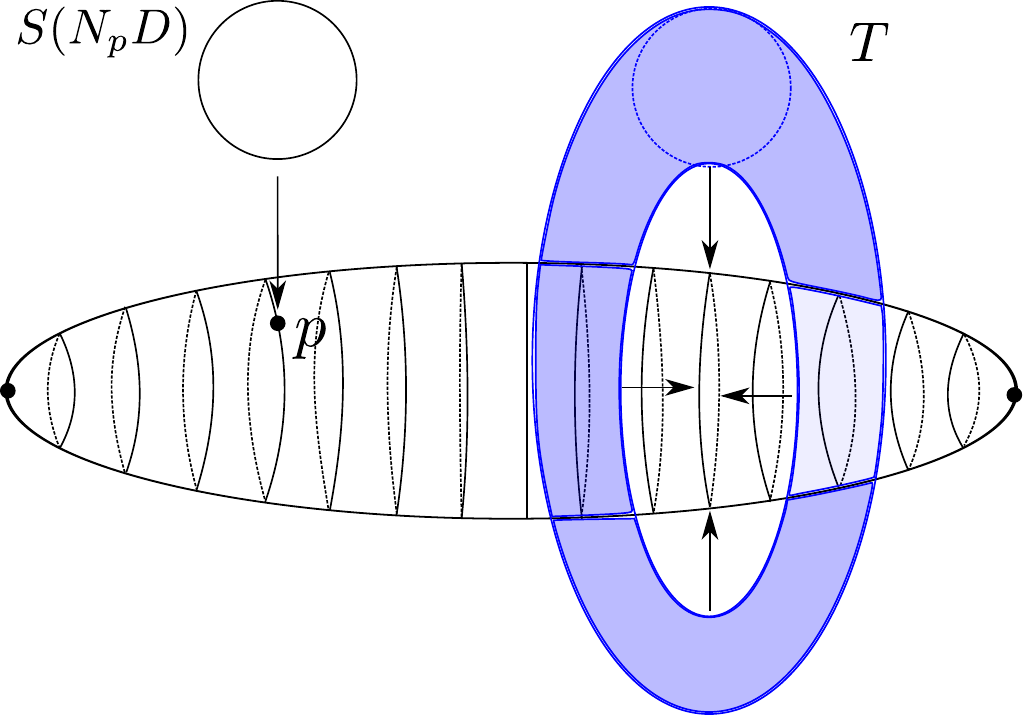}
  \caption{Tori in the normal bundle to a rational curve $D_i$. At left, a circle fiber of the normal bundle projecting to a point on $D_i$. At right, a torus projecting to a circle on $D_i$.} 
  \label{fig:edge-tori}
\end{figure}

\subsubsection{Step 2: making the Liouville form symmetric near the crossings}
\label{sec:step2}
This step is identical to \cite[\S 4, Step 2]{seidel-biased} so we will be brief. The outcome is that there is a smooth function $k$ on $Y$ such that $\lambda' = \lambda + dk$ is $T^2$--invariant near each of the crossings, and that the sub-level set $\overline{U} = \phi^{-1}(\infty,C]$ (for a large regular value $C$) equipped with the restriction of $\lambda'$ is still a Liouville domain. This Liouville structure is deformation equivalent to the original structure.

\subsubsection{Step 3: making the boundary torus-invariant}
\label{sec:step3}

The goal is to construct an exhausting family of Liouville domains $\overline{U}' \subset U$ such that the boundary $\partial\overline{U}'$ is invariant under the local torus actions. This condition tells us what to do: in a portion of the neighborhood of the divisor where the moment map is defined, let $S$ be a path that goes very close to the boundary of the moment map image. More precisely, we take the boundary of the moment map image, push it off into the interior of this image, and then round the corners in the simplest way. See Figure \ref{fig:node-boundary}. Then let $\Sigma$ be the union of the torus orbits over this path. By looking in several charts we may define $\Sigma$ as a closed real three dimensional manifold contained in a neighborhood of the divisor, which has the topology of a $T^2$--bundle over the circle $S$.

\begin{figure}
  \centering
  \includegraphics[width=2in]{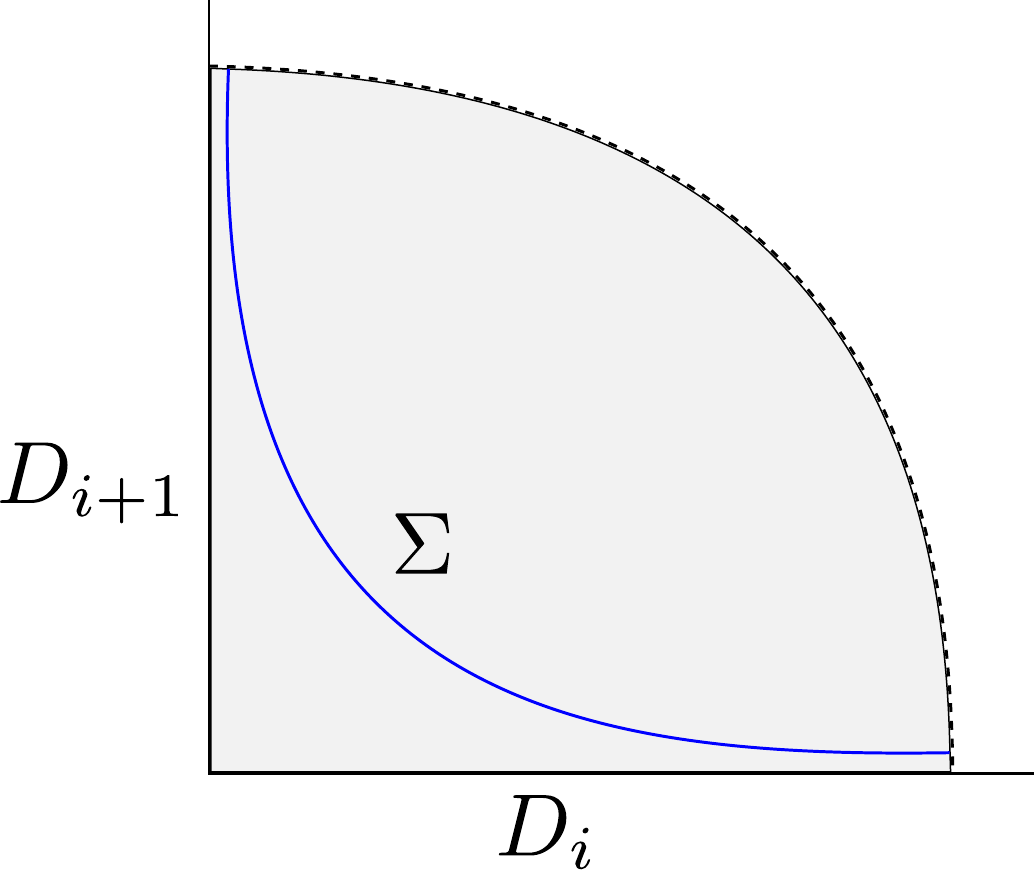}
  \caption{Moment map picture for the contact hypersurface near a node.}
  \label{fig:node-boundary}
\end{figure}

We take $\overline{U}'$ to be the inside of the real hypersurface $\Sigma$, namely the side not containing $D$. The arguments of \cite[\S 4, Step 3]{seidel-biased} apply to this hypersurface $\Sigma$ to show that if the path $S$ is taken close enough to the boundary, the Liouville vector field $Z'$ from Step 2 points outward along $\Sigma$. We also find that $(\overline{U},\lambda'|_{\overline{U}})$ and $(\overline{U}',\lambda'|_{\overline{U}'})$ are Liouville isomorphic.

\subsubsection{Step 4: making the contact form torus-invariant}
\label{sec:step4}

This is the same as in \cite{seidel-biased} but we use the $T^2$--action instead of just the $S^1$--action. Near the crossings, $\lambda'$ is already $T^2$--invariant, as is $\Sigma$, so $\lambda'|_\Sigma$ is a $T^2$--invariant contact form. Away from the crossings, we find that $Z'$ points outwards all along the torus fibers of $\Sigma$, so by averaging over the local $T^2$--actions we obtain a Liouville form $\lambda''$ defined in a neighborhood of $\Sigma$ that is $T^2$--invariant, and whose dual vector field $Z''$ points outwards along $\Sigma$. (Even though the $T^2$-actions are only locally defined, they are compatible; if we average in one region, and then another region, the second averaging does not destroy invariance in the first region.) The desired contact form is $\alpha = \lambda''|_\Sigma$. The contact structure is isomorphic to the one induced by $\lambda'$, and hence, just as in \cite[\S 4, Step 4]{seidel-biased}, we can find a Liouville structure on $\overline{U}'$ isomorphic to the one induced by $\lambda'$ and such that $\partial\overline{U}'$ has the contact form $\alpha$.

Let us now quote the following explicit description of the Reeb dynamics of $(\partial{\overline{U}'},\alpha)$ from \cite[\S 4, Step 4]{seidel-biased}, which will be used crucially in the paper: Over the parts of the boundary lying close to the smooth points of $D$, the Reeb flow is a circle action. Near the nodes, the local model is $\R \times T^2$, and the Reeb flow translates the torus $\{s\} \times T^2$ with some speed $\xi(s) = (\xi_1(s),\xi_2(s)) \in \R^2_+$. We have $\xi(s) = (\xi_+,0)$ for large $s \geq S$ and $\xi(s) = (0,\xi_-)$ for $s \leq -S$, so that the Reeb flow matches with the circle actions over the smooth parts of $D$. The function $\xi(s)$ will depend on the choice of boundary we made in Step 3, but by a suitable choice we can ensure that for $-S < s < S$, we have $\partial_s \xi_1 > 0$ and $\partial_s \xi_2 < 0$. 

\subsection{Reeb dynamics and Liouville classes}
\label{sec:reeb-dynamics}

Now we investigate further the contact boundary $(\Sigma,\alpha)$ that is the output of Step 4. Recall that $\Sigma$ is a torus bundle over a circle $S$.
\begin{equation}
\xymatrix{
  T^2 \ar[r] & \Sigma \ar[d]^\pi\\
  & S
}
\end{equation}
and the contact form $\alpha$ is invariant under the local $T^2$--action. Note that $\Sigma$ is not a principal $T^2$--bundle since there is no global $T^2$--action, and indeed, the structural group of the fiber bundle is not $T^2$--translations but rather diffeomorphisms of the torus. 

Let $I \subset S$ be some interval, and consider $\Sigma|_I = \pi^{-1}(I)$. Then $\Sigma|_I$ has a $T^2$ action, and by choosing a section $\sigma|_I : I \to \Sigma|_I$ of the fibration, we obtain an equivariant diffeomorphism
\begin{equation}
  \Sigma|_I \cong T^2 \times I
\end{equation}
Introduce coordinates $(\theta_1,\theta_2,s)$ on $T^2\times I$, where $(\theta_1,\theta_2)$ are $2\pi$--periodic coordinates on the fiber, and $s \in I$ is a coordinate on the base. Since the contact form $\alpha$ is a $T^2$--invariant one--form, in this coordinate system it can be written
\begin{equation}
  \label{eq:invariant-contact-form}
  \alpha = f(s)\,d\theta_1 + g(s)\,d\theta_2 + h(s)\,ds
\end{equation}
for some functions $f,g,h$ that depend on $s$ but not on the angular coordinates. We compute
\begin{equation}
  \label{eq:dalpha}
  d\alpha = f'(s)\,ds\wedge d\theta_1 + g'(s)\,ds\wedge d\theta_2 = ds \wedge \frac{\partial \alpha}{\partial s}
\end{equation}
Where the partial derivative $\partial \alpha/\partial s$ is taken with respect to the coframe $(d\theta_1,d\theta_2,ds)$. The volume form is
\begin{equation}
  \label{eq:volume-form}
  \alpha \wedge d\alpha = \alpha \wedge ds \wedge \frac{\partial \alpha}{\partial s} = -ds\wedge \alpha \wedge \frac{\partial \alpha}{\partial s}
  = -\begin{vmatrix} f(s) & g(s)\\ f'(s) & g'(s)\end{vmatrix} ds\wedge d\theta_1 \wedge d\theta_2
\end{equation}

\subsubsection{An orientation convention}
\label{sec:orientation-convention}

The orientation on $\Sigma$ is induced from the filling $\overline{U}'$ by the ``outward normal first'' convention, which is the same as the orientation induced by the volume form $\alpha \wedge d\alpha$. Since we have a fibration structure where the total space is canonically oriented, an orientation of the base is equivalent to an orientation of the fiber. The convention is that for $F \to E \to B$ a fibration, $\Lambda^{\text{max}}T^*E \cong \Lambda^{\text{max}}T^*B \otimes \Lambda^{\text{max}}T^*F$. Assuming that the base is oriented by the form $ds$, we see that the fiber is oriented by the volume form
\begin{equation}
  -\begin{vmatrix} f(s) & g(s)\\ f'(s) & g'(s)\end{vmatrix} d\theta_1 \wedge d\theta_2
\end{equation}
By switching the roles of $f,g$ and $\theta_1,\theta_2$ if necessary, we may assume that the determinant is negative, so that $d\theta_1\wedge d\theta_2$ is a positive volume form. A change in the orientation of the base also leads to a switch of this form.

\subsubsection{The Reeb vector field}
\label{sec:reeb-vf}

The Reeb vector field $R$ is determined by the conditions $\iota_R d\alpha = 0$ and $\alpha(R) = 1$. Expanding $R$ in the coordinate frame,
\begin{equation}
  R = R_1\partial_{\theta_1} + R_2\partial_{\theta_2} + R_s \partial_s
\end{equation}
We find that
\begin{equation}
  \label{eq:reeb-orth}
\begin{split}
  R_s &= 0\\
  f'R_1 + g'R_2 &= 0\\
  fR_1 + gR_2 & = 1
\end{split}
\end{equation}
Observe that $R$ is vertical, that is, tangent to the fibers of $\pi$. Since the Reeb vector field spans the characteristic foliation on $\Sigma$, we see that this foliation is tangent to the torus fibers, and on each torus fiber it consists of lines of some (rational or irrational) slope. The slope of this foliation varies as the fiber moves. See Figure \ref{fig:characteristic-foliation}.

\begin{figure}
  \centering
  \includegraphics[width=4in]{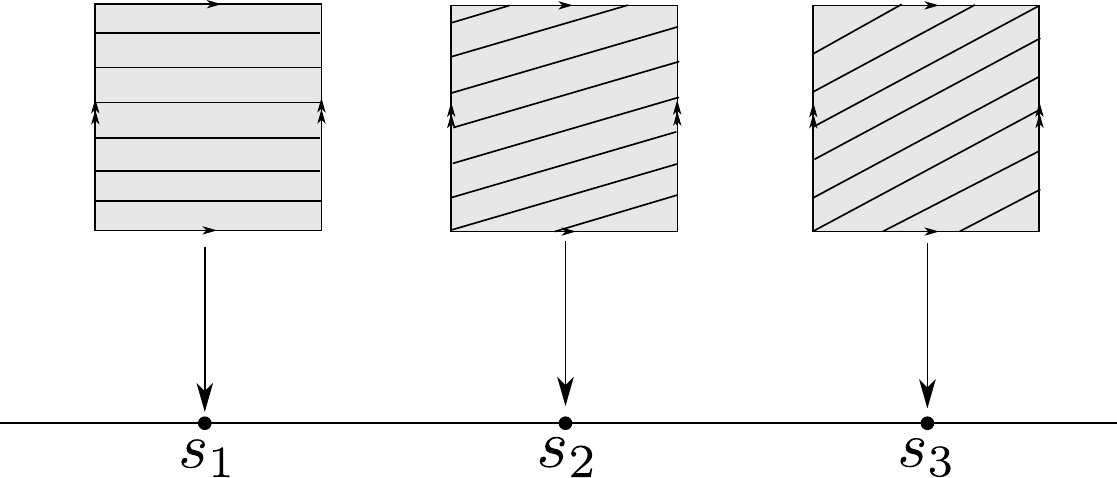}
  \caption{The characteristic foliation on various torus fibers.}
  \label{fig:characteristic-foliation}
\end{figure}

\subsubsection{The Liouville class}
\label{sec:liouville-class}

Now consider the torus fibers $T^2\times \{s\}$. By equation \eqref{eq:dalpha} we see that $d\alpha|_{T^2\times\{s\}} = 0$, that is, each torus is pre-Lagrangian. Thus we have the Liouville class
\begin{equation}
  \label{eq:liouville-class}
  A(s) = [\alpha|_{T^2\times\{s\}}] = f(s)[d\theta_1]+g(s)[d\theta_2] \in H^1(T^2\times\{s\},\R)
\end{equation}
The cohomology--valued function $A$ may be regarded as a section of the rank two vector bundle
\begin{equation}
  R^1\pi_*(\underline{\R}) \to S
\end{equation}
whose fibers are the first cohomology groups of the torus fibers $T^2 \times \{s\}$.
This vector bundle has a flat Gauss--Manin connection $\nabla$, defined so that the parallel transport coincides with transporting cohomology classes from fiber to fiber using the local product structure. Denote the monodromy of this connection by $\mu \in \SL_2(\Z)$. In particular, the expression $\nabla A$ defines a one-form on $S$ with values in this bundle. The cup product on the fibers then yields an element
\begin{equation}
A \cup \nabla A \in \Omega^1(S, R^2\pi_*(\underline{\R}))  
\end{equation}
Here $R^2\pi_*(\underline{\R})$ is a bundle whose fiber at $s$ is $H^2(T^2\times\{s\},\R) \cong \R$, and this vector bundle is trivializable using the orientation convention discussed above.

\subsubsection{Legendrian sections and elimination of $h(s)$}
\label{sec:legendrian-sections}

We observe that the condition that $\alpha$ is contact is equivalent to the non-vanishing of the determinant in the last expression of equation \eqref{eq:volume-form}. Since this determinant does not involve $h(s)$, we may change it arbitrarily while preserving the contact condition, and, by Gray's theorem, the isomorphism class of the contact structure.  The following lemma will be used to eliminate the $h(s)\,ds$ term by a change of coordinates.

\begin{lemma}
  \label{lem:legendrian-section}
  The fibration $\pi: \Sigma \to S$ admits a section $\sigma: S \to \Sigma$ whose image is a Legendrian circle in $\Sigma$.
\end{lemma}

\begin{proof}
  Choose a global smooth section $\sigma_0 : S\to \Sigma$ (it exists because the base is a circle and the fiber is connected). Letting $s$ denote a coordinate on $S$ and $\partial_s$ the coordinate vector field, consider $\alpha(T\sigma_0(\partial_s))$, which is a function on $S$. Then
\begin{equation}
\label{eq:alpha-forms}
  \alpha_t = \alpha - t\cdot\pi^*(\alpha(T\sigma_0(\partial_s))\,ds)
\end{equation}
is a family of contact forms starting from $\alpha_0 = \alpha$ and such that $\sigma_0(S)$ is a Legendrian circle for $\alpha_1$. Denote the corresponding contact structures by $\xi_t$. By Gray's theorem, there is an isotopy $\psi_t: \Sigma \to \Sigma$ such that $T\psi_t(\xi_0) = \xi_t$. Thus $\psi_1^{-1}(\sigma_0(S))$ is a Legendrian circle for the contact form $\alpha_0$.

To see that this circle is still a section of $\pi$, we look closer at the application of Moser's trick that is used to construct the isotopy \cite[p.~60]{geiges-contact-topology}. The isotopy $\psi_t$ is the flow of the time--dependent vector field $X_t \in \xi_t$ which is chosen to satisfy
\begin{equation}
  \dot{\alpha}_t + \iota_{X_t}d\alpha_t = \mu_t \alpha_t
\end{equation}
where
\begin{equation}
  \mu_t = \dot{\alpha}_t(R_t)
\end{equation}
Since $d\alpha_t$ is independent of $t$, we find that all the Reeb vector fields $R_t$ are vertical, and since the deformation term is pulled back from the base, they are actually all the same. Since $\dot{\alpha}_t$ is pulled back from $S$ and $R_t = R$ is vertical, we see that $\mu_t \equiv 0$. Equation \eqref{eq:alpha-forms} becomes
\begin{equation}
 -  \pi^*(\alpha(T\sigma_0(\partial_s))\,ds) + \iota_{X_t}d\alpha = 0
\end{equation}
Due to the specific form of $d\alpha$ in equation \eqref{eq:dalpha}, we see that $X_t$ is vertical, and is invariant under the local $T^2$--actions at all times. We conclude that the isotopy $\psi_t$ acts vertically (preserving the fibers) and equivariantly for the local $T^2$--actions. Thus $\sigma = \psi_1^{-1} \circ \sigma_0$ is a Legendrian section.
\end{proof}

From now on we will pick some Legendrian section $\sigma: S \to \Sigma$. Restricting $\sigma$ to some interval $I \subset S$, we obtain an equivariant diffeomorphism $\Sigma|_I \to T^2 \times I$ such that $\alpha$ has the form
\begin{equation}
  \alpha = f(s)\,d\theta_1 + g(s)\,d\theta_2
\end{equation}
The $h(s)\,ds$ term from equation \eqref{eq:invariant-contact-form} is not present since the lines $\{(\theta_1,\theta_2)\}\times I$ are now Legendrian, and so $\alpha(\partial_s)=0$. We now observe that such a contact form $\alpha$ is entirely determined by the Liouville class $A(s)$ \eqref{eq:liouville-class}.

\subsubsection{Properties of the Liouville class}
\label{sec:properties-liouville-class}

Observe that for a (pre-)Lagrangian torus $L \subset \Sigma \subset \overline{U}'$, we have
\begin{equation}
  [\alpha|_L] = [\lambda''|_L] = [(\iota(Z'')\omega)|_L]
\end{equation}
where $\lambda'', Z''$ are from the Liouville structure that is the output of Step 4. The properties of the Liouville class we need to use are just a translation of the properties of $Z''$ that are ensured by the construction in section \ref{sec:refinements}. However, the Liouville class is easier to compute since, being a cohomology--level object, it is more stable under deformations.

\begin{proposition}
\label{prop:stability-liouville-class}
Let $L \subset \Sigma \subset \overline{U}'$ be a Lagrangian torus near the boundary divisor $D$ which is an orbit of the local torus actions, as constructed in Step 1. Then the Liouville class $[\lambda''|_L]$ for the Liouville structure coming from Step 4 is equal to the Liouville class $[\lambda|_L]$ for the Liouville structure after Step 1.
\end{proposition}

\begin{proof}
  Step 2 does not alter the Liouville class because it changes the one-form by a globally exact form:
  \begin{equation}
    \lambda' = \lambda + dk
  \end{equation}
  Step 3 involves no change in the Liouville structure. Step 4 changes the Liouville structure by averaging $\lambda'$ over the local torus actions. Since $L$ is assumed to be an orbit of the torus action, we find
  \begin{equation}
    \lambda'' = \int_{T^2} \left((t_1,t_2)^*\lambda'\right)\,\frac{dt_1\,dt_2}{(2\pi)^2}
  \end{equation}
  but all the forms $(t_1,t_2)^*\lambda'$ in the integrand are cohomologous (since they are isotopic closed forms, this follows from the Cartan homotopy formula). So at the cohomology level we are just averaging a constant function.
\end{proof}

Because of the stability expressed in proposition \ref{prop:stability-liouville-class}, we can compute the Liouville classes of our tori using the Liouville structure coming from Step 1. Recall that this Liouville structure comes from K\"{a}hler geometry: $D = s^{-1}(0)$ for some section $s$ of the ample line bundle $\cL$, we have chosen $\Vert\cdot \Vert$ an Hermitian metric on $\cL$ such that $\phi = -\log \Vert s \Vert$ is the K\"{a}hler potential, and in Step 1 we ensured that the K\"{a}hler form is standard near the crossings.

Now we describe this in an analytic coordinate chart $V \subset Y$, possibly containing part of the divisor $D$. Over $V$ we also choose a holomorphic trivialization $\cL|_V \cong \cO|_V$. With respect to this trivialization, the holomorphic section $s$ becomes a holomorphic function that we also denote by $s$. The Hermitian metric $\Vert \cdot \Vert$ is, at each point, a positive multiple of the absolute value norm on $\cO$, so $\Vert s \Vert = e^\psi |s|$ for some function $\psi : V \to \R$. Thus
\begin{equation}
  \phi = -\log \Vert s \Vert = -\log e^\psi|s| = -\psi -\log|s|
\end{equation}
Now since $s$ is holomorphic and vanishes on $D$, we find that $-\log|s|$ is a function which is
discontinuous along $D$, and $dd^c(-\log|s|) = 0$ on the complement of $D$ (that is, $-\log|s|$ is pluriharmonic outside of $D$; this follows from Equation \eqref{eq:kahler-potential} and the fact that the trivial bundle is flat). If we interpret the formula in terms of weak derivatives, $dd^c(-\log|s|)$ is a current supported along $D$ (the Poincar\'{e}--Lelong formula). Thus, on the complement of $D$, $-\psi$ is another K\"{a}hler potential for the same form.

However, the term $-\log|s|$ contributes greatly to the Liouville one-form $\lambda = d^c \phi$. Consider the case near a crossing. We may choose the holomorphic coordinate chart $V$ and the holomorphic trivialization of $\cL$ so that $s = z_1^{a_1}z_2^{a_2}$ in local coordinates. If $z_j =\exp(\rho_j + i\theta_j)$,
\begin{equation}
  d^c(\log|z_j|) = d\theta_j
\end{equation}
Thus
\begin{equation}
  \lambda = d^c\phi = -d^c\psi - a_1d\theta_1 - a_2d\theta_2
\end{equation}
The term $-d^c\psi$ extends continuously across the divisor $D$, while the other terms do not. Thus we may write, where $z_j = x_j + iy_j$,
\begin{equation}
  -d^c\psi = \beta_1\,dx_1+\beta_2\,dy_1+\beta_3\,dx_2+\beta_4\,dy_2
\end{equation}
where the coefficients $\beta_1,\beta_2,\beta_3,\beta_4$ are bounded continuous functions on the chart $V$. 

Now suppose that $i: T^2 \to V$ is the embedding of a crossing torus $\{|z_1| = \epsilon_1, |z_2| = \epsilon_2\}$.
\begin{equation}
  i(\theta_1,\theta_2) = (\epsilon_1 \cos\theta_1,\epsilon_1\sin\theta_1,\epsilon_2\cos\theta_2,\epsilon_2\sin\theta_2)
\end{equation}
Then clearly $i^*dx_j$ and $i^*dy_j$ are both in the class $O(\epsilon_j)\,d\theta_j$ (big--O notation). Thus
\begin{equation}
  i^*(-d^c\phi) = O(\epsilon_1)\,d\theta_1 + O(\epsilon_2)\,d\theta_2
\end{equation}
and the Liouville class is
\begin{equation}
\label{eq:liouville-node-estimate}
  [\lambda|_{T^2}] = \left(-a_1+O(\epsilon_1)\right)[d\theta_1] + \left(-a_2+O(\epsilon_2)\right)[d\theta_2]
\end{equation}

A similar analysis works along the smooth parts of the divisors. The smooth part of each divisor is complex analytically a $\C^\times$, and we can take an analytic coordinate $z_B$ there ($B$ stands for base). We restrict $z_B$ to lie in some large annulus $A$ in order to avoid going right up to the nodes. The normal bundle of $D$ restricted to $A$ is then holomorphically trivial, and we let $z_F$ be a coordinate on the fibers ($F$ stands for fiber). We may write $s = z_F^a$.
\begin{equation}
  \lambda = d^c(-\log\Vert s\Vert) = -d^c\psi - d^c(\log|z_F^a|) = -d^c\psi - a\,d\theta_F
\end{equation}
Where $d\theta_F$ is the angular one-form on the fibers of the normal bundle. The tori in question are not necessarily standard in these coordinates, but if we take a torus $i: T^2 \to V$ which is within distance $\epsilon$ of the divisor, we have
\begin{equation}
\label{eq:liouville-divisor-estimate}
  [\lambda|_{T^2}] = (-a+O(\epsilon))[d\theta_F] + \beta[d\theta_B]
\end{equation}
where we do not assume any control over the function $\beta$.

\subsubsection{Step 5: making the Liouville class ``locally convex''}
\label{sec:step5}

Using the estimates \eqref{eq:liouville-node-estimate} and \eqref{eq:liouville-divisor-estimate}, we can plot the Liouville class $A$ in the first cohomology of the torus. The class $A$ is a section of the flat vector bundle $R^1\pi_*(\underline{\R}) \to S$ with monodromy $\mu \in \SL_2(\Z)$. Choose a basepoint $s_0 \in S$, and a parametrization $\tau: [0,1] \to S$ such that $\tau(0) = \tau(1) = s_0$. The flat connection trivializes the bundle pulled back to $[0,1]$, and identifies all the fibers of the bundle with a model $H^1(T^2,\R)$. With these identifications, the class $A(s) \in H^1(\pi^{-1}(s),\R)$ then becomes a path in $H^1(T^2,\R)$, which is such that
\begin{equation}
  A(\tau(1))= \mu A(\tau(0))
\end{equation}
The picture is of a path in $H^1(T^2,\R)$ such that the endpoint is the monodromy image of the starting point; see Figure \ref{fig:liouville-class}.
\begin{figure}
  \centering
  \includegraphics[width=2in]{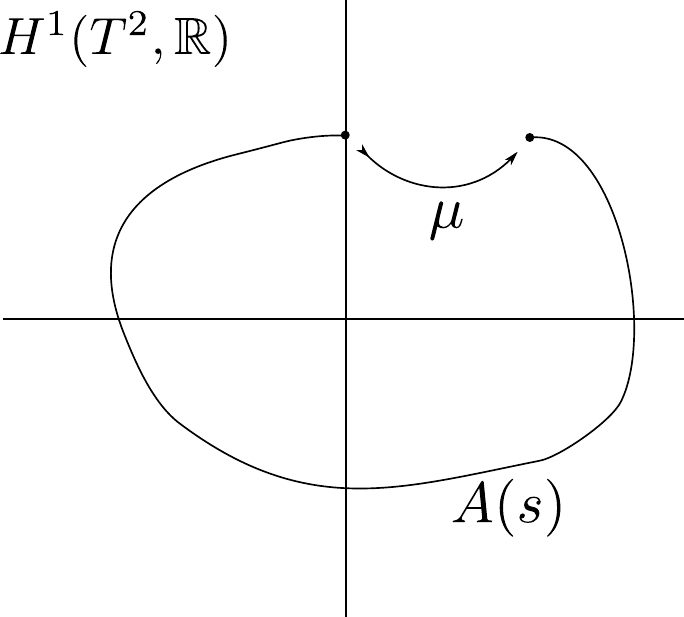}
  \caption{The Liouville class as a path in $H^1(T^2,\R)$.}
  \label{fig:liouville-class}
\end{figure}

The local model of the Reeb dynamics near the divisor from Step 4, together with the preceding analysis, implies the following properties of the Liouville class $A(s)$.

\begin{itemize}
\item The contact condition $\alpha \wedge d\alpha > 0$ becomes the condition that $A(s)$ as a path is always rotating clockwise with respect to the origin, once we orient everything as in section \ref{sec:orientation-convention} (see Equation \eqref{eq:volume-form}). This is equivalent to the condition $A(s) \wedge A'(s) < 0$.
\item Because the Reeb flow near the smooth parts of the divisors is a circle action, we find that the derivative $A'(s)$ must have constant direction in these regions, as the Reeb vector field is orthogonal to $A'(s)$ from \eqref{eq:reeb-orth}. Thus $A'(s) \wedge A''(s) = 0$ in these regions.
\item Recall that we have set things up in Step 4 so that the local model of the Reeb flow near the nodes is a translation by $\xi(s) = (\xi_1(s),\xi_1(s))$, where $\xi(s) = (\xi_+,0)$ for $s \geq S$, $\xi(s) = (0,\xi_-)$ for $s \leq -S$, and for $-S < s < S$, we have $\partial_s\xi_1 > 0$ and $\partial_s\xi_2 < 0$. Thus $A'(s)$ (with respect to the dual frame) equals $f(s)(-\xi_2(s),\xi_1(s))$, for some non-vanishing function $f(s)$, and we can easily compute that for $-S < s < S$,
  \begin{equation}
    A'(s) \wedge A''(s) = f^2(-\xi_2\partial_s\xi_1 - \xi_1(-\partial_s \xi_2)) < 0
  \end{equation}
\end{itemize}

The condition we wish to ensure is that this path is locally convex with respect to the origin.
\begin{definition}
  Let $\gamma: [0,1] \to \R^2$ be a path. Then $\gamma$ is locally convex with respect to the origin if sufficiently short secant lines of $\gamma$ lie closer to the origin than $\gamma$ does. For a $C^2$ path such that $\gamma\wedge \gamma' \neq 0$, this becomes the differential condition
  \begin{equation}
    \frac{\gamma'\wedge \gamma''}{\gamma \wedge \gamma'} > 0
  \end{equation}
\end{definition}

Since the path $A(s)$ has the property that $A\wedge A' < 0$, we just need to ensure that $A'\wedge A'' < 0$ as well. Another way to say this is that, in addition to $A(s)$ rotating clockwise, $A'(s)$ should always be turning to the right. According to the points above, we have that $A' \wedge A'' \leq 0$ every where. The only problem is that $A' \wedge A'' = 0$ on the parts of $\Sigma$ that are not near the nodes.

\begin{lemma}
  There is a conformal rescaling of $A$ (and hence of the contact form $\alpha$) such that $A' \wedge A'' < 0$.
\end{lemma}

\begin{proof}
Since $A' \wedge A'' \leq 0$, the only problem occurs in segments where the direction of $A'$ is constant. In these regions $A$ is a straight segment. The idea is to introduce a small ``bulge'' in the path $A$ such that $A' \wedge A'' < 0$. 
An explicit prescription is as follows. Let $[s_0,s_1]$ be a maximal interval of parameter values such that the direction of $A'(s)$ is a constant vector $v$ for $s \in [s_0,s_1]$. For small $\epsilon > 0$, $v_0 := A'(s_0-\epsilon)$ lies to the left of $v$, and $v_1:= A'(s_1+\epsilon)$ lies to the right of $v$. Then, on the interval $[s_0-\epsilon,s_1+\epsilon]$, we modify $A(s)$ so that the direction of $A'(s)$ slowly makes a small right turn from $v_0$ to $v_1$. Then we have $A' \wedge A'' < 0$.

\end{proof}

\begin{remark}
  We mention that contact structures on torus bundles with locally convex Liouville class were briefly considered in the work of Hutchings and Sullivan \cite[\S 12.2.2]{hutchings-sullivan}.
\end{remark}

This modification is effected by a conformal rescaling of the contact form $\alpha$ by a function that depends only on $s$, so it does not change the underlying contact structure or destroy the toroidal symmetry. Hence there is an isomorphic Liouville structure on $\overline{U}'$ whose boundary contact form is exactly the locally convex $\alpha$.

So as not to completely lose the reader with this somewhat strange definition, we put here the crucial lemma that this modification makes possible. 

\begin{lemma}
  \label{lem:convex-maximum-property}
  Let $\varrho \in H_1(\pi^{-1}(s_0),\Z)$ be a class of loops such that 
  \begin{equation}
    \int_{\varrho} A(s_0) > 0, \quad\int_{\varrho} A'(s_0) = 0.
  \end{equation}
  By parallel transport extend to $\varrho \in H_1(\pi^{-1}(s),\Z)$ for nearby fibers ($s$ near $s_0$). If $A(s)$ is locally convex, then the period integral
  \begin{equation}
    I(s) = \int_{\varrho} A(s)
  \end{equation}
has a non-degenerate local maximum at $s = s_0$.
\end{lemma}

Before the proof, we observe that by section \ref{sec:reeb-vf}, the hypothesis is satisfied whenever $\rho$ is the homology class represented by a closed Reeb orbit (see section \ref{sec:reeb-orbits}).

\begin{proof}
  First observe that $s_0$ is a critical point:
  \begin{equation}
    I'(s_0) = \int_{\varrho}A'(s_0) = 0
  \end{equation}
 It remains to show that $I''(s_0) < 0$. The local convexity condition
 \begin{equation}
   \frac{A''\wedge A'}{A \wedge A'} < 0
 \end{equation}
 means that $A$ and $A''$ lie on opposite sides of the line spanned by $A'$. Since $\int_\varrho$ is a linear function that vanishes on that line, it must take opposite signs on $A$ and $A''$. Since it is positive on $A$, it must be negative on $A''$:
 \begin{equation}
   I''(s_0) = \int_\varrho A''(s_0) \text{\quad is opposite to \quad} \int_{\varrho} A(s_0) > 0
 \end{equation}

\end{proof}

\section{The affine manifold}
\label{sec:affine-man}

We recall the construction of the affine manifolds from \cite{ghk}, but with an emphasis on how it follows from the topology of the neighborhood of the boundary divisor. This section consists of elementary $4$--dimensional topology.

\subsection{Initial data}
\label{sec:initial-data}

Let $D$ be a cycle of $m$ rational curves. By this we mean a curve with $m$ nodes whose normalization is the disjoint union of $m$ rational curves, and such that the dual intersection complex is a cycle. If $m = 1$ this means a nodal elliptic curve. In what follows we assume $m \geq 2$; we can reduce to this case by blowing up the node if $m = 1$. As a matter of notation, we will index the \emph{nodes} by $i \in \Z/m\Z$, and we will index the rational curves by either a single index or a pair: 
\begin{equation}
D_i = D_{i,i+1} \text{ connects node $i$ to node $i+1$.}
\end{equation}

Suppose that $\iota: D \to Y$ is an embedding of the singular curve $D$ into an algebraic surface $Y$. If $m > 1$ then each component $D_i \cong \PP^1$ is embedded and has a normal bundle of some degree $k_i$:
\begin{equation}
N_i = \iota^*TY/TD_i \cong \cO_{\PP^1}(k_i).
\end{equation}
If $m = 1$ we mean $N_1 = (\iota\circ\nu)^*(TY)/T\PP^1$, where $\nu : \PP^1 \to D$ is the normalization of the nodal elliptic curve. The topology of a neighborhood of $D$ in $Y$ is completely determined by the numbers $k_i$, which are also characterized as the self-intersection numbers in $Y$:
\begin{equation}
  k_i = D_i^2 
\end{equation}

\subsection{The local model for nodes}
\label{sec:node-model}

 Let $\Delta = \{|z| \leq 1\}$ and $\Delta^* = \Delta \setminus \{0\}$ denote the complex unit disk and its punctured version.

Let us consider the $i$th node. Here the two divisors $D_{i-1,i}$ and $D_{i,i+1}$ intersect transversely. Let us take take a neighborhood $V_i$ and local analytic coordinates $(z_i,w_i) \in \Delta \times \Delta \cong V_i$ such that, locally near the node,
\begin{align}
  D_{i-1,i} &= \{w_i = 0\}\\
  D_{i,i+1} &= \{z_i = 0\}
\end{align}

We may define certain $1$--cycles:
\begin{align}
  \Gamma_{z_i} &= \{(z_i,w_i) = (e^{-i\theta}, 1)\mid \theta \in [0,2\pi)\}\\
  \Gamma_{w_i} &= \{(z_i,w_i) = (1, e^{-i\theta}) \mid \theta \in [0,2\pi)\}
\end{align}
The local model for $U = Y\setminus D$ near the node is simply $(z_i,w_i) \in \Delta^*\times \Delta^*$, which is fibered by the tori $T_{r,s} = \{|z_i| = r, |w_i| = s\}$ where $0<r,s\leq 1$. The loops $\Gamma_{z_i}$ and $\Gamma_{w_i}$ are a basis for the first homology of this neighborhood in $U$.

\begin{remark}
Note that the orientation is such that $\Gamma_{z_i}$ winds \emph{clockwise} in the $z_i$--plane around $\{z_i=0\} = D_{i,i+1}$. In a later section this will be justified by the observation that Reeb orbits near the node are homologous to non-negative linear combinations of $\Gamma_{z_i}$ and $\Gamma_{w_i}$.
\end{remark}

\subsection{The local model for divisors}
\label{sec:divisor-model}

The divisor $D_i$ has a tubular neighborhood isomorphic to a disk bundle $DN_i$ in the normal bundle $N_i \cong \cO(k_i)$. Let $SN_i$ denote the boundary of this tubular neighborhood; it is a circle bundle over $D_i$ and is diffeomorphic to a lens space $L(-k_i,1)$. Two points on $D_i = D_{i,i+1}$ are distinguished: the $i$th and $(i+1)$th nodes. Let $W_{i,i+1}$ denote the complement of $V_i$ and $V_{i+1}$ (the neighborhoods of the nodes) in $DN_i$: 
\begin{equation}
W_{i,i+1} = DN_i \setminus (V_i \cup V_{i+1}).
\end{equation}
For sufficiently small $DN_i$,  $W_{i,i+1}$ is a disk bundle over an annulus, namely the $2$--sphere $D_{i,i+1}$ with two disks removed at the nodes, which we denote by $A_{i,i+1}$. The associated circle bundle $SW_{i,i+1} = W_{i,i+1} \cap SN_i$ is the lens space $SN_i$ with two solid tori removed, and it is non-canonically diffeomorphic to a product $T^2\times I$. The complement of $D$ in this neighborhood $W_{i,i+1} \setminus D_{i,i+1}$ deformation retracts onto the circle bundle, so it has the same first homology. 

The circle fibration gives a long exact sequence in homotopy groups, which reduces to 
\begin{equation}
  \label{eq:h1-exact-seq}
  0 \to \pi_1(F) \to \pi_1(SW_{i,i+1}) \to \pi_1(A_{i,i+1}) \to 0
\end{equation}
where $F$ is the circle fiber over some basepoint in $A_{i,i+1}$. As all three fundamental groups are abelian, this sequence holds with $\pi_1$ replaced by $H_1$, and we will use this notation from now on to sidestep concerns over basepoints. The sequence \eqref{eq:h1-exact-seq} is split, but not canonically. In fact, we have two geometric splittings, induced by the bases $(\Gamma_{z_i},\Gamma_{w_i})$ and $(\Gamma_{z_{i+1}}, \Gamma_{w_{i+1}})$ at the two nodes, as we shall now elaborate.

Near the $i$th node, $D_{i,i+1}$ is given by the equation $z_i = 0$. Thus $\Gamma_{z_i}$ is a loop which links the divisor $D_{i,i+1}$, and hence is homologous to the circle fiber of $SW_{i,i+1}$ equipped with some orientation. Similarly, near the $(i+1)$th node, $D_{i,i+1}$ is given by the equation $w_{i+1} = 0$, and so $\Gamma_{w_{i+1}}$ is a loop which links the divisor, and is also homologous to the circle fiber equipped with some orientation. In fact these orientations agree, as each is a loop in the fiber of the normal bundle which encircles zero clockwise with respect to the natural orientation of the fibers of the normal bundle as complex lines. Thus:
\begin{equation}
  \label{eq:fiber-class}
  \Gamma_{z_i} \sim f \sim \Gamma_{w_{i+1}}
\end{equation}
where $f \in H_1(F)$ is the class of the fiber with appropriate orientation, and $\sim$ means ``is homologous to.''

Near the $i$th node the loop $\Gamma_{w_i}$ projects onto $D_{i,i+1}$ as a loop encircling the $i$th node clockwise with respect to the complex orientation on $D_{i,i+1}$. Near the $(i+1)$th node, the loop $\Gamma_{z_{i+1}}$ projects onto $D_{i,i+1}$ as a loop encircling the $(i+1)$th node clockwise. So together these loops form the oriented boundary of the annulus $A_{i,i+1}$, and we find
\begin{equation}
  -\pi_*\Gamma_{w_i} \sim b \sim \pi_*\Gamma_{z_{i+1}}
\end{equation}
where $\pi: SW_{i,i+1} \to A_{i,i+1}$ is the circle bundle over the annulus and $b \in H_1(A_{i,i+1})$ is an appropriate generator.

\begin{lemma}
It holds that
\begin{equation}
  \label{eq:homology-relation-1}
  -\Gamma_{w_i} \sim \Gamma_{z_{i+1}} + k_i f
\end{equation}
\end{lemma}

\begin{proof}
  It is clear from the preceding discussion that $-\Gamma_{w_i} \sim \Gamma_{z_{i+1}} + \alpha f$ in $H_1(SW_{i,i+1})$ for some constant $\alpha$ to be determined.  We must show $\alpha = k_i$.

An elementary computation with the clutching functions of $N_i \cong \cO_{\PP^1}(k_i)$ shows how $\alpha$ is related to $k_i$. Write $\PP^1$ as the union of two complex coordinate charts $U = \{\zeta \in \C\}$, $V = \{\eta \in \C\}$, glued by the correspondence $\zeta = 1/\eta$. Let $\cL = \cO_{\PP^1}(p)$ be a line bundle of degree $p$. Then we have local trivializations
\begin{align}
  (\zeta,\xi) &\in U \times \C \cong \cL|_U\\
(\eta, \gamma) &\in V \times \C \cong \cL|_V\\
(\zeta, \xi) &= (1/\eta,\zeta^p \gamma)
\end{align}
For $p \geq 0$, one verifies that the sections $s$ given by 
\begin{align}
  \xi &= (s|_U)(\zeta) = \zeta^r\\  
  \gamma &= (s|_V)(\eta) = \eta^{p-r}
\end{align}
for $0 \leq r \leq p$ are valid holomorphic sections, so this line bundle really does have degree $p$. 

Let $\zeta = \ell(\theta) = e^{i\theta}$ be a loop in $\PP^1$. We may lift this loop to the $U$--trivialization as $(\zeta,\xi) = \ell'(\theta) = (e^{i\theta}, 1)$, and to the $V$--trivialization as $(\eta,\gamma) = \ell'(\theta) = (e^{-i\theta},e^{-pi\theta})$ (these loops are geometrically identical). On the other hand, we have the lift $\ell''(\theta)$ given by $(\zeta,\xi) = \ell''(\theta) = (e^{i\theta}, e^{pi\theta})$ and $(\eta,\gamma) = \ell''(\theta) = (e^{-i\theta}, 1)$

 Let $f(\tau) = (1,e^{-i\tau})$ be a loop in the fiber over the point $\zeta = \eta = 1$, which encircles the zero--section clockwise; this loop is given by the same formula in either trivialization. Let $Z = \cL|_{U\cap V} \setminus (U\cap V)$ be the complement of the fibers over $\zeta = 0,\infty$ and the zero--section. Then $H_1(Z;\Z) = \Z^2$, and the classes $[\ell'],[\ell'']$ and $[f]$ are elements of this group. Evidently, we have the relation
 \begin{equation}
   \label{eq:homology-relation-2}
   [\ell'] = [\ell''] + p[f]
 \end{equation}

The space $Z$ is homotopy equivalent to $SW_{i,i+1}$, taking $p := k_i$. Under this correspondence, $[f]$ corresponds to the class of $f$ as in equation \eqref{eq:fiber-class}, $[\ell']$ corresponds to the class of $-\Gamma_{w_i}$, and $[\ell'']$ corresponds to the class of $\Gamma_{z_{i+1}}$. Comparing equations \eqref{eq:homology-relation-1} and \eqref{eq:homology-relation-2}, we see that $\alpha = p = k_i$.
\end{proof}

\subsection{Affine charts and gluing}
\label{sec:affine-charts}

For each node, indexed by $i \in \Z/m\Z$, define an integral cone $Q_i$ that is the non-negative span of $\Gamma_{z_i}$ and $\Gamma_{w_i}$:
\begin{equation}
  Q_i := \{a[\Gamma_{z_i}]+b[\Gamma_{w_i}] \mid a,b \in \Z_{\geq 0}\} 
\end{equation}
These sums may be interpreted as classes in $H_1(V_i\setminus D; \Z)$, where we recall that $V_i$ is the neighborhood of the node. 

We shall show how these cones glue up into an integral linear manifold. Recall that an \emph{integral affine structure} on a manifold is an atlas with transition maps in $\GL(n,\Z)\ltimes \Z^n$. An \emph{integral linear structure} has all transition maps in $\GL(n,\Z)$, basically meaning that the manifold has a well-defined origin.

The analysis of \S \ref{sec:divisor-model} indicates how to glue $Q_i$ and $Q_{i+1}$. Equations \eqref{eq:fiber-class} and \eqref{eq:homology-relation-1} tell us that
\begin{equation}
  \label{eq:gluing-matrix}
  a\Gamma_{z_i}+b\Gamma_{w_i} \sim a'\Gamma_{z_{i+1}}+b' \Gamma_{w_{i+1}} \iff
  \begin{pmatrix}a' \\ b'\end{pmatrix} = 
  \begin{pmatrix} 0 & -1\\ 1 & -k_i\end{pmatrix}
  \begin{pmatrix}a\\b  \end{pmatrix}
\end{equation}
where the homological relation in equation \eqref{eq:gluing-matrix} holds in $SW_{i,i+1}$. When we identify the lattices $H_1(V_i\setminus D;\Z)$ and $H_1(V_{i+1}\setminus D;\Z)$ using this linear transformation, we find that the images of $Q_i$ and $Q_{i+1}$ intersect along the ray spanned by $\Gamma_{z_i} \sim \Gamma_{w_{i+1}}$:
\begin{equation}
  Q_i \cap Q_{i+1} = \Z_{\geq 0} \cdot [\Gamma_{z_i}] = \Z_{\geq 0} \cdot [\Gamma_{w_{i+1}}]
\end{equation}
Thus we may glue $Q_i$ to $Q_{i+1}$ along this common edge to define an integral linear structure on the union $Q_i \cup Q_{i+1}$. This structure is the one induced by embedding the two cones into a common lattice as above.






There is another characterization of this linear structure in terms of the intersection form of $Y$. A integral linear structure is determined by the corresponding sheaf of integral linear functions. An integral linear function $f: Q_i \cup Q_{i+1} \to \Z$ is determined by three numbers $\alpha = f(\Gamma_{w_i})$, $\beta = f(\Gamma_{z_i}) = f(\Gamma_{w_{i+1}})$ and $\gamma = f(\Gamma_{z_{i+1}})$, as the linear structure within each cone is standard. In order for $f$ to be linear we need
\begin{equation}
  \alpha = f(\Gamma_{w_i}) = f(-\Gamma_{z_{i+1}}-k_i \Gamma_{w_{i+1}}) = -\gamma - k_i \beta
\end{equation}
Or in other words $\alpha + k_i \beta + \gamma = 0$. Recalling that $k_i = D_i^2$, and $D_{i-1}$ and $D_{i+1}$ are transverse to $D_i$, this is equivalent to the orthogonality condition
\begin{equation}
  (\alpha D_{i-1} + \beta D_i + \gamma D_{i+1})\cdot D_i = 0
\end{equation}

We define a singular integral linear manifold $U^\trop$ to be the union of the cones $Q_i$, glued along edges as above. An example is depicted in Figure \ref{fig:integral-points}. This manifold has $Q_i \cup Q_{i+1}$ as charts, but there is no way to extend the linear (or even affine) structure to the origin (the triple overlaps of the charts), so we simply regard that as a singularity. As the manifold $U^\trop$ is defined over $\Z$, we use the notation $U^\trop(\Z)$ or $U^\trop(\R)$ to denote integral and real points respectively. The real points $U^\trop(\R) \setminus \{0\}$ form an affine manifold in the usual sense. A special feature of the surface case is that $U^\trop(\R)$ (with the singular point included) is actually a topological manifold homeomorphic to $\R^2$.

A locally linear function $f: U^\trop \setminus \{0\}\to \Z$ is determined by its values on the rays $\alpha_j = f(\Gamma_{z_j})$ and the condition of global linearity is equivalent to the orthogonality condition
\begin{equation}
 (\forall i) \quad \left(\sum_{j=1}^m \alpha_j D_j\right)\cdot D_i = 0
\end{equation}

\section{The differential on symplectic cohomology}
\label{sec:differential}

In this section we compute the differential on symplectic cohomology and prove theorem \ref{thm:main}. 

\subsection{Holomorphic curves in $\Sigma$}
\label{sec:holomorphic-cylinders}

The basis for our method of computation is the following nonexistence result. It is proved by adapting a method of Bourgeois and Colin \cite{bourgeois-colin}. We remark that this approach very much uses the low-dimensionality of our situation. Throughout this subsection, it may be helpful to have in mind Figure \ref{fig:hol-curve} that depicts the projection to $\Sigma$ of a holomorphic curve in relation to the torus fibration $\pi : \Sigma \to S$. The following Theorem concerns holomorphic curves with several positive and negative punctures, but in the case where there is more than one positive puncture, we require all the corresponding Reeb orbits to lie in a single periodic torus (there is no analogue without this condition).
\begin{figure}
  \centering
  \includegraphics[width=3in]{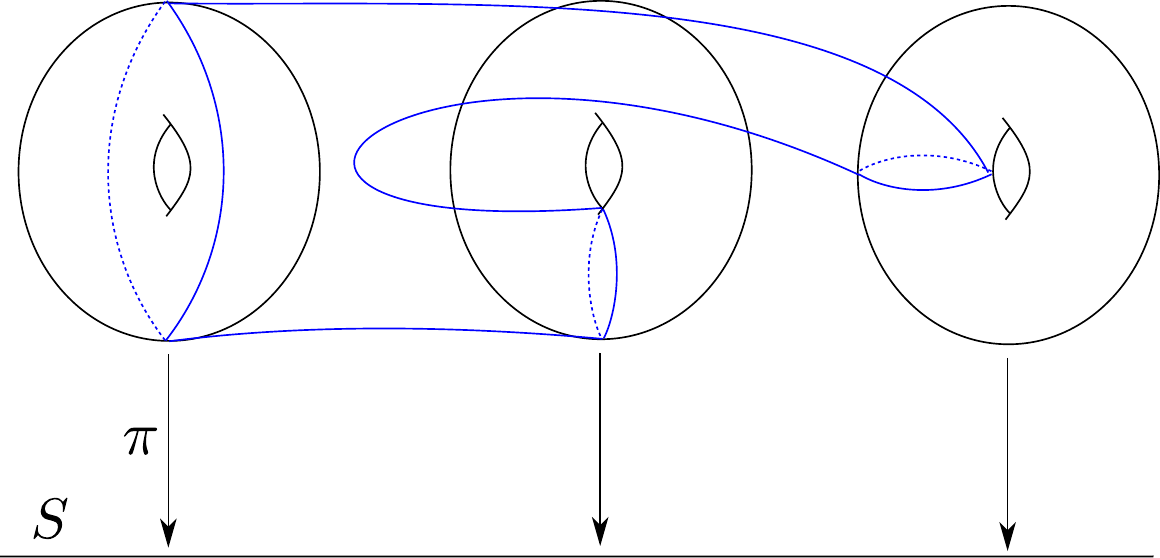}
  \caption{The projection to $\Sigma$ of a holomorphic curve in relation to the projection $\pi : \Sigma \to S$.}
  \label{fig:hol-curve}
\end{figure}

\begin{theorem}
  \label{thm:no-hol-cylinders}
  Let $(\Sigma,\alpha)$ be as above, so that the Liouville class is locally convex. Let $f: C \to \Sigma$ be a map with finite energy satisfying $J_{\xi} \circ \pi_{\xi} \circ df = \pi_{\xi} \circ df \circ j$, where the domain Riemann surface $C$ has genus zero, with several positive punctures and several negative punctures asymptotic to Reeb orbits in $\Sigma$. We assume that all Reeb orbits at the positive punctures lie in a \emph{single} periodic torus in $\Sigma$ (this is automatic if there is only one positive puncture). Then $f$ is trivial, that is, the image of $f$ is a closed integral curve of the Reeb field $R$.
\end{theorem}

\begin{proof}
By definition, the map $f$ satisfies
  \begin{equation}
    J_\xi \circ \pi_\xi \circ df = \pi_\xi \circ df \circ j
  \end{equation}
  where $j$ is the complex structure on the domain $C$, $\pi_\xi : T\Sigma \to \xi$ is the projection with kernel spanned by the Reeb vector, and $J_\xi$ is the almost complex structure on the contact distribution. Consider the two form $f^*d\alpha$ on $C$. By proposition \ref{prop:positivity-contact-target}, $f^*d\alpha$ is non-negative and only vanishes at points  where $df$ maps $TC$ into the line spanned by $R$.

We furthermore consider the composition of $f$ with the projection to the base $\pi: \Sigma \to S$. We first show that if $\pi\circ f$ is constant then the conclusion of the theorem follows. Observe that the tangent map $d\pi$ maps $\xi$ surjectively onto $TS$ with a one-dimensional kernel. If $\pi \circ f$ is constant, then $df(v)$ lies in $\ker d\pi$ for every $v$, and so $\pi_\xi \circ df(v)$ lies in the line bundle $\ker d\pi \cap \xi$ for every $v$. Thus $\pi_\xi \circ df(v)$ and $\pi_\xi\circ df(jv)$ are proportional, and their symplectic pairing under $d\alpha$ vanishes. Thus $\pi_\xi\circ df(v) = 0$ for every $v$. Hence $df(v)$ is always proportional to $R$, and the image of $f$ must necessarily be a closed integral curve of $R$, which was to be shown.

It remains to rule out the possibility that $\pi\circ f$ is non-constant. This is due to a \emph{local} energy obstruction to the existence of maps with certain topologies. We have split off this part of the argument into the following lemma, which completes the proof of the theorem.
\end{proof}

\begin{lemma}
  Let $f$ satisfy the hypothesis of theorem \ref{thm:no-hol-cylinders}. Then the composition $\pi \circ f$ is constant.
\end{lemma}
\begin{proof}
Suppose that the projection $\pi \circ f$ is non-constant. At each puncture, $f$ is asymptotic to a Reeb orbit, which by assumptions lies entirely within a torus fiber of $\pi$. We may compactify $C$ to $\bar{C}$ by adding a circle around every puncture and extend $f$ to a continuous map $\bar{C} \to \Sigma$, which we also denote by $f$, that maps these boundary circles to the Reeb orbits. Because these Reeb orbits are mapped to points by $\pi$, the composite map $\pi \circ f : C \to S$ may be extended to a continuous map whose domain is $\CP^1$. Hence a small loop around the puncture maps to a null-homotopic loop in $S$. As the fundamental group of $C$ is generated by such loops, we see that $\pi \circ f$ is null-homotopic. If $\tilde{\pi}: \tilde{\Sigma} \to \tilde{S}$ is the pull back of the torus fibration to the universal cover of the base, we find that $f : \bar{C} \to \Sigma$ lifts to $\tilde{f} : \bar{C} \to \tilde{\Sigma}$. Now consider the projection $\tilde{\pi}\circ \tilde{f} : \bar{C} \to \tilde{S}$, and note $\tilde{S} \cong \R$. As $\bar{C}$ is compact, we find that $\tilde{\pi}\circ\tilde{f}$ has a global maximum and minimum. Since $\pi\circ f$ is assumed non-constant, and all positive punctures are assumed to be asymptotic to orbits lying in a single periodic torus, at most one of these global extrema can be the limit of a positive puncture. So we obtain a point $s_0 \in \tilde{S}$ which is a global extreme value of $\tilde{\pi}\circ \tilde{f}$, and which is not the limit of a positive puncture.

Now we claim that $s_0$ must be the limit at a negative puncture. If not, then let $s_1 \in \tilde{S}$ be a regular value of $\tilde{\pi}\circ\tilde{f}$ that lies between $s_0$ and the nearest point which is a limit of a puncture. Then $P = (\tilde{\pi}\circ\tilde{f})^{-1}[s_0,s_1] \subset C$ is a compact submanifold whose boundary is a collection of smooth curves. These boundary curves map to the fiber $\tilde{\pi}^{-1}(s_1)$. We claim
\begin{equation}
  \int_P d\alpha = \int_{\partial P} \alpha = 0
\end{equation}
This follows once we show that $u|_{\partial P}$ is null-homologous in the fiber torus $\tilde{\pi}^{-1}(s_1)$, since $\alpha$ is closed on this torus. But indeed, the capping surface $u|_{P}$ can be pushed entirely inside of this fiber, since the torus fibration over the interval $[s_0,s_1]$ is trivial. Now we have that $u|_{P}$ is a holomorphic curve with vanishing $d\alpha$--area. Thus $u(P)$ must be contained in a closed integral curve of $R$, contradicting the construction of $P$.

The only remaining possibility is that the extremum $s_0$ is the limit at a negative puncture. In the contact manifold $\tilde{\Sigma}$, the map $\tilde{f}$ is asymptotic to a closed Reeb orbit at that puncture. Since the puncture is negative, the orientation of this curve by the vector field $R$ is \emph{opposed} to the orientation of this curve as the boundary of $\bar{C}$ (where we compactify $C$ to $\bar{C}$ by adding a circle around every puncture). Again choose a regular value $s_1$ between $s_0$ and the nearest point which is a limit of a puncture. Then $P = (\tilde{\pi}\circ\tilde{f})^{-1}[s_0,s_1] \subset \bar{C}$ is a compact surface in $\bar{C}$ whose boundary consists of some circles $\partial_1 P \subset \tilde{\pi}^{-1}(s_1)$, along with some circles in the fiber $\partial_0P \subset \tilde{\pi}^{-1}(s_0)$. We orient the boundary so that $\partial P = \partial_1P - \partial_0P$, and then the circles comprising $\partial_0P$ are all geometrically Reeb orbits with the Reeb orientation.

Because, as before, the entire surface $P$ can be pushed into one fiber of $\tilde{\pi}$, we find that each of the regular level sets $(\tilde{\pi}\circ \tilde{f})^{-1}(s)$ represents the same homology class in the torus fiber for $s_0 < s < s_1$. Call this class $\varrho$. Looking at $\partial_0P$, we see that this class $\varrho$ is a strictly positive multiple of the class of the primitive Reeb orbit in the fiber $\tilde{\pi}^{-1}(s_0)$.

Now we shall apply lemma \ref{lem:convex-maximum-property}. Consider the ``period integral''
\begin{equation}
  I(s) = \int_\varrho A(s)
\end{equation}
By applying lemma \ref{lem:convex-maximum-property} to $\varrho$, we find that $I(s)$ has a non-degenerate local maximum at $s = s_0$. Now the $d\alpha$-area of $f$ is
\begin{equation}
  \int_P d\alpha = \int_{\partial_1P} \alpha - \int_{\partial_0P} \alpha = \int_\varrho A(s_1) - \int_\varrho A(s_0) = I(s_1) - I(s_0) < 0
\end{equation}
which is impossible.

\end{proof}

\begin{corollary}
  \label{cor:no-hol-cylinders-in-symplectization}
  Let $u: C \to \Sigma \times \R$ be a pseudo-holomorphic curve in the symplectization of $\Sigma$ (where the complex structure on $\Sigma \times \R$ is cylindrical), such that all positive punctures are asymptotic to Reeb orbits lying in a single periodic torus. Then $u$ is trivial, that is, the image of its projection to $\Sigma$ consists of a closed integral curve of $R$.

  Let $H : \Sigma \times \R \to \R$ be a function that depends only on the $\R$ component. Let $u: C \to \Sigma \times \R$ be an inhomogeneous pseudo-holomorphic curve with Hamiltonian $H$ satisfying the same condition on the positive punctures. Then $u$ is trivial.\end{corollary}
\begin{proof}
  Clear from propositions \ref{prop:hol-map-decomp} and \ref{prop:floer-map-decomp}.
\end{proof}

\subsection{Neck stretching}
\label{sec:neck-stretching}

It is possible to combine these nonexistence results with an SFT-style neck-stretching argument to prove nonexistence in other situations. This neck-stretching technique was introduced by Bourgeois and Oancea \cite[\S 5]{bo-exact-sequence} in order to construct an exact sequence relating symplectic cohomology and contact homology. We comment that the main result of Bourgeois and Oancea has stringent technical hypotheses (expected to be alleviated using the polyfold theory of Hofer-Wysocki-Zehnder) related to transversality of holomorphic curve moduli spaces. However, in the present paper, we only use the ``compactness'' direction of their argument, and we neither ``count'' nor ``glue'' holomorphic curves in any situation where transversality cannot be achieved by perturbation of $H$ and $J$.

\begin{remark}
  Diogo's thesis \cite{diogo-thesis}, as well as forthcoming work of Diogo-Lisi expand on the idea of using the neck-stretching process to understand symplectic cohomology, in situations relevant to the present paper, namely, manifolds obtained as the complement of a symplectic hypersurface in a compact symplectic manifold.
\end{remark}

For symplectic cohomology, we have a Hamiltonian function $H$ on $M = \overline{U}' \cup_\Sigma \Sigma \times [0,\infty)$, which is zero in the interior and depends only on $\rho$ on the cylindrical end, and its time-dependent perturbation $K: S^1\times M \to \R$, and we have a time-dependent perturbation of the almost complex structure $J(t,x)$. We will consider a deformation of these structures parametrized by $t \in [1,\infty)$: 
\begin{enumerate}
\item As $t \to \infty$, we deform the almost complex structure by ``stretching the neck'' along a contact type hypersurface $\Sigma'$ constructed as follows. Recall that our manifold $M = \overline{U}' \cup_\Sigma \Sigma \times [0,\infty)$ is composed of the interior $\overline{U}'$ and the end $\Sigma \times [0,\infty)$. The Hamiltonian is $C^2$-small on the interior, and grows on the end. Take a hypersurface $\Sigma'$ that is obtained by pushing $\Sigma$ into $\overline{U}'$ by the compressing Liouville flow, so that $\Sigma'$ sits in the region where the Hamiltonian is $C^2$-small. This in particular means that all Hamiltonian orbits corresponding to Reeb orbits are ``outside'' the hypersurface we stretch along. 
\item As $t \to \infty$, we take the perturbation of the Hamiltonian to zero, making the Hamiltonian closer to the original autonomous, radial Hamiltonian. This in particular means that the Hamiltonian becomes zero in the stretching region. 
\item As $t \to \infty$, the almost complex structure $J(t,x)$ becomes time-independent and cylindrical in both the stretching region and the end.
\end{enumerate}
This deformation is chosen so that we can invoke the SFT compactness theorem \cite{behwz}, which is stated for pseudo-holomorphic curves without Hamiltonian perturbation.

We shall consider this neck-stretching for four different operations. Three of them use the cylinder as domain: The differential $d$, the BV operator $\Delta$, and the continuation map $\phi$. The fourth is the product $\theta^n * \theta^m$ of two generators corresponding to iterates of the \emph{same} periodic torus, where the domain is a pair of pants. We emphasize that we are \emph{not} presently considering the product of generators corresponding to \emph{different} periodic tori. We remark that, in the case of the continuation map, we have two different Hamiltonians; we assume that both are converging to Hamiltonians that vanish in the stretching region and which are radial on the end.

\begin{definition}
If $\gamma$ is a Hamiltonian orbit that is obtained from perturbation of the torus corresponding to the $m$-th iterate of a periodic torus of Reeb orbits, we call $m$ the \emph{multiplicity} of $\gamma$ and write $\Mult(\gamma) = m$. We call $\gamma$ \emph{primitive} if $\Mult(\gamma) = 1$.
\end{definition}

\begin{proposition}
\label{prop:neck-stretching}
Let $\cM(\gamma_-,\gamma_+)$ be a moduli space of solutions to Floer's equation on a cylinder, asymptotic to Hamiltonian orbits $\gamma_-$ (at the output) and $\gamma_+$ (at the input) in the cylindrical end of $M$, that is counted either by the differential, the BV operator, or a continuation map. If there are arbitrarily large values of the neck-stretching parameter $t \in [1,\infty)$ such that $\cM(\gamma_-,\gamma_+)$ is nonempty, then $\gamma_-$ and $\gamma_+$ correspond to iterates of the same periodic torus. Furthermore, the multiplicities satisfy $\Mult(\gamma_-) \leq \Mult(\gamma_+)$.
\end{proposition}

\begin{proof}
  Suppose there is a sequence of parameter values $t_i$ converging to infinity such that $\cM(\gamma_-,\gamma_+)$ is always nonempty. Then, by the SFT compactness theorem \cite{behwz}, after possibly passing to a subsequence, there is a sequence of Floer solutions for each of these parameter values that converges to a generalized holomorphic building. The bottom level of this building is a holomorphic curve in $M$ for the complex structure $J$, the the top level is a Floer solution for the radial Hamiltonian and cylindrical complex structure in the symplectization $\Sigma \times \R$, and any intermediate levels are holomorphic curves in $\Sigma \times \R$. The curves in these various levels are asymptotic to Reeb orbits in $\Sigma$. 

Since the original asymptotics $\gamma_-$ and $\gamma_+$ are in the end, there must be at least one symplectization level. Since the original curves all have the topology of a cylinder, each component of the curve in each level has at most one positive puncture. Thus Corollary \ref{cor:no-hol-cylinders-in-symplectization} implies that all such curves are trivial. This in particular applies to the top level, so $\gamma_-$ and $\gamma_+$ must correspond iterates of the same periodic torus. If this top level has any other negative punctures, they must also correspond to iterates of the same periodic torus. Since the total multiplicity of $\gamma_-$ and these negative punctures must equal the multiplicity of $\gamma_+$, we find that the multiplicity of $\gamma_-$ is less than or equal to that of $\gamma_+$.
\end{proof}

\begin{proposition}
\label{prop:neck-stretch-product}
Let $\gamma_{1,+}$ and $\gamma_{2,+}$ be Hamiltonian orbits corresponding to iterates of the same periodic torus. Let $\cM(\gamma_-,\gamma_{1,+},\gamma_{2,+})$ be the moduli space of pairs of pants used to compute the coefficient of $\gamma_-$ to the product of $\gamma_{1,+}$ and $\gamma_{2,+}$. If there are arbitrarily large values of the neck-stretching parameter $t \in [1,\infty)$ such that $\cM(\gamma_-,\gamma_{1,+},\gamma_{2,+})$ is nonempty, then $\gamma_-$ corresponds to an iterate of the same periodic torus as $\gamma_{1,+}$ and $\gamma_{2,+}$. Furthermore, the multiplicities satisfy $\Mult(\gamma_-) \leq \Mult(\gamma_{1,+}) + \Mult(\gamma_{2,+})$.
\end{proposition}

\begin{proof}
  This proof is analogous to the previous one but uses Corollary \ref{cor:no-hol-cylinders-in-symplectization} in the case where there is more than one positive puncture, but both positive punctures are asymptotic to orbits corresponding to the same periodic torus. 

The building that we obtain from SFT compactness has several levels, but we are only interested in the top one. Since, in the top level, both positive punctures are asymptotic to orbits lying on the same torus, namely $\gamma_{1,+}$ and $\gamma_{2,+}$, we find that by Corollary \ref{cor:no-hol-cylinders-in-symplectization} that level is trivial. Thus $\gamma_-$, which also lives in the top level, must correspond to the same torus as $\gamma_{1,+}$ and $\gamma_{2,+}$. Furthermore, any other negative punctures on that level must also correspond to the same torus, from which we obtain $\Mult(\gamma_-) \leq \Mult(\gamma_{1,+}) + \Mult(\gamma_{2,+})$.
\end{proof}

\begin{remark}
  \label{rem:interior-gen-product}
  The preceding proposition has an extension to the case where either $\gamma_{1,+}$ or $\gamma_{2,+}$ is an interior generator, where we interpret interior generators as having multiplicity zero. Without loss of generality assume that $\gamma_{2,+}$ is an interior generator. We contend that $\Mult(\gamma_{-}) \leq \Mult(\gamma_{1,+})$, for any possible $\gamma_{-}$ appearing in the product. We apply the neck-stretching as in the preceding propositions, and obtain a limiting Morse-Bott building. In this case, the top level has only one positive puncture corresponding to $\gamma_{1,+}$ (the other input $\gamma_{2,+}$ remains inside the interior), and the argument proceeds as in the proof of Proposition \ref{prop:neck-stretching}.
\end{remark}

\subsection{The complex torus as a local model}
\label{sec:torus-local}

In this section we will revisit from our current standpoint the case of the complex torus $U = (\C^\times)^2$ that was discussed in Section \ref{sec:complex-torus}. We know there is a Viterbo isomorphism $SH^*((\C^\times)^2) \cong H_{2-*}(\cL T^2)$ that intertwines the product and BV operator. We shall give a more explicit description of how this isomorphism and the operations on $SH^*((\C^\times)^2)$ behave at the chain level (or at least on the $E_1$ page of a Morse-Bott spectral sequence), since this will be used as a local model for computations in the general case.

If we compactify $(\C^\times)^2$ to a toric variety $Y$, we get a log Calabi-Yau pair $(Y,D)$ where $D$ is the toric boundary divisor. The constructions of Sections \ref{sec:construction} and \ref{sec:affine-man} lead to the following picture. The contact boundary $\Sigma$ is a trivial fibration $T^2 \times S^1 \to S^1$, so the monodromy is trivial. In this case, the affine manifold $U^\trop$ is just $\R^2$. The Liouville class $A(s)$ is a locally convex loop in $H^1(T^2;\R)$. Thus $A'(s)$ is also a loop that rotates monotonically. For each parameter value $s$ such that $A'(s)$ vanishes on some integral vector in $H_1(T^2;\Z)$, the Reeb flow has a periodic torus. Thus there is one periodic torus for each rational direction in $H_1(T^2;\Z)$ (recall that a direction is called rational if the line in that direction contains an integral vector).

Taking a radial quadratic Hamiltonian $H$, we find that each periodic torus creates a family of tori of Hamiltonian orbits corresponding to all the iterates of simple Reeb orbits on a single torus. Each torus of Hamiltonian orbits has a corresponding class in $H_1((\C^\times)^2;\Z)$, namely the class represented by a single orbit in that torus. Observe each nonzero class in $H_1((\C^\times)^2)$ is represented by exactly one such torus of Hamiltonian orbits; this is a key consequence of our assumption that $A(s)$ is locally convex. 

\begin{proposition}
\label{prop:c-star-spectral-sequence}
  Taking a small perturbation of the radial Hamiltonian $H$ and close to the neck-stretching limit, we obtain a Morse-Bott spectral sequence converging to $SH^*((\C^\times)^2)$ whose $E_1$ term is
  \begin{equation}
    E_1 = H^*(T^2) \oplus \left(\bigoplus_{p \in H_1((\C^\times)^2;\Z)\setminus \{0\} }H^*(T_p)\right)
  \end{equation}
  where $T_p$ denotes the torus of Hamiltonian orbits representing the class $p$. This spectral sequence degenerates at $E_1$. Furthermore, the Viterbo isomorphism carries $H^0(T^2_p)$ isomorphically onto the component of $H_2(\cL T^2)$ supported on loops in the class $p$.
\end{proposition}

\begin{proof}
  After a small non-autonomous perturbation is added to the Hamiltonian, each torus of periodic orbits breaks into several non-degenerate orbits. We assume that the non-autonomous term is supported in a union of small pairwise disjoint neighborhoods of the periodic tori; this is possible since the tori are not arbitrarily close to one another. 

First we claim that the orbits created by perturbation of the Hamiltonian stay near the torus $T_p$. This is because outside a small neighborhood of the $T_p$, the flow is unperturbed, and is given by translation along the fibers of $\pi \times \text{id}: \Sigma \times \R \to S^1 \times \R$. Thus there are no trajectories connecting the unperturbed region to the perturbed region, and we find that all trajectories starting in the perturbed region remain within it.

Next we claim that the Floer trajectories connecting two orbits $\gamma_+$ and $\gamma_-$ coming from the perturbation of $T_p$ lie in a neighborhood of $T_p$. We use the idea from the proof of Proposition \ref{prop:neck-stretching}. There we showed that in the neck-stretching limit, a sequence of cylinders under consideration has a subsequence converging to a Morse-Bott inhomogeneous pseudo-holomorphic building such that all levels mapping to the symplectization are trivial. Since $\Mult(\gamma_+) = \Mult(\gamma_-)$ by assumption, the only possibility is that the building has a single level consisting of a trivial cylinder, together with two Morse trajectories on the torus $T_p$, one at each end of the cylinder. The limit cylinder has energy exactly zero, and its image is some Reeb orbit contained in $T_p$. This shows that the original sequence of cylinders has a subsequence that converges uniformly to a map whose image is contained in $T_p$. Given $\epsilon > 0$, if there were a sequence of cylinders that is not eventually contained in an $\epsilon$-neighborhood of $T_p$, this condition would be violated, so we conclude that near the neck-stretching limit the cylinders connecting $\gamma_+$ and $\gamma_-$ all lie near $T_p$.

As is standard in the Morse-Bott situation, the cohomology with respect to differential that counts the low-energy cylinders remaining close to the torus recovers the cohomology of the original torus of orbits. In the interior, we take the Hamiltonian to be $C^2$-small, so we just recover the ordinary cohomology. This is the $E_1$ page written above.

The fact that the spectral sequence must degenerate at $E_1$ is \emph{in this particular case} a consequence of the Viterbo isomorphism. For any manifold $M$, both $SH^*(M)$ and $H_{n-*}(\cL M)$ have decompositions according to homology classes of loops, and this decomposition is respected by all operations and by the Viterbo isomorphism. Take a class $p \in H_1((\C^\times)^2;\Z)$. The $p$-summand of the $E_1$ page is four-dimensional, which is the same as the dimension of the $p$-summand of the loop space homology. Therefore there can be no further differentials. 

The isomorphism of $H^0(T_p)$ with the component of $H_2(\cL T^2)$ supported on loops in the class $p$ now follows from degree considerations.
\end{proof}

\begin{remark}
  The contribution of each periodic torus to the $E_1$ page, here and below (Proposition \ref{prop:perturbation-indices}) is an example of the local Floer homology as studied by McLean \cite{mclean-local-floer}.
\end{remark}

Now, since the BV operator and the product respect the action filtration, they are compatible with the spectral sequence, and give rise to operations on the $E_1$ page. We shall now describe these operations explicitly.

\begin{proposition}
\label{prop:cstar-bv-product}
  On the $E_1$ page, the BV operator preserves the subspaces $H^*(T_p)$, and it is given by counting cylinders that remain near the periodic tori of $H$. The product is homogeneous with respect to the grading by $p \in H_1((\C^\times)^2;\Z)$, and for fixed primitive class $p$ and $r, r', r'' > 0$ such that $r'' = r + r'$, the component of the product
  \begin{equation}
    H^*(T_{rp}) \otimes H^*(T_{r'p}) \to H^*(T_{r''p})
  \end{equation}
  is given by counting pairs of pants contained in a small neighborhood of the trivial cylinder over the torus of Reeb orbits in the class $p$. 
\end{proposition}

\begin{proof}
  The fact that the operations are homogeneous with respect to the $H_1((\C^\times)^2;\Z)$ grading is obvious from the fact that they count maps of Riemann surfaces.

In the proof of Proposition \ref{prop:c-star-spectral-sequence}, we saw that the cylinders contributing to the differential remain close to the periodic torus. The same argument (based on Proposition \ref{prop:neck-stretching}) applies to the cylinders contributing to the BV operator.

For the statement about the products of elements corresponding to iterates of the same primitive class, we adapt the argument using Proposition \ref{prop:neck-stretch-product}. Once again, in the neck-stretching limit any pairs of pants contributing to the product must limit to a Morse-Bott inhomogeneous pseudo-holomorphic building, all of whose levels mapping to the symplectization are trivial, and all Reeb orbits involved correspond to iterates of $T_p$. Since no combination of such orbits is homologically trivial in $(\C^\times)^2$, the building can have no level mapping to the interior. Thus the pair of pants in question must eventually be contained in a neighborhood of the trivial cylinder over $T_p$.

\end{proof}

At various points in the following arguments we will want to argue that the lowest-energy contribution to an operation involving certain generators is the corresponding operation in the case of $(\C^\times)^2$. The idea behind this is as follows. Let $(\Sigma,\alpha)$ denote the contact boundary of our $U$, and let $(\Sigma_0,\alpha_0)$ denote the contact boundary in the case of $(\C^\times)^2$ considered above. Both of these manifolds are torus bundles over the circle; let $\tilde{\Sigma}$ and $\tilde{\Sigma}_0$ denote the pullbacks over the universal covering $\R \to S^1$. 

\begin{lemma}
\label{lem:contact-lift}
Both $\tilde{\Sigma}$ and $\tilde{\Sigma}_0$ are diffeomorphic to $\R \times T^2$, and there is a reparametrization of the base $\R$ that makes them isomorphic as contact manifolds, and such that the contact forms differ by a scaling factor that depends only on the base coordinate $s \in \R$. Any genus zero pseudo-holomorphic map in the symplectization of $\Sigma$ or $\Sigma_0$ will lift to the respective coverings. 
\end{lemma}

\begin{proof}
 The reason why $\tilde{\Sigma}$ and $\tilde{\Sigma}_0$ are isomorphic is that the contact structure is in both cases determined by the Liouville class $A(s)$ up to scale, and in both cases this Liouville class $A(s) : \R \to H^1(T^2;\R)$ is a (locally convex) path that winds infinitely many times around the origin.

For a genus zero curve, we see that since all orbits in question project to contractible loops in $S^1$, the image of the fundamental group of the domain in the fundamental group of $S^1$ is trivial. Thus it can be lifted to the covering.
\end{proof}

The upshot of this lemma is that, when we are comparing the Floer theory of $U$ to that of $(\C^\times)^2$, the contributions to the operations that come from curves living entirely in the cylindrical end must correspond to one another, since both can be pulled back to $\tilde{\Sigma}\times \R \cong \tilde{\Sigma}_0 \times \R$. On the other hand, there may be some difference coming from the presence of curves that leave the cylindrical end and enter the interior of the manifold. However, these curves will have higher energy than the curves contained in the end, and so they must contributed to higher action terms in the output.

\subsection{Generators of symplectic cohomology}
\label{sec:generators}

We continue the analysis of section \ref{sec:construction}, picking up with the output of Step 5, a Liouville structure on $\overline{U}'$ such that the boundary contact form is toroidally symmetric and the Liouville class is locally convex.

\subsubsection{Periodic orbits}
\label{sec:reeb-orbits}

Recalling section \ref{sec:reeb-vf}, we have that the Reeb vector field $R$ is vertical with respect to the torus fibration, and it is locally torus symmetric. Thus the Reeb flow translates each torus by some amount. If $R = R_1\partial_{\theta_1} + R_2\partial_{\theta_2}$ is the Reeb vector field and $\alpha = f(s)\,d\theta_1+ g(s)\,d\theta_2$ is the contact form, recall that $f'R_1+ g'R_2 = 0$ determines the direction of the Reeb field. Thus the torus $\pi^{-1}(s)$ is periodic whenever
\begin{equation}
  A'(s)^\perp = \left\{\varrho\in H_1(\pi^{-1}(s),\R) \mid \int_\varrho A'(s) = 0\right\} \subset H^1(\pi^{-1}(s),\R)
\end{equation}
is a rational subspace, meaning that it contains an integral vector. An equivalent condition for the torus $\pi^{-1}(s)$ to be periodic is that $R_1(s)$ and $R_2(s)$ satisfy a linear equation with rational coefficients. Each such periodic torus leads to a $T^2$ of simple (not multiply covered) Reeb orbits, and an $\N^+ \times T^2$ family of periodic Reeb orbits, where $\N^+$ keeps track of multiplicity.

Recall from section \ref{sec:affine-charts} that associated to each node $i$ of $D$ there is an integral cone $Q_i \subset H_1(V_i\setminus D,\Z)$, where $V_i$ is a neighborhood of the node $i$. We denote by $Q_{i,\R}$ the real version of this cone.

\begin{lemma}
  \label{lem:rational-directions}
For each $i$, there is an interval $s \in [s_{1,i},s_{2,i}]$ of values of the $s$ parameter, such that the space $A'(s)^\perp$ passes through each rational direction in $Q_{i,\R}$ exactly once. In this interval, the Reeb orbits represent homology classes in $Q_i$. Hence each primitive integral point in $Q_i$ corresponds to a periodic torus, and each integral point corresponds to a periodic torus and a particular multiplicity.
\end{lemma}

\begin{proof}
  This is a recasting of the description of the Reeb flow from Step 5
  above (\ref{sec:step5}) in the language of the affine manifold from
  Section \ref{sec:affine-man}.  At some point near each divisor, the Reeb flow coincides with the circle action rotating the normal circle to the divisor, and hence the
  line $A'(s)^\perp$ is spanned by the class of the normal circle to
  the divisor. For the divisor connecting node $i$ to node $i+1$, define $s_{2,i} = s_{1,i+1}$ to be the $s$-value where this occurs. 

A local analysis shows that the Reeb flow winds negatively around the divisor, so at $s_{2,i} = s_{1,i+1}$, the simple Reeb orbit represents the class $\Gamma_{z_i} \sim \Gamma_{w_{i+1}}$ considered in Section \ref{sec:affine-man}. These are the rays in the affine
  manifold. Between $s_{1,i}$ and $s_{2,i}$, we have the local model described in \ref{sec:step5}, where near each node the direction of translation on the
  torus fiber rotates between the normal circle directions of the two
  divisors. Thus the Reeb orbits appearing for $s \in [s_{1,i},s_{2,i}]$ are homologous to non-negative linear combinations of the classes $\Gamma_{w_i}$ and $\Gamma_{z_i}$, so these classes lie in $Q_i$. The local convexity condition means that the subspace
  $A'(s)^\perp$ rotates monotonically, so rational
  directions are never repeated.
\end{proof}

Recall from section \ref{sec:conventions} that when we complete $\overline{U}'$ along the boundary $\Sigma$, we get a manifold $M = \overline{U}' \cup_{\Sigma} \Sigma \times [0,\infty)$ with a cylindrical end. The Hamiltonian $H = (e^\rho)^2/2$ has periodic orbits that correspond to Reeb orbits, so that a periodic Reeb orbit of period $T$ corresponds to a periodic orbit of $X_H$ of period $1$ sitting in the hypersurface $e^\rho = T$. Thus we have
\begin{corollary}

  \label{cor:ham-orbits}
  The time-one-periodic orbits of the Hamiltonian $H = (e^\rho)^2/2$ in the cylindrical end form a disjoint union of tori, and these tori are in bijective correspondence with the integral points of the affine manifold $U^\trop(\Z)$.
\end{corollary}

We index the tori of periodic orbits by pairs $p = (s,r)$, where $s \in S$ is a point where the Reeb vector is rational, and $r \in \N^\times$ is a positive integer giving the multiplicity of iteration of the orbit. The torus of periodic orbits is called $T_{s,r}$ or $T_p$. See Figure \ref{fig:integral-points}.
\begin{figure}
  \centering
  \includegraphics[width=2in]{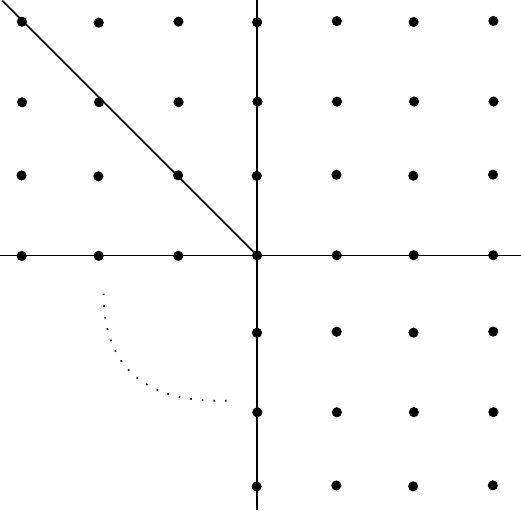}
  \caption{Integral points in the affine manifold. The axes and the diagonal ray bound quadrants corresponding to nodes.}
  \label{fig:integral-points}
\end{figure}

\subsubsection{Maslov and Conley--Zehnder indices}
\label{sec:cz-index}

Recall from Lemma \ref{lem:volume-form} that $U$ carries a holomorphic volume form $\Omega$ that is determined up to a constant multiple by the condition that it has simple poles on $D$. In the following, we always use this $\Omega$ to trivialize the canonical bundle of $U$.

\begin{proposition}
  \label{prop:maslov-tori}
  The Lagrangian tori $L \subset \Sigma \subset \overline{U}'$ near the boundary divisor (cf.~Proposition \ref{prop:stability-liouville-class}) have vanishing Maslov class in $Y \setminus D$. 
\end{proposition}

\begin{proof}

 First assume that $L$ is near a node, so that in a local coordinate chart $(z_1,z_2)$ near a node, $\Omega$ and $L$ have the form
  \begin{equation}
    \Omega = f(z)\frac{dz_1}{z_1}\wedge \frac{dz_2}{z_2}, \quad L = \{|z_1|=\epsilon_1,|z_2|=\epsilon_2\}
  \end{equation}
  where $f(z)$ is holomorphic and non-vanishing along $D$. Recall that the phase function $\arg(\Omega|_L): L \to \U(1)$ is defined to be the phase of $(\Omega|_L)/\nu$, where $\nu$ is a real volume form on $L$. In order to show that $L$ has vanishing Maslov class, it suffices to show that the phase function admits an $\R$-valued lift.

Take $\nu = d\theta_1\wedge d\theta_2$ as the volume form on $L$.
The phase of $L$ at $(z_1,z_2)=(\epsilon_1e^{i\theta_1},\epsilon_2e^{i\theta_2})$ is that of
  \begin{equation}
    \Omega|_L = f(z_1,z_2)(i\,d\theta_1)\wedge(i\,d\theta_2) = -f(z_1,z_2)\,\nu
  \end{equation}
  Since the function $f$ is non-vanishing for small $(z_1,z_2)$, the argument function $\arg (f|_L): L \to \U(1)$ admits a lift to $\R$, and so does $\arg (\Omega|_L)$.


Since all the tori $L$ are Lagrangian isotopic to a torus near a node, they all also have vanishing Maslov class.
\end{proof}

Now we will formulate the analogous spectral sequence to the one described in Proposition \ref{prop:c-star-spectral-sequence}. 

\begin{proposition}
  \label{prop:perturbation-indices}
  A small perturbation of the Hamiltonian breaks the torus $T_p$ of periodic orbits into several non-degenerate orbits. There are at least $\binom{2}{k}$ orbits of Conley--Zehnder index $k$, for $k = 0, 1, 2$. In the neck-stretching limit, the Floer trajectories connecting these orbits stay in a neighborhood of the torus and yield a differential $d_0$ whose cohomology can be identified with the ordinary cohomology of the torus $H^*(T_p)$. Summing over all $p$, the cohomology of $d_0$ is the $E_1$-page of a spectral sequence converging to $SH^*(U)$.
\begin{equation}
  E_1^* = H^*(U) \oplus \left(\bigoplus_{p \in U^\trop(\Z)\setminus\{0\}} H^*(T_p)\right)
\end{equation}
\end{proposition}

\begin{proof}
  All of the statements to be proved are local near the Reeb periodic torus $L$ underlying the torus of periodic orbits $T_p$. Just as in the case of the complex torus $(\C^\times)^2$, we assume that the non-autonomous term is supported in a union of small pairwise disjoint neighborhoods of the periodic tori; this is possible since the tori are not arbitrarily close to one another. In fact, locally each torus of periodic orbits looks like one in the case of $(\C^\times)^2$, so we can use that case as a local model to understand the perturbation, and in particular the degrees of the generators. Just as in the proof of Proposition \ref{prop:c-star-spectral-sequence}, we can arrange our perturbations so that the orbits created by perturbation of the Hamiltonian stay near the torus $T_p$, and so that the Floer trajectories connecting two orbits $\gamma_+$ and $\gamma_-$ coming from the perturbation of $T_p$ lie in a neighborhood of $T_p$.

The statement about Conley-Zehnder indices necessarily holds up to an overall shift, since the differences between the Conley-Zehnder indices does not depend on the choice of trivialization of the canonical bundle. To fix the overall shift, we argue by comparison with the $(\C^\times)^2$ case.

To see that the homotopy classes of trivializations of the canonical bundle match, note that in both $U$ and $(\C^\times)^2$, the Lagrangian torus $L$ on which the orbits lie has Maslov class zero. Since $L$ carries all the topology of a neighborhood of $L$, we find that the trivialization of the canonical bundle of this neighborhood is determined by the Maslov class of $L$. Since these Maslov classes both vanish, the corresponding trivializations match.

Since our Liouville class is locally convex, we find that there is a Liouville structure with locally convex Liouville class in the $(\C^\times)^2$ case containing a periodic torus whose neighborhood is isomorphic to a neighborhood of our given torus in $U$. Since, in the $(\C^\times)^2$ case, local convexity implies that the spectral sequence degenerates at $E_1$, the local contribution to the $E_1$ page from this torus must match with the cohomology in the $(\C^\times)^2$ case, showing that the degrees must match as stated.
\end{proof}

\begin{remark}
  If we did not assume that the Liouville class were locally convex, we would find that some periodic tori would contribute shifted copies of $H^*(T^2)$ to the $E_1$ page. This can already be seen in the $(\C^\times)^2$ case, where if the Liouville class is not locally convex, there are multiple periodic tori containing homologous Reeb orbits, and there must be higher differentials connecting them in order for the Viterbo isomorphism to hold.
\end{remark}

The higher differentials in the spectral sequence count inhomogeneous pseudo-holomorphic curves connecting different critical manifolds, and these are analyzed in Section \ref{sec:compute-differential}. Define $\theta_0 = 1 \in H^0(U)$ and $\theta_p = PD[T_p] \in H^0(T_p)$ for $p \in U^\trop(\Z)\setminus\{0\}$. We now reformulate Theorem \ref{thm:main}.
\begin{theorem}
  \label{thm:degeneration}
  The degree 0 part of the spectral sequence degenerates at $E_1$, so 
  \begin{equation}
    SH^0(U) \cong E_1^0 = \Span\{\theta_p \mid p \in U^\trop(\Z)\}
  \end{equation}
\end{theorem}

The preceding results are formulated in terms of a Hamiltonian that is quadratic at infinity. If we use a Hamiltonian $H^m$ that is linear at infinity, analogous results hold for the Floer homology $HF^*(H^m)$. The difference is that only the periodic tori whose Reeb length $\ell_p = \int_{\gamma_p} \alpha$ satisfies $\ell_p < m$ contribute to $HF^*(H^m)$. There is a spectral sequence converging to $HF^*(H^m)$ whose $E_1$ page is
\begin{equation}
  E_1^{*} = H^*(U) \oplus \left(\bigoplus_{p, \ell_p < m} H^*(T_p) \right)
\end{equation}
\begin{theorem}
\label{thm:degeneration-linear}
  The degree 0 part of the spectral sequence for $HF^*(H^m)$ degenerates at $E_1$, so
  \begin{equation}
    HF^0(H^m) \cong E_1^0 = \Span\{\theta_p \mid p \in U^\trop(\Z), \ell_p < m\}
  \end{equation}
\end{theorem}

For the rest of the computation, we will work with the linear Hamiltonians. This is because the computation of $HF^0(H^m)$ requires us to choose perturbations with favorable properties, and it is not necessarily possible to choose these perturbations in a way that works for the entire cochain complex associated to the quadratic Hamiltonian. Furthermore, as the asymptotic slope $m$ increases, new generators are added to the complex, and we may need to modify the perturbations in order to accommodate these new generators into the method used to compute the differential.

This choice to work with linear Hamiltonians then requires us to analyze the continuation maps $HF^*(H^m) \to HF^*(H^{m'})$ for $m \leq m'$. We show in the last subsection that these maps are injective on $HF^0$, and therefore that, in the colimit used to construct $SH^0(U)$ from $HF^0(H^m)$, all classes at each stage survive to the limit, finishing the proof of Theorem \ref{thm:degeneration}.

\subsection{Computation of differentials}
\label{sec:compute-differential}

The purpose of this section is to prove that the higher differentials in the spectral sequence vanish on any degree zero element $\theta \in E_1^0$.

Let $\theta$ be a degree zero generator corresponding to the fundamental class of a periodic torus of Reeb orbits.  A boundary generator is one coming from the periodic Reeb orbits in the contact boundary, while an interior generator is one coming from the $H^*(U)$ component of the $E_1$ page. 

The strategy for proving $d\theta = 0$ has several ingredients. We use the way that the differential $d$ interacts with the BV operator $\Delta$ and the product, namely, that $\Delta$ is a chain map and $d$ is a derivation of the product. We also combine this with the nonexistence results for Floer cylinders in section \ref{sec:neck-stretching}. All of these results in this section therefore require that the structures be sufficiently close to the neck-stretching limit. Furthermore, we use some aspects of the computation of symplectic cohomology for the complex torus $(\C^\times)^2$. We described the BV operator and product in section \ref{sec:complex-torus}. The general principle here is that the structures of the operations near a single periodic torus are the same as those in the case of $(\C^\times)^2$, as long as we only consider the contributions of inhomogeneous pseudo-holomorphic curves that stay near that torus (the ``low-energy'' contributions). Of course there may be other contributions to these operations coming from holomorphic curves that extend outside of this neighborhood, but they will have higher energy. 



For $p \in U^\trop(\Z) \setminus 0$, we denote by $T_p$ the corresponding torus of orbits for the unperturbed Hamiltonian $H$. We denote by $CF^*(H)_p$ and $E_1^*(H)_p$ the corresponding components of the Floer cochain complex and the $E_1$ page respectively. We denote by $d_0$ the low-energy component of the differential; recall that $E_1(H)_p$ is the cohomology of $CF(H)_p$ with respect to $d_0$. 

\begin{remark}
  Since the Morse-Bott manifolds of orbits $T_p$ are all tori, and the torus admits a perfect Morse function, one could actually assume that the perturbations are made in such a way that $d_0$ vanishes on $CF^*(H)_p$ for $p \in U^\trop(\Z) \setminus 0$. In the arguments that follow, we have not assumed this so as to make the structure of the argument clearer.
\end{remark}

\begin{proposition}
  \label{prop:prim-boundary-closed}
  Suppose $p$ is primitive. For any $\zeta \in CF^*(H)_p$, we have 
  \begin{equation}
    d\zeta = d_0\zeta + (\text{interior generators}).
  \end{equation}
  
\end{proposition}

\begin{proof}
  We must show the matrix element of the differential connecting $\zeta$ to any other boundary generator $\beta$ not counted in $d_0\zeta$ is zero. This matrix element is a count of inhomogeneous pseudo-holomorphic cylinders connecting $\zeta$ at the positive puncture to $\beta$ at the negative puncture. By Proposition \ref{prop:neck-stretching}, this moduli space of cylinders will eventually be empty in the neck-stretching limit unless $\beta$ corresponds to an iterate of the same periodic torus and $\Mult(\beta) \leq  \Mult(\zeta)$. Since  $\Mult(\zeta) = 1$, we find that $\beta$ must correspond to the same torus as $\zeta$. Then, since $\zeta$ and $\beta$ have approximately the same action, any differentials connecting them would have low energy, and are counted in $d_0\zeta$.
\end{proof}


\begin{proposition}
  \label{prop:bv-exactness}
  Suppose $p$ is primitive. Let $\theta \in E_1^0(H)_p$. Then there is a lift $\tilde{\theta} \in CF^0(H)_p$ and a element $\eta \in CF^1(H)_p$ such that $d_0\eta = 0$ and 
  \begin{equation}
    \Delta(\eta) = \tilde{\theta} + (\text{interior generators}).
  \end{equation}
\end{proposition}

\begin{proof}
  We know that for any $\eta \in CF^1(H)_p$, $\Delta(\eta)$ will be of the form $\beta + (\text{interior generators})$ for some $\beta \in CF^0(H)_p$. Indeed, Proposition \ref{prop:neck-stretching} says that in the neck-stretching limit all cylinders contributing to $\Delta(\eta)$ end on generators corresponding to iterates of the same periodic torus as $\eta$, of less or equal multiplicity. Since $\eta$ is primitive, the only possible terms in $\Delta(\eta)$ are $\beta \in CF^0(H)_p$ and interior generators. 

It remains to show that we can arrange things so that $\eta$ is $d_0$-closed and $\Delta(\eta)$ contains $\tilde{\theta}$ that projects to $\theta \in E_1^0(H)_p$.

Let $\theta \in E_1^0(H)_p$ be given. Since that space is one dimensional, it will suffice to prove the proposition for the basis element, which is the Poincar\'{e} dual of the fundamental class of $T_p$.

  Let $T_p \subset \cL M$ be the torus of parametrized orbits corresponding to $p$. There is a circle action on $T_p$ given by rotating the parametrization of the the orbit. Denoting the class of the orbit of this circle action by $a \in H_1(T_p;\Z)$, let $b \in H_1(T_p;\Z)$ be a class that is dual to $a$. Let $\bar{\eta} \in H^1(T_p) \subset E_1^1(H)_p$ denote the Poincar\'{e} dual of $b$, and let $\eta \in CF^1(H)_p$ denote a lift of $\bar{\eta}$. Note that $\eta$ necessarily satisfies $d_0\eta = 0$.
  
We know that $\Delta(\eta) = \beta + (\text{interior generators})$ for some $\beta \in CF^0(H)_p$, and we need to show that this $\beta$ is a lift of $\theta$. By the same argument as in the proof of Proposition \ref{prop:cstar-bv-product}, all of the curves that contribute to the $\beta$ term remain within a small neighborhood $V$ of the torus $T_p$. 

Now we can argue by comparison with the $(\C^\times)^2$ case. We can set up the computation in the $(\C^\times)^2$ case so that it contains a neighborhood where all of the structures match those in $V$. The BV operator in this case was described in Section \ref{sec:complex-torus}. The Viterbo isomorphism and the fact that the spectral sequence degenerates at $E_1$ determine this operation: $\Delta(\bar{\eta})$ is the class Poincar\'{e} dual to the class swept out by $b$ under the circle action. This is the fundamental class of the torus, that is, $\theta$. Lifting to the chain level, we find that $\Delta(\eta) = \tilde{\theta}$, where $\tilde{\theta}$ is a lift of $\theta$.

\end{proof}

\begin{proposition}
  \label{prop:prim-closed}
  Suppose $p$ is primitive. Let $\theta \in E_1^0(H)_p$ be a primitive degree zero generator. Then there is a lift $\tilde{\theta} \in CF^0(H)_p$ such that $d\tilde{\theta} = d(\text{interior generators})$. Therefore $\theta$ is closed for all higher differentials in the spectral sequence.
\end{proposition}

\begin{proof}
  Given $\theta$, Proposition \ref{prop:bv-exactness} gives us $\eta$ and $\tilde{\theta}$ such that $\Delta(\eta) = \tilde{\theta} + x$, where $x$ is an interior cochain. 
  Using the fact that $\Delta \circ d + d \circ \Delta = 0$ (as $\Delta$ is an odd chain map), we obtain
  \begin{equation}
    -\Delta(d\eta) = d\Delta(\eta) = d\tilde{\theta} + dx.
  \end{equation}
  On the other hand, by Proposition \ref{prop:prim-boundary-closed}, $d\eta = d_0\eta + (\text{interior generators})$; since $d_0 \eta = 0$, we find $d\eta = y$, where $y$ is an interior cochain. Thus we obtain
  \begin{equation}
    -\Delta(y) = d\tilde{\theta} + dx.
  \end{equation}
  Now we use the fact that $\Delta$ vanishes on the cohomology of the interior. At chain level, this says that $\Delta(y) = dz$ for some interior cochain $z$.
  Thus
  \begin{equation}
    -dz =  d\tilde{\theta} + dx
  \end{equation}
  and so $d\tilde{\theta} = -d(x+z)$ as claimed.

  For the last claim, recall that the higher differentials in the spectral sequence are just the original differential restricted to certain sub-quotients of the complex. But at any page after the first, $d(x+z) = d_0(x+z)$ is identified with zero.
\end{proof}

\begin{remark}
  There is another class in the image of the BV operator, call it $\zeta \in H^1(T_p)$, which is the image of the Poincar\'{e} dual of the point class in $H^2(T_p)$. Geometrically it represents a single orbit of the circle action. The conclusion of proposition \ref{prop:prim-closed} also holds for the class $\zeta$ corresponding to a torus of primitive orbits.
\end{remark}


\begin{proposition}
  \label{prop:prim-generate}
  Fix $p \in U^\trop(\Z)$ primitive. For $r > 0$, denote by $\theta_{rp} \in E_1^0(H'')_{rp}$ the Poincar\'{e} dual of the fundamental class of $T_{rp}$. Then $\theta_{rp}$ admits a lift $\tilde{\theta}_{rp} \in CF^0(H'')_p$, constructed by induction on $r$, such that, if $*$ denotes the product $CF^0(H) \otimes CF^0(H') \to CF^0(H'')$ (where the asymptotic slope of $H''$ is at least the sum of the asymptotic slopes of $H$ and $H'$), we have
  \begin{equation}
    \tilde{\theta}_{rp} = \tilde{\theta}_p * \tilde{\theta}_{(r-1)p} + X
  \end{equation}
  where $X \in CF^0(H'')_0 \oplus \bigoplus_{s=1}^{r-1} CF^0(H'')_{sp}$ is a sum of interior generators and generators corresponding to iterates of $p$ of multiplicity $< r$, and $\tilde{\theta}_{p}$ and $\tilde{\theta}_{(r-1)p}$ are the lifts constructed earlier in the induction.
\end{proposition}

\begin{proof}
  For fixed $p$, we construct the lifts by induction on $r$. For $r = 1$, we take the lift provided by Proposition \ref{prop:prim-closed}. Suppose that $\tilde{\theta}_{sp}$ has been constructed for all $s < r$. We now consider what terms may appear in the product $\tilde{\theta}_p * \tilde{\theta}_{(r-1)p}$. It follows immediately from Proposition \ref{prop:neck-stretch-product} that in the neck-stretching limit $\tilde{\theta}_p * \tilde{\theta}_{(r-1)p}$ can only contain terms from $CF^*(H'')_0 \oplus \bigoplus_{s=1}^r CF^0(H'')_{rp}$, that is interior generators and generators corresponding to iterates of $p$ up to multiplicity $r$. We define $\tilde{\theta}_{rp}$ to be the component of $\tilde{\theta}_p * \tilde{\theta}_{(r-1)p}$ that sits in $CF^0(H'')_{rp}$. It remains to show that $\tilde{\theta}_{rp}$ is a lift of $\theta_{rp} \in E_1^0(H'')$.

For this we use comparison with the $(\C^\times)^2$ case. Proposition \ref{prop:neck-stretch-product} shows us that all curves contributing to the $rp$-component of $\tilde{\theta}_p * \tilde{\theta}_{(r-1)p}$ remain in the end; in fact, since these curves approach trivial cylinders, the computation localizes to a neighborhood of $T \times \R$ in the symplectization, where $T \subset \Sigma$ is the periodic torus corresponding to $p$. We may set up the corresponding computation in the $(\C^\times)^2$ case so that it contains a subset where all of the structures match those in this neighborhood. Then, since the spectral sequence degenerates at $E_1$ in the $(\C^\times)^2$ case, the Viterbo isomorphism tells us that $\theta_p* \theta_{(r-1)p} = \theta_{rp}$ in the $(\C^\times)^2$ case. This means that $\tilde{\theta}_p * \tilde{\theta}_{(r-1)p}$ is a lift of $\theta_{rp}$ in the $(\C^\times)^2$ case. Transporting this back to $U$, we have the result.
\end{proof}

\begin{proposition}
\label{prop:all-closed}
  Let $\theta_{rp} \in E_1^0(H)_{rp}$, where $p$ is primitive and $r > 0$. Then there is a lift $\tilde{\theta}_{rp}$ such that $d\tilde{\theta}_{rp} = dX$, where $X$ is a sum of interior generators and generators corresponding to iterates of $p$ of multiplicity $< r$. Therefore $\theta_{rp}$ is closed for all higher differentials in the spectral sequence.
\end{proposition}

\begin{proof}
  We take $\tilde{\theta}_{rp}$ as provided by Proposition \ref{prop:prim-generate}. We proceed by induction on $r$. For $r = 1$, this is Proposition \ref{prop:prim-closed}. Suppose that for all $s < r$, $\tilde{\theta}_{sp}$ has the property that $d\tilde{\theta}_{sp} = dX_s$, where $X_s$ is a sum of interior generators and generators of multiplicity $< s$. Then we have
  \begin{equation}
    \tilde{\theta}_{rp} = \tilde{\theta}_p * \tilde{\theta}_{(r-1)p} + Y
  \end{equation}
  where all terms in $Y$ have multiplicity $< r$. Apply $d$ to this equation and use the fact that it is a derivation of the product to obtain
  \begin{equation}
    d\tilde{\theta}_{rp} = dX_1 * \tilde{\theta}_{(r-1)p} + \tilde{\theta}_p * dX_{r-1} + dY
  \end{equation}
  We see that all terms on the right-hand side correspond either to interior generators or generators of multiplicity $< r$, using Propositions \ref{prop:neck-stretching} and \ref{prop:neck-stretch-product} and Remark \ref{rem:interior-gen-product}.

Now we use the fact that multiplicity corresponds to action, namely that higher iterates have more negative action. Thus what we have shown is that $\tilde{\theta}_{rp}$ satisfies $d\tilde{\theta}_{rp} = d(\text{higher action terms})$. This means that, by the time we reach the page where the differential of $\theta_{rp}$ might land, its value has already been killed by the higher action terms.
\end{proof}

Since the elements $\theta_{rp}$ for $p$ primitive and $r > 0$, together with $1 \in H^0(U)$, span $E_1^0$, Proposition \ref{prop:all-closed} shows that all elements of $E_1^0$ are closed for all higher differentials in the spectral sequence. Since the $E_1$ page has nothing in negative degrees, we see that that the degree zero part of the spectral sequence degenerates at $E_1$. This completes the proof of Theorem \ref{thm:degeneration-linear}. 

\subsection{Continuation maps}
\label{sec:continuation-maps}

In order to complete the proof of Theorem \ref{thm:degeneration}, the last aspect we need to address to tie the calculation together is the continuation maps relating the various Floer cohomology groups used in the definition of symplectic cohomology. Recall from section \ref{sec:floer-cohomology}, one way to define symplectic cohomology is using Hamiltonian functions that are linear at infinity. A Hamiltonian with slope $m$ is denoted by $H^m$. The prohibited values for $m$ are the lengths of the periodic Reeb orbits in the contact hypersurface $\Sigma$, but otherwise we get a Floer cohomology group $HF^*(H^m)$. There are continuation maps $HF^*(H^m) \to HF^*(H^{m'})$ when $m \leq m'$. These maps are isomorphisms when $m = m'$, even if two different Hamiltonians of the same asymptotic slope are used to define the source and target spaces. The symplectic cohomology is defined as the limit of the directed system constructed from the continuation maps:
\begin{equation}
  SH^*(U) \cong \lim_{m\to\infty} HF^*(H^m)
\end{equation}

The significance of the continuation maps is that, in the computations of the various groups $HF^*(H^m)$, we must choose perturbations of the Hamiltonian and other structures, and the continuation maps express in a canonical way the invariance of $HF^*(H^m)$, or more precisely, its dependence solely on the asymptotic slope $m$. Furthermore, as we increase this asymptotic slope, we may need to modify the perturbations used to define the various generators $\theta_p$, and the continuation map from slope $m$ to slope $m'$ expresses in a canonical way how the generator $\theta_p$ defined in $HF^*(H^m)$ is related to the generator with the same name in $HF^*(H^{m'})$. 

\begin{proposition}
  \label{prop:continuation-maps}
  Let $m \leq m'$. The continuation map $\phi: HF^0(H^m) \to
  HF^0(H^{m'})$ is injective. For $p \in U^\trop(\Z)$, let
  $\theta_p^m$ and $\theta_p^{m'}$ be the degree zero generators in
  $HF^0(H^m)$ and $HF^0(H^{m'})$ corresponding to the same
  torus of periodic orbits $T_p$. There a choice of Hamiltonians that
  ensures that $\phi(\theta_p^m) = \theta_p^{m'} + \text{(higher action
    terms)}$
\end{proposition}

\begin{proof}
  First observe that, by an action filtration argument, the claim concerning $\phi(\theta_p^m)$ implies the claim that $\phi$ is an embedding. Indeed, using the elements $\theta_p^m$ and $\theta_p^{m'}$, ordered by decreasing levels of action, as basis of the source and target spaces, we find that the matrix of $\phi$ is upper triangular (a fact which is true for continuation maps in all situations). The claim of the proposition amounts to saying that the matrix elements along the diagonal are equal to one, so that $\phi$ is an embedding.
  
Up to this point we have not been particularly specific about our choice of Hamiltonians, since it is largely immaterial, but here we will make a more specific choice. Assume that both $H^m$ and $H^{m'}$ approximate the same quadratic Hamiltonian. By this we mean that $H^m$ is a piecewise polynomial function of $e^\rho$, with a quadratic piece and a linear piece: it is equal to the quadratic $H^Q = h(e^\rho)$ up to the point where $e^\rho = (h')^{-1}(m)$, and it is linear of slope $m$ afterward. As a result the function is $C^1$. The Hamiltonian $H^{m'}$ is similar, but, it stays quadratic over a longer interval up to when $e^\rho = (h')^{-1}(m')$. As a result, Hamiltonian periodic tori for $H^m$ are precisely a subset of those for $H^{m'}$, and the actions of these periodic tori are the same when computed using either Hamiltonian. Therefore, after perturbation, the actions of $\theta_p^m$ and $\theta_p^{m'}$ are approximately equal. Since the continuation map always increases action, we find that $\phi(\theta_p^m)$ contains a possible contribution from $\theta_p^{m'}$, and all other terms have higher action.

It remains to justify that $\theta_p^{m'}$ has a nonzero coefficient in $\phi(\theta_p^m)$. When computing $HF^0(H^m)$ and $HF^0(H^{m'})$, the Hamiltonians are perturbed in potentially different ways. Nevertheless, the orbits $\theta_p^m$ and $\theta_p^{m'}$ remain close to each other, and the cylinders connecting them have small energy. 


We can relate the desired computation to the case of $(\C^\times)^2$. One can set up the computation in that case so that there is a small neighborhood $\tilde{V}$ where all of the structures (Liouville form, Hamiltonians, complex structures) match with the structures in $V$. Denote the corresponding generators in the $(\C^\times)^2$ case by $\tilde{\theta}_p^m$ and $\tilde{\theta}_p^{m'}$. Since the Liouville class is locally convex, the spectral sequence degenerates at $E_1$ for both $HF((\C^\times)^2,H^m)$ and $HF((\C^\times)^2,H^{m'})$. Hence $\tilde{\theta}_p^m$ and $\tilde{\theta}_p^{m'}$ can be canonically identified with their images in cohomology.

We need to show that the generator $\theta_p^{m'}$ appears in $\phi(\theta_p^m)$ with coefficient 1. By Proposition \ref{prop:neck-stretching}, near the neck-stretching limit the cylinders contributing to this map are close to trivial and hence remain in a small neighborhood of the periodic torus $T_p$. This local computation is the same in $U$ and in $(\C^\times)^2$. Since the generators of the form $\theta_p^m$ and $\theta_p^{m'}$ represent the same class in $SH^0((\C^\times)^2) \cong H_2(\cL T^2)$, this matrix element of this continuation map in the case of $(\C^\times)^2$ must be one.
\end{proof}

The same idea is also used to relate the Floer cohomologies $HF^0(H^m)$ of Hamiltonians with linear growth to the Floer cohomology $HF^0(H^Q)$ of a Hamiltonian with quadratic growth. The flow of the quadratic Hamiltonian $H^Q$ creates at once  all of the periodic orbits we need to consider, and after perturbation we get degree-zero generators $\theta_p^Q \in CF^0(H^Q)$ corresponding to the torus of periodic orbits $T_p$. There are continuation maps $\phi: CF^0(H^m) \to CF^0(H^Q)$ from the finite slope version to the quadratic version. 


\begin{proposition}
  For any $m$, the continuation map $\phi: HF^0(H^m) \to HF^0(H^Q)$ is injective. If $\theta_p^m \in E_1^0(H^m)_p$ and $\theta_p^Q \in E_1^0(H^Q)_p$ are the classes corresponding to a torus of periodic orbits $T_p$, then there are lifts $\tilde{\theta}_p^m \in CF^0(H^m)_p$ and $\tilde{\theta}_p^Q \in CF^0(H^Q)_p$ and a choice of Hamiltonians that ensures $\phi(\tilde{\theta}_p^m) = \tilde{\theta}_p^Q + \text{(higher action terms)}$, and $\theta_p^Q$ is closed for all higher differentials in the spectral sequence.
\end{proposition}

\begin{proof}
  Let $\tilde{\theta}_p^m$ be the lift whose existence is guaranteed by Proposition \ref{prop:all-closed}. The same argument as in Proposition \ref{prop:continuation-maps} shows that there is a $d_0$-closed cochain $\tilde{\theta}_p^Q \in CF^0(H^Q)_p$ representing $\theta_p^Q \in E_1^0(H^Q)$ such that $\phi(\tilde{\theta}_p^m) = \tilde{\theta}_p^Q + Y$, where $Y$ consists of terms of higher action. Since $d\tilde{\theta}_p^m = dX$ for some cochain $X$ of higher action, we find that
  \begin{equation}
    0= \phi(d\tilde{\theta}_p^m - dX) = d\phi(\tilde{\theta}_p^m) -  d\phi(X) = d\tilde{\theta}_p^Q + dY - d\phi(X)
  \end{equation}
  Thus $\tilde{\theta}_p^Q$ is closed up to the differential of terms of higher action, and so $\theta_p^Q$ is closed for higher differentials in the spectral sequence, and this spectral sequence degenerates. 

  This in particular shows that the continuation map $\phi : HF^0(H^m) \to HF^0(H^Q)$ is triangular with respect to the action filtration, with ones along the diagonal, and therefore that $\phi$ is injective on cohomology.
\end{proof}

To deduce Theorem \ref{thm:degeneration} and hence Theorem \ref{thm:main}, we take $\theta_p = \theta_p^Q$.

\begin{remark}
  The fact that the continuation maps are triangular with respect to the action filtration (rather than strictly diagonal) is unsatisfying if we want to claim to have found a ``canonical basis'' for the symplectic cohomology.  It seems to the author quite likely that the higher action terms in $\phi(\theta_p^m)$ vanish if the Hamiltonian is chosen correctly. Knowing this would make the result somewhat sharper, in that we would have a more compelling reason to identify the basis elements $\theta_p$ with the canonical basis elements of Gross-Hacking-Keel. But even if we were able to prove that, it seems to the author that the only true test of whether these $\theta_p$ really are the Gross-Hacking-Keel theta functions will come when one matches up the product structure on $SH^0(U)$ with the product of theta functions.
\end{remark}

\section{Wrapped Floer cohomology}
\label{sec:wrapped}

We will now describe a relationship between the symplectic cohomology of $U$ and the wrapped Floer cohomology of certain Lagrangian submanifolds in $U$, and make a connection with the results of \cite{binodal}. 

We consider Lagrangian submanifolds $L$, which may be either compact
or \emph{cylindrical at infinity}, meaning that, within the
cylindrical end $\Sigma \times [0,\infty)$ of the completion of our
Liouville domain, $L$ has the form $\Lambda \times [0,\infty)$, where
$\Lambda$ is a Legendrian submanifold of $\Sigma$. 

Given two such Lagrangians $L$ and $K$, the wrapped Floer cohomology
$HW^*(L,K)$ is a $\K$--vector space, which in the situation we consider
will be $\Z$--graded. The definition is parallel to that of symplectic
cohomology: We fix a Hamiltonian $H$. There is a cochain complex
$CF^*(L,K;H)$ generated by time--one chords of $H$ starting on $L$ and
ending on $K$ (which are periodic orbits of $H$ starting at $L$, if $K = L$). The differential
now counts pseudo-holomorphic strips joining such chords, rather than
cylinders. If we choose a quadratic Hamiltonian $H^Q$, we denote this complex by $CW^*(L,K)$, and its cohomology is $HW^*(L,K)$. If we use a linear Hamiltonian $H^m$ of slope $m$, then as before we need to take a direct limit as $m \to \infty$.

Using the wrapped Floer cohomology $HW^*(L,K)$ as the space of
morphisms from $L$ to $K$, we obtain the cohomology--level version of
the wrapped Fukaya category $\cW(U)$. At the chain level, $\cW(U)$ is
an $A_\infty$--category \cite{as-open-string}, which forms the A--side
of the HMS correspondence for open symplectic
manifolds. In particular, the endomorphisms of a single object, $HW^*(L,L)$, forms a ring. This will be our main object of interest. 

To relate symplectic cohomology and wrapped Floer cohomology, we use
\emph{closed-to-open string maps}
\cite{abouzaid-geometric-criterion}. There are various versions, all
defined by counting pseudo-holomorphic curves with boundary, and with
a mixture of interior punctures (corresponding to generators of symplectic
cohomology) and boundary punctures (corresponding to generators of
wrapped Floer cohomology). The first of these is a map $\mathcal{CO}_0 : SH^*(U)
\to HW^*(L,L)$. This map fits into a larger structure, a map
\begin{equation}
  \label{eq:closed-open}
  \mathcal{CO}: SH^*(U) \to HH^*(CW^*(L,L))
\end{equation}
where $HH^*(CW^*(L,L))$ denotes the Hochschild cohomology of the
$A_\infty$--algebra $CW^*(L,L)$ (with coefficients in itself). The map $\mathcal{CO}$ is a map of rings \cite[Proposition 5.3]{ganatra-thesis}.

\subsection{Lagrangian sections}
\label{sec:lagrangian-sections}

In the case of an affine log Calabi--Yau surface $U$ with compactification $(Y,D)$, there is a natural class of Lagrangian submanifolds to consider for wrapped Floer cohomology, namely Lagrangians which are sections of the torus fibration near the divisor. By Lemma \ref{lem:legendrian-section} there is a Legendrian section of the torus fibration $\pi: \Sigma \to S$ on the contact hypersurface. What we desire is a Lagrangian $L$ that caps off this circle in $\Sigma$ to a disk in $U$. For our present purposes, we say that a Lagrangian $L$ is a \emph{section} if
\begin{enumerate}
\item $L$ is diffeomorphic to a disk, and
\item At infinity, $L$ is a cylinder over a Legendrian section of the torus fibration on $\Sigma$.
\end{enumerate}

The wrapped Floer cohomology $HW^*(L,L)$ is simple to compute using a quadratic Hamiltonian. Whereas in computing symplectic cohomology we encountered $T^2$--families of periodic orbits for the Hamiltonian, chords of the Hamiltonian flow joining $L$ to itself have less symmetry. Since $L$ intersects each torus in $\Sigma$ in one point, each torus of periodic orbits (in the free loop space) contains exactly one orbit that is a chord from $L$ to $L$. As these tori are indexed by points $p \in U^\trop(\Z)\setminus \{0\}$, denote by $\theta_{L,p}$ the corresponding chord. The chords $\theta_{L,p}$ are generators for $CW^*(L,L)$, they have degree zero by a comparison to the case of $(\R_{>0})^2 \subset (\C^\times)^2$ analogous to Proposition \ref{prop:perturbation-indices}. There is one more generator $\theta_{L,0} \in CW^0(L,L)$ corresponding to the ordinary cohomology $H^0(L)$, which has rank one since $L$ is topologically a disk. Because $CW^*(L,L)$ is concentrated in degree zero, the differential vanishes trivially. 

\begin{proposition}
  Let $L$ be a Lagrangian section. Then the wrapped Floer complex $CW^*(L,L)$ is concentrated in degree zero, is isomorphic to its cohomology, and has a basis of chords indexed by the points of $U^\trop(\Z)$.
  \begin{equation}
    CW^0(L,L) \cong HW^0(L,L) \cong \Span \{\theta_{L,p}\mid p \in U^\trop(\Z)\}
  \end{equation}
\end{proposition}

Next we consider the closed-to-open string map $\mathcal{CO}_0 : SH^0(U) \to HW^0(L,L)$. It is useful to consider once again the case of the complex torus $(\C^\times)^n$ (as always, with locally convex Liouville class) as a model. For the Lagrangian section we take the real positive locus $L = (\R_{>0})^n \subset (\C^\times)^n$. The wrapped Floer cohomology is isomorphic (as a ring) to the space of Laurent polynomials
\begin{equation}
  HW^0(L,L) \cong \K[x_1^{\pm 1},\dots, x_n^{\pm 1}]
\end{equation}
and hence abstractly isomorphic to $SH^0((\C^\times)^n)$. The map $\mathcal{CO}_0$ implements this isomorphism concretely, by counting disks satisfying an inhomogeneous pseudo-holomorphic map equation with one interior puncture corresponding to the input, one boundary puncture corresponding to the output, and a Lagrangian boundary condition on $L$. The map $\mathcal{CO}_0$ sends the degree zero generator in $SH^0((\C^\times)^n)$ corresponding to a torus of periodic orbits to the unique chord in $HW^0(L,L)$ corresponding to the same torus. Indeed, since $L$ is contractible, $HW^0(L,L)$ carries a grading by $H_1((\C^\times)^n;\Z)$, as does $SH^0((\C^\times)^n)$, and $\mathcal{CO}_0$ is a homogeneous map. Therefore there can be no disks even topologically connecting $\theta_p$ to $\theta_{L,q}$ for $q \neq p$.

In the general situation of a Lagrangian section $L$ in a log Calabi--Yau surface $U$, the leading order term of the map $\mathcal{CO}_0$ looks the same as in the case of $(\C^\times)^2$.  

\begin{proposition}
  \label{prop:section-iso}
  For a Lagrangian section $L$ in an affine log Calabi--Yau surface with maximal boundary $U$, the closed-open map
  \begin{equation}
    \mathcal{CO}_0 : SH^0(U) \to HW^0(L,L)
  \end{equation}
  is a ring isomorphism that satisfies 
\begin{equation}
  \mathcal{CO}_0(\theta_p) = \theta_{L,p} + \text{(higher action terms)}
\end{equation}
\end{proposition}

\begin{proof}
  Given degree zero generator $\theta_p \in SH^0(U)$, there are low-energy pseudo-holomorphic curves connecting $\theta_p$ to the corresponding generator $\theta_{L,p} \in HW^0(L,L)$. The proof is analogous to Proposition \ref{prop:continuation-maps}. For this, we need the analogue of Proposition \ref{prop:neck-stretching} in the case of the closed-open map, which in turn depends on an analogues of Theorem \ref{thm:no-hol-cylinders} and Corollary \ref{cor:no-hol-cylinders-in-symplectization}. In fact, there is such an analogue for curves with boundary on a Lagrangian section. In order to adapt the proofs to this case, first note that since $L$ is a section, Reeb chords on $L$ correspond to a subset of the closed Reeb orbits. The key argument is based on intersecting a holomorphic curve with the fibers of the projection $\pi : \Sigma \to S$, and analyzing the integral of the contact form over the resulting one-manifolds which are arcs that may have boundary on $L$. But since $L$ is a cone over a Legendrian section of $\pi$ (intersecting each fiber in a single point), they are in fact still closed cycles in the fibers of $\pi$, and the argument can proceed as before.

Since the closed-open map is triangular with respect to the action filtration with ones along the diagonal, it is an isomorphism.
\end{proof}

In many cases, it is simple to construct a Lagrangian submanifold with the desired properties. In the toric case, we take $L = (\R_{>0})^2 \subset (\C^\times)^2$. Variations on this work in other cases. When $Y$ is obtained by the blow up of $\CP^2$ in several points, we can often arrange for the blow up points and the anticanonical divisor to be compatible with the real structure on $\CP^2$. Removing the anticanonical divisor will then disconnect the real locus of $Y$ into several components. If done right one of them will be a disk. 

For example, in the case of the cubic surface with a triangle of lines (\S \ref{sec:cubic}), we can take the blow up points $p_1, \dots, p_6$ to be real. The anticanonical divisor has the form $D = L_{ab}+ E_b + C_a$. If we choose this so that $C_a$ is represented in the real picture as an ellipse, and the region bounded by $C_a$ and $L_{ab}$ contains $p_b$ as the only blow up point on its boundary, then this region is a connected component of the real locus of $Y \setminus D$, which we may take as our Lagrangian section.

In the case of the degree 5 del Pezzo surface (\S \ref{sec:M05}), we blow up $\CP^2$ in four points $p_1,\dots,p_4$, which we take to be real. When we remove an anticanonical $5$--cycle, one component of the real locus is a disk. In fact, removing all $10$ of the $(-1)$--curves at once disconnects the real locus (which is a non-orientable surface of Euler characteristic $-3$) into $12$ disks (which are combinatorially pentagons).

Another way to find appropriate Lagrangians is to explicitly consider the Lagrangian torus fibration on the whole of $U$, rather than just near the boundary, and try to construct a section thereof. Such a torus fibration exists on $U$ due to the existence of a toric model \cite{ghk}, and the results of Symington on almost-toric structures on blow-ups \cite{symington-4from2}. For example, in the case of a punctured $A_n$ Milnor fiber $U$ (\S \ref{sec:a-n-milnor-fiber}), there is a Lefschetz fibration on $U \to \C^\times$ whose fibers are affine conics with $n+1$ singular fibers at the points $\xi_k = \exp(2\pi ik/(n+1))$. Following the sort of construction found in \cite{auroux07}, we can construct a Lagrangian torus by taking a circle of radius $r$ in the base $\C^\times$, and looking at the family of circles at some fixed ``height'' in the conic fibers. Such tori foliate $U$, and there is one singular fiber, which is a torus with $n+1$ nodes. To construct a Lagrangian section of this torus fibration, begin with a path $\ell$ in the base $\C^\times$ joining $0$ to $\infty$, and not passing through any critical value. Over a particular point in $\ell$, the fiber is an affine conic, also isomorphic to $\C^\times$, and we may take again an infinite path joining the two ends. Under the symplectic parallel transport along $\ell$, this path in the fiber sweeps out a Lagrangian in the total space, which is our $L$. If the path $\ell$ crosses each circle centered at the origin once, and the path in the fiber crosses each circle of constant ``height'' once, then $L$ is actually a section of the torus fibration. If the constructions of Section \ref{sec:construction} are done compatibly with $L$, then $L$ will be section in the above sense, so Proposition \ref{prop:section-iso} applies.



\subsection{The case of the affine plane minus a conic}
\label{sec:c2-conic}

In \cite{binodal}, the present author considered the case where $U$ is
the complement of a smooth conic in $\C^2$, which is to say the
complement of a conic and a line in $\PP^2$. In this case the mirror
$U^\vee$ is likewise $\Af^2_\K$ minus a conic (an accident of low
dimensions). The ring of global functions on $U^\vee$ is:
\begin{equation}
  \cO(U^\vee) \cong \K[x,y][(xy-1)^{-1}]
\end{equation}
We considered a Lagrangian torus fibration on the whole of $U$, and a Lagrangian section $L$. This Lagrangian section fits into the discussion above, so its wrapped Floer cohomology is concentrated in degree zero and there is no differential. Using techniques particular to this case (and others like it), we computed the ring structure on the wrapped Floer cohomology, and showed that it is isomorphic to the same ring:
\begin{equation}
  CW^0(L,L) \cong HW^0(L,L) \cong \K[x,y][(xy-1)^{-1}]
\end{equation}
This isomorphism carries the basis elements in $HW^0(L,L)$ corresponding to chords to the functions of the form 
\begin{equation}
\{x^ay^b(xy-1)^c\mid a \geq 0, b \geq 0, c \in \Z\}
\end{equation}

With this result at hand, we can extend our discussion of the closed-to-open string map. Since we know that $CW^*(L,L)$ is a commutative ring concentrated in degree zero, we find that $HH^0(CW^*(L,L)) \cong CW^0(L,L) \cong \cO(U^\vee)$, where the second isomorphism is proven in \cite{binodal}. Since this ring is moreover smooth over $\K$, the Hochschild--Kostant--Rosenberg theorem implies that 
\begin{equation}
  HH^p(CW^*(L,L)) \cong H^0(U^\vee,\wedge^p T_{U^\vee})
\end{equation}
The degree zero piece of the map $\mathcal{CO}$
\begin{equation}
  \mathcal{CO}: SH^0(U) \to HH^0(CW^*(L,L)) \cong CW^0(L,L)
\end{equation}
is just the map $\mathcal{CO}_0$ considered above. Since the map $\mathcal{CO}$ is naturally a ring map, we can combine this with the previous discussion to obtain that $SH^0(U)$ is isomorphic to $CW^0(L,L)$ as a ring, and thus
\begin{equation}
  SH^0(U) \cong CW^0(L,L) \cong \cO(U^\vee).
\end{equation}
This confirms the Gross-Hacking-Keel conjecture, including the ring structure, in this example.

The results of \cite{binodal} can also be used to treat the case of the punctured $A_n$ Milnor fibers (in which the affine manifold has ``parallel monodromy-invariant directions''). See \cite{chan-ueda} for a discussion of related cases.

\section{Coefficients in a positive cone and degeneration to the vertex}
\label{sec:coeff-degeneration}

From the point of view of algebraic geometry, one of the most interesting features of the symplectic cohomology $SH^*(U)$ is that it is a graded commutative ring. In particular, the degree zero part $SH^0(U)$ is a commutative ring that is putatively the ring of functions on the mirror of $U$. In this section we will discuss this ring structure using the techniques developed in the previous sections. Since these techniques are mainly suited for proving the emptiness of certain moduli spaces, we only obtain results saying that a certain generator cannot appear in the product of two other generators. A fuller analysis including an enumeration of the nonempty moduli spaces would be required to prove the full strength of Gross--Hacking--Keel's conjectures.

In this section, we shall write $H_{2}(Y)$ for $H_{2}(Y;\Z)$. We will consider three variants of the symplectic cohomology with different coefficient rings:
\begin{enumerate}
\item $SH^{*}(U;\K)$, with coefficients in $\K$. This is what we have been considering up to this point.
\item $SH^{*}(U;\K[H_{2}(Y)])$, with coefficients in the group algebra $\K[H_{2}(Y)]$ of the abelian group $H_{2}(Y)$. This is constructed in Section \ref{sec:homology-pop}.
\item $SH^{*}(U;\K[P])$, with coefficients in the monoid algebra $\K[P]$ of a certain submonoid $P\subset H_{2}(Y)$. This is constructed in Section \ref{sec:positive-cone}.
\end{enumerate}
The main result of this section is Theorem \ref{thm:degeneration-to-vertex}. 

\subsection{Homology class associated to a pair of pants}
\label{sec:homology-pop}

The product on symplectic cohomology is defined in terms of counting inhomogeneous pseudo-holomorphic maps of pairs of pants into $M$, the completion of the domain $\overline{U}'$. For the rest of this section we will identify $M$ and $U$. The punctures of the pairs of pants map asymptotically to periodic orbits of the Hamiltonian function. To each such map one can associate a relative homology class: if $\gamma_1$, $\gamma_2$, and $\gamma_3$ are periodic orbits, and $u : S \to U$ is a map contributing to the coefficient of $\gamma_3$ in the product of $\gamma_1$ and $\gamma_2$, we obtain a relative homology class
\begin{equation}
  [u] \in H_2(U,\gamma_1\cup \gamma_2 \cup \gamma_3)
\end{equation}
This class is awkward to work with, since the space it lives in depends on the periodic orbits under consideration, but it is possible to promote it to an absolute homology class in a compactification. Let $Y$ be a fixed compactification of $U$ by a cycle of rational curves as before. Letting $C = \gamma_1\cup\gamma_2\cup \gamma_3$ The exact sequence of the pair $(Y,C)$ is
\begin{equation}
  H_2(C) = 0 \to H_2(Y) \to H_2(Y, C) \to H_1(C) \cong \Z^3
\end{equation}
We wish to split the map $H_2(Y) \to H_2(Y,C)$. Since this sequence is left exact, it suffices to split the map $H_2(Y,C) \to H_1(C)$. This amounts to choosing a disk in $Y$ that bounds each periodic orbit $\gamma_i$. For the given compactification $Y$, there is a natural way to do this. The periodic orbits corresponding to the cohomology of $U$ are contractible by construction, and in fact each is localized in a neighborhood of a critical point of the function used to perturb the Hamiltonian in the interior, so we simply use a small disk in that neighborhood (this is independent of the compactification). For the orbits near the divisor, we can use disks that pass through the divisor. To fix the remaining ambiguity, we require that, for a orbit near the node $D_i \cap D_{i+1}$, we use a disk that is entirely contained in a neighborhood of that node (any two such disks are homologous within the neighborhood). There are also orbits in the middle of each divisor, consisting of normal circles to the divisor, and for these we use a normal disk to the divisor.

This prescription guarantees that, if $\gamma$ is a Reeb orbit that
links the divisors $D_i$ and $D_{i+1}$, the capping disk only
intersects $D_i$ and $D_{i+1}$, with some multiplicities. These
multiplicities are non-positive when the disk is oriented so
that its boundary is $\gamma$ with the orientation given by the Reeb
flow.

We denote by $\phi(u) \in H_2(Y;\Z)$ the homology class in the compactification given by capping $u$ with the disks chosen above. It is immediate that the preceding discussion generalizes from the pair of pants to maps of any Riemann surface with punctures, such that the punctures map asymptotically to periodic orbits.

Now we may define the symplectic cohomology $SH^{*}(U;\K[H_{2}(Y)])$ with coefficients in the group algebra $\K[H_2(Y)]$ whose basis elements are denoted $q^{c}$ for $c \in H_2(Y)$. This is standard: all of the Hamiltonian Floer complexes used in the definition are now taken to be free $\K[H_{2}(Y)]$-modules spanned by the periodic orbits (rather than $\K$-vector spaces as before). Whenever a pseudo-holomorphic map $u : S \to U$ contributes to an operation, this term counts with coefficient $q^{\phi(u)}$. This affects the differential, the continuation maps, the BV operator, and the product (as well as higher operations). The original $SH^{*}(U;\K)$ is recovered by setting $q^{c} = 1$ for all $c \in H_2(Y)$. 

\subsection{A positive cone in $H_2(Y;\Z)$}
\label{sec:positive-cone}

Let $\NE(Y)$ denote the cone of curves, namely, the cone in $H_2(Y;\Z)$ spanned by the classes of effective curves, that is, linear combinations of homology classes of complex curves with non-negative coefficients. The optimal goal would be to show that the coefficient ring of symplectic cohomology, in the cases under consideration, can be reduced to $\K[\NE(Y)]$, the monoid ring of the cone of curves. For technical reasons, we will actually work with a larger monoid $P$, defined as
\begin{equation}
  \label{eq:P-monoid}
  P = \langle [D_1], \dots, [D_n]\rangle  + \langle [C] \mid   (\forall i)(C \cdot D_i \geq 0) \text{ and }C\cdot D > 0\rangle \subset H_2(Y;\Z)
\end{equation}
In this context the angle brackets denote the submonoid generated by the enclosed elements. Thus an element of $P$ is either a positive combination of components of the boundary divisor, or a class (not necessarily effective) that intersects each component of the boundary non-negatively and is not disjoint from $D$, or a sum of such. Note that, in the case at hand, where $D$ supports an ample divisor, the left summand in the definition is contained in the one on the right. Note also that, since $P$ is a monoid, $0 \in P$. The reason for enlarging the monoid is that we are counting inhomogeneous pseudo-holomorphic curves rather than honest holomorphic curves for the integrable complex structure. Thus it is not clear that the classes we get are effective in the standard sense, while it is possible to show that they lie in $P$.
\begin{proposition}
  \label{prop:P-monoid}
  The monoid $P$ is strictly convex (that is, $v \in P\setminus \{0\}$ implies $-v \notin P$), and it contains $\NE(Y)$. If $A = \sum a_i D_i$ is an ample divisor with all $a_i > 0$, then $A$ is strictly positive on any non-zero element of $P$.
\end{proposition}
\begin{proof}
  To see that $P$ contains $\NE(Y)$, let $C$ be an irreducible (effective) curve. Either $C$ coincides with a component of $D$, in which case $[C]$ is in the first term of \eqref{eq:P-monoid}, or $C$ is not a component of $D$. In the latter case, $C \cdot D_i \geq 0$ follows by positivity of intersection, as does $C \cdot D \geq 0$. To prove the strict inequality $C \cdot D > 0$, we use the fact that $D$ supports an ample divisor, so that $C$ cannot be disjoint from $D$. Thus $[C]$ lies in the second term of \eqref{eq:P-monoid}.

The statement that $P$ is strictly convex follows from the statement that $A$ is strictly positive on $P \setminus \{0\}$. To prove the latter statement, it suffices to check the generators. First, $A\cdot D_i > 0$ since $A$ is ample and $D_i$ is effective. Second, if $C$ is such that $C \cdot D_i \geq 0$, and $C \cdot D > 0$, then the numbers $C \cdot D_i$, as $i$ varies, are non-negative and not all zero. Thus $A\cdot C = \sum a_i(C \cdot D_i) > 0$.
\end{proof}

Now we can state the main result of this section. Let $\K[P] \subset \K[H_{2}(Y)]$ be the subalgebra generated by $P \subset H_{2}(Y)$. 

\begin{proposition}
  \label{prop:effectivity}
  There is a well-defined subspace $SH^0(U;\K[P]) \subset SH^{0}(U;\K[H_{2}(Y])$ spanned by $\K[P]$-linear combinations of periodic orbits. The homology classes associated to
  pseudo-holomorphic curves contributing to the product
  lie in $P$. Thus, $SH^0(U,\K[P])$ is closed under multiplication, and has the structure of a $\K[P]$--algebra.
\end{proposition}

\begin{remark}
   Regarding the symplectic cohomology as a family over the base given by the spectrum of the coefficient ring, we may interpret the change from $\K$ to $\K[H_2(Y)]$ to $\K[P]$ as changing the base first from a point to an algebraic torus containing that point as its identity element, and then to a partial compactification of that algebraic torus.
\end{remark}

The proof of Proposition \ref{prop:effectivity} uses the same neck--stretching argument used to describe the differential. Consider three periodic orbits $\gamma_1$, $\gamma_2$, and $\gamma_3$, and suppose that $u : S \to U$ is a map from the pair of pants that is considered when computing the coefficient of $\gamma_3$ in $\gamma_1 \cdot \gamma_2$. This has some class $\phi(u) \in H_2(Y;\Z)$ given by capping off the periodic orbits in the prescribed way. To show that $\phi(u)$ is in $P$, we deform the situation just as in section \ref{sec:neck-stretching}. Suppose that curves in class $\phi(u)$ contribute a nonzero count to the coefficient of $\gamma_3$ in $\gamma_1\cdot \gamma_2$, and that such curves persist through out the neck--stretching process, which also involves canceling the perturbation of the Hamiltonian to make $H$ radial. The limiting configuration is a building, with one level in $U$, and other levels in $\Sigma \times \R$. The homology class of the limiting configuration is the sum of the homology classes associated to each of the levels, where we may embed $\R\times \Sigma$ into $Y$ as a tubular neighborhood of $D$ with $D$ itself removed. This involves capping off all of the Reeb orbits where the curve breaks, but because different levels are joined along Reeb orbits, all of the ``caps'' cancel out except for those associated to the original boundary orbits $\gamma_1,\gamma_2$, and $\gamma_3$. We will show that the homology class associated to each level of this broken curve lies in $P$.

\begin{lemma}
  \label{lem:interior-level}
  If $C$ is the homology class of the level of the broken curve lying in $U$, then $C$ is either zero, or else satisfies $C\cdot D_i \geq 0$ for all $i$ and $C \cdot D > 0$, and hence is in $P$.
\end{lemma}

\begin{proof}
  Observe that the level in $U$ has asymptotics at Reeb orbits in $\Sigma$, which are oriented so that the orientation of the orbit given by the Reeb flow agrees with the orientation as the boundary of the Riemann surface. We must cap these off using disks whose boundary orientation is \emph{opposite} to the Reeb orientation. Since the Reeb orbits wind \emph{negatively} around the components of $D$, these caps intersect the components of $D$ \emph{positively}. Thus $C$ intersects each $D_i$ positively, if at all. If the product operation we are computing involves orbits in the cylindrical end, then the level in $U$ must have some connection to the other levels, and so there must be at least one such Reeb orbit, and a positive intersection with some $D_i$. If this is not the case, then we must be computing the product $1\cdot 1 = 1$, in which case the homology class is zero.
\end{proof}

\begin{lemma}
  \label{lem:symplectization-level}
  If $C$ is the homology class of any level of the broken curve lying in $\Sigma \times \R$, then $C \in \langle[D_1],\dots,[D_n]\rangle \subset P$.
\end{lemma}

\begin{proof}
  Let $u: T \to \Sigma \times \R$ be an inhomogeneous pseudo-holomorphic curve in the symplectization that appears as one of the levels in the broken curve. Here the domain Riemann surface $T$ has several positive and negative punctures, and the inhomogeneous term is given by a Hamiltonian function that depends only on the $\R$ component of the target. To determine the homology class $C = \phi(u)$ associated to $u$, we embed $\Sigma \times \R$ into a normal neighborhood $N(D)$ of the divisor $D$, and we cap off all of the punctures using disks that pass through $D$. Thus we may compute the homology class in the group $H_2(N(D);\Z) \cong \bigoplus_{i=1}^n \Z[D_i]$, spanned by the irreducible components of $D$. This involves projecting the image of $u$ (and the various capping disks) onto the various components of $D$, using a retraction $r: N(D) \to D$. To construct such a retraction, recall that $N(D)$ is topologically a plumbing of several spheres in a cycle. We may construct $r$ first over $D_i$ minus a neighborhood of the nodes so that it is a disk fibration, and then extend these maps over neighborhoods of the nodes. One way to construct the latter map is to consider the local model $\{(z_1,z_2) \in \C^2 \mid |z_1|^2 + |z_2|^2 \leq 1\}$, and consider the symplectic parallel transport for the Lefschetz fibration $(z_1,z_2) \mapsto z_1z_2$ into the central fiber $z_1z_2 = 0$.

Recall that the contact manifold $\Sigma$ is a torus bundle $\pi : \Sigma \to S$, where $S$ is a circle. Consider a component $D_i$. From Step 5 (Section \ref{sec:step5}) of the construction of the Liouville domain, there is a particular value $s_i \in S$ such that the Reeb orbits on the torus $\pi^{-1}(s_i)$ consist of orbits of the circle action along the smooth part of $D_i$ constructed in Step 4 \ref{sec:step4}. Consider a small interval $(s_i - \epsilon, s_i + \epsilon)$ about this point, and the preimage $V = \pi^{-1}(s_i - \epsilon,s_i+\epsilon)$. Recalling that $\Sigma$ itself was chosen as the boundary of a neighborhood of $D$, we can set up the retraction $r: N(D)\to D$, so that $r(V)$ is an annulus $A$ in $D_i$, and $r|_V : V \to A$ is a circle fibration whose fibers are orbits of the circle action along the smooth part of $D_i$.

Along the torus $\pi^{-1}(s_i)$, the planes of the contact structure $\xi$ are transverse to the fiber of $r|_V$, since this fiber is the Reeb orbit. Furthermore, under the projection $r|_V$, the orientation of $\xi$ agrees with the complex orientation of $A \subset D_i$ (recall that the orientation of $\Sigma \times \R$, which agrees with the orientation of $N(D)$, is given by taking, in order, the vector pointing radially \emph{towards} $D$, the angular vector winding \emph{negatively} around $D$, and the contact plane $\xi$). Therefore, adjusting $\epsilon$ if necessary, we may guarantee that at every point of $V$, the projection $r|_V$ maps the contact plane $\xi$ onto the tangent space of $D_i$ isomorphically preserving orientation. 

Consider the map $r \circ u : T \to D$, mapping the open set $(r\circ u)^{-1}(A)$ to $A$. Combining the previous paragraph with Proposition \ref{prop:positivity-contact-target}, we find that, at any point in the domain mapping to the circle $r(\pi^{-1}(s_i))$, the rank of $r \circ u$ is either zero or two (the characteristic rank dichotomy for holomorphic maps between Riemann surfaces). This is because the differential of the map $r$ factors through the projection $\pi_\xi : T\Sigma \to \xi$, and $\pi_\xi \circ du$ has rank either zero or two. We claim that $r\circ u$ has a regular value sitting on the circle $r(\pi^{-1}(s_i))$. To see why this is so, consider the further projection $\sigma$ mapping the annulus $A$ to the core circle $r(\pi^{-1}(s_i))$. By Sard's Theorem, the composite $\sigma \circ (r \circ u)$ must have a dense set of regular values, where the rank at any point in the preimage is one. Let $p$ be one such regular value; we may assume that $p$ is disjoint from any of the capping disks used in the construction of $\phi(u)$, for if a capping disk intersects $\pi^{-1}(s_{i})$, it means that the corresponding Reeb orbit lies in that torus, and in that case the capping disk maps to a single point under $r$. This same value $p$ is regular for $(r \circ u)$, for if not, there would be a point in the domain where the rank of $(r \circ u)$ is less than two, hence zero by the dichotomy, and hence $\sigma \circ (r \circ u)$ would also have rank zero at $p$. The degree of the map $(r \circ u)$ is non-negative at $p$, because the maps $\pi_\xi \circ du : TT \to \xi$ and $dr : \xi \to TD_i$ are orientation preserving.

The coefficient of $[D_i]$ in the homology class $C$ is equal to the degree at the regular value $p$. One way to see this is to use Poincar\'{e}--Lefschetz duality in $N(D)$ and intersect with the disk $r^{-1}(p)$.
\end{proof}

Lemmas \ref{lem:interior-level} and \ref{lem:symplectization-level} obstruct the existence of broken curves representing homology classes outside of $P$. To apply them, we argue as follows. Suppose we wish to compute one of the matrix coefficients with respect to a basis of periodic orbits of some operation (differential, continuation map, product, and so on). We restrict attention to the components of the relevant moduli space of maps representing homology classes \emph{not in} $P$. If this space is not eventually empty as we stretch the neck and turn off the perturbation of the Hamiltonian $H$, Gromov compactness yields a subsequence that converges to a broken curve, which, by the nature of Gromov convergence, has a homology class not in $P$, which is impossible. Thus, by deforming the situation close enough to the limit, we guarantee that all curves contributing to the operation have homology classes in $P$.

\begin{remark}
  In this argument, it is of some importance that the \emph{a priori} energy bound in terms of action differences gives us compactness across all homology classes simultaneously, rather than just one homology class at a time.
\end{remark}

To construct $SH^{0}(U;\K[P])$ we proceed as follows. First recall that $HF^{0}(H^{m}) = CF^{0}(H^{m})$ (with any coefficients) since the differential vanishes. Therefore we may take
\begin{equation}
  HF^{0}(H^{m};\K[P]) = CF^{0}(H^{m};\K[P]) \subset CF^{0}(H^{m};\K[H_{2}(Y)])  
\end{equation}
to be the subspace spanned by $\K[P]$-linear combinations of periodic orbits. As for the continuation maps $HF^0(H^m;\K[H_{2}(Y)]) \to HF^0(H^{m'};\K[H_{2}(Y)])$, the neck--stretching argument shows that we can set these up so that the curves all have homology classes lying in $P$. Thus the continuation maps preserve the subspaces $HF^{0}(U;\K[P])$, and we may define $SH^{0}(U;\K[P])$ to be the limit of these subspaces. Finally, the neck--stretching argument shows that the curves contributing to the product have homology classes in $P$, so the product $HF^0(H^{m_1};\K[P])\otimes HF^0(H^{m_2};\K[P]) \to HF^0(H^{m_1+m_2};\K[P])$ makes sense and passes to a well-defined product on $SH^{0}(U;\K[P])$. This completes the proof of Proposition \ref{prop:effectivity}.

\subsection{Broken line diagrams}
\label{sec:drawing}

Recall that associated to the log Calabi--Yau surface $U$ there is an affine manifold $U^\trop$, which is in fact obtained by gluing together several quadrants $(\R_{\geq 0})^2$ by $\SL_2(\Z)$ transformations. In this section, we will show how to associate a diagram in $U^\trop$ to a holomorphic curve in $\Sigma \times \R$ such as those appearing in the neck-stretching limits discussed previously. These diagrams are graphs in $U^\trop$ consisting of straight line segments and rays, which we call \emph{broken line diagrams}. These diagrams are a specific type of tropical curve in $U^\trop$. While we do not provide here an enumerative theory of such curves, much less prove that the counts of holomorphic curves equal the putative tropical analog, they have some immediate topological applications. For instance, the broken line diagram encodes some information about the homology class of the holomorphic curve. The diagrams we describe are closely related to the tropical curves called \emph{broken lines} in \cite{ghk} from which they get their name. (As for the difference between the concepts, the diagrams described here may consist of several broken lines put together.) In fact, it was by thinking about such broken lines and the putative correspondence to holomorphic curves that we were led to the results of the present paper.

Figure \ref{fig:broken-line} shows an example of a broken line diagram. Two quadrants in the affine manifold $U^\trop$ are depicted, corresponding to the nodes $D_{i-1}\cap D_i$ and $D_i \cap D_{i+1}$. We have assumed $(D_i)^2 = 0$. The tropical curve has three legs: an infinite horizontal one with weight two, an infinite one of slope $-1$, and a finite leg going into the origin with slope $1$. This figure corresponds to a pair of pants contributing to the product of the $(2,0)$ class in the right-hand quadrant with the $(1,1)$ class in the left-hand quadrant, and resulting in the $(1,1)$ class in the right-hand quadrant.
\begin{figure}
  \centering
  \includegraphics[width=3in]{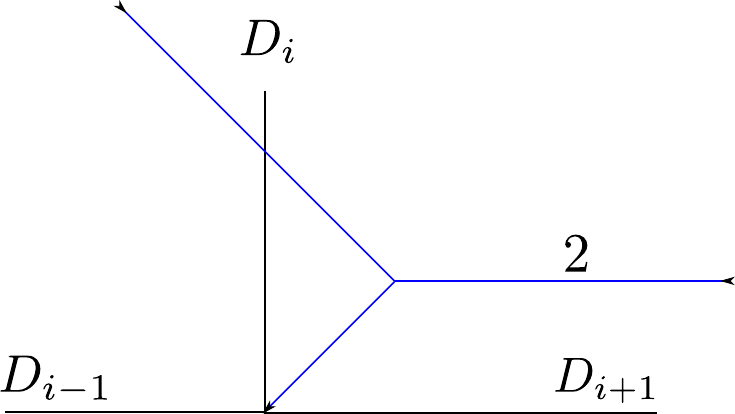}
  \caption{A example of a broken line diagram.}
  \label{fig:broken-line}
\end{figure}

Rather than using any of the more standard techniques of tropical geometry (see for instance \cite{parker}), our approach is almost completely topological, using the torus fibration $\pi : \Sigma \to S$ on the contact hypersurface at infinity, and the Liouville class $A(s) \in H^1(\pi^{-1}(s);\R)$. By projecting a holomorphic curve to $\Sigma$ and intersecting with a generic torus fiber, we obtain a loop. The holomorphicity condition comes in to say that the integral of $A(s)$ over this loop is monotonic as a function of $s$, and this is what allows us to analyze the curve and translate it into a diagram in $U^\trop$, as we shall now explain. 

First consider the general situation of a holomorphic curve $u = (f,a) : T \to \Sigma \times \R$, with several positive and negative punctures asymptotic to Reeb orbits. We are mainly interested in the $\Sigma$-component map $f : T \to \Sigma$. We have the projection $\pi \circ f: T \to S$, a map from the domain Riemann surface $T$ to the circle $S$. Since Reeb orbits map to points in $S$, the map $(\pi \circ f)$ extends continuously to the compactification $\overline{T}$. The level sets $(\pi \circ f)^{-1}(s)$, that is, the intersections of $f(T)$ with the torus fibers of the map $\pi$, induce a singular foliation of $\overline{T}$. The leaves are circles except when $s \in S$ is a critical value of $(\pi \circ f)$, or when $s$ is the image of a puncture point, which henceforth we include in the set of critical values. The fiber $(\pi \circ f)^{-1}(s_0)$ is generically a union of several such circles, so let us consider one component $\rho$. Once we equip $\rho$ with an orientation we obtain a homology class $[\rho]\in H_1(\pi^{-1}(s_0);\Z)$. For $s$ near $s_0$, let $T_{s_0}^s$ denote the parallel transport acting on $H_1(\pi^{-1}(s);\Z)$. Recall the Liouville class $A(s) = [\alpha|_{\pi^{-1}(s)}]\in H^1(\pi^{-1}(s);\R)$. For $s$ near $s_0$, this yields a function 
\begin{equation}
  I_{\rho}(s) = \langle A(s),T_{s_0}^s[\rho] \rangle = \int_{T_{s_0}^s[\rho]} \alpha
\end{equation}

The building block of a holomorphic curve is a \emph{tube} $V$ in $T$, which is to say a compact connected subset that is the union of smooth leaves. For such a tube $V$, the map $(\pi \circ f)|_V : V \to S$ is a circle fibration. The image is some interval $[s_0,s_1]$. Call the boundary leaves $\rho_0$ and $\rho_1$, and orient these as the boundary of $V$. Now, Proposition \ref{prop:positivity-contact-target} implies that $f^*d\alpha$ is a positive two-form on $V$. Thus we obtain
\begin{equation}
  0 < \int_V f^*d\alpha = \int_{\rho_0} f^*\alpha + \int_{\rho_1} f^*\alpha = I_{\rho_0}(s_0) + I_{\rho_1}(s_1)
\end{equation}
Now observe that $\rho_0$ and $\rho_1$ are negatives of each other in the homology of $\pi^{-1}[s_0,s_1]$. Thus we find that $T_{s_0}^s [\rho_0] = -T_{s_0}^s[\rho_1]$ for any $s \in [s_0,s_1]$. Thus we obtain the inequalities
\begin{align}
  \label{eq:monotonicity-left}
  0 <& I_{\rho_0}(s_0) - I_{\rho_0}(s_1)\\
  \label{eq:monotonicity-right}
  0 <& I_{\rho_1}(s_1) - I_{\rho_1}(s_0)
\end{align}
These inequalities require some interpretation. Regard the circle fiber as moving and tracing out the tube. We can think of it either moving to the right (from $s_0$ to $s_1$), or moving to the left (from $s_1$ to $s_0$). Let us make the convention that the circle is oriented as the \emph{boundary of its past}: for the right-moving tube, the circle is oriented like $\rho_1$, and for the left-moving tube, the circle is oriented like $\rho_0$. Now integrate the Liouville class on this loop to obtain the function $I_\rho(s)$. The above inequalities say that $I_\rho(s)$ is \emph{increasing in the direction of motion}. Evidently, this is consistent because reversing the direction of motion also reverses the orientation of the circle.

The tropical analogue of a tube as defined above is a line segment. As motivation, there is an analogue of the preceding inequality coming from elementary plane geometry. Consider any line in $\R^2$ not passing through the origin. At any point along the line, we may compute the dot product between the unit tangent vector (which depends on an orientation of the line) and the unit radial vector pointing away from the origin. This dot product is increasing in the direction of motion (and this fact is independent of the orientation of the line). In the holomorphic-to-tropical correspondence, the radial vector field can roughly be identified with the Liouville class $A(s)$, the unit tangent vector with the class $[\rho]$, and the dot product with the function $I_\rho(s)$.

We wish to associate a line segment in $U^\trop$ to a such tube $V$. First of all, there is a correspondence between the base circle $S$ and the set of rays from the origin in $U^\trop$. For each $s$, the Liouville class $A(s) \in H^1(\pi^{-1}(s);\R)$ determines two subspaces of $H_1(\pi^{-1}(s);\R)$:
\begin{align}
  H_\perp &= \ker \langle A(s), - \rangle\\
  H_\parallel &= \ker \langle A'(s), - \rangle
\end{align}
Theses spaces are rank-one and linearly independent (by the contact condition). There is a natural element $v_{\Reeb}(s) \in H_\parallel$ defined by the condition $\langle A(s), v_{\Reeb}(s) \rangle = 1$, and a positive ray $\R_{\geq 0}\cdot v_{\Reeb}(s)$. Now, the space $U^\trop$ has charts given by quadrants in $H_1$ of the various torus fibers. The correspondence is that $s \in S$ corresponds to the ray in $U^\trop$ that is identified with $\R_{\geq 0}\cdot v_{\Reeb}(s)$ in a chart. The construction of the Liouville class $A(s)$, and particularly the local convexity condition, guarantees that this correspondence is a bijection. Also, this correspondence provides a way to identify the tangent spaces of $U^\trop$ along the ray with $H_1(\pi^{-1}(s);\R)$.


Now we associate a line segment in $U^\trop$ to a tube $V$. For the purposes of this discussion, we allow line segments to be infinite in one or both directions.
\begin{proposition}
  Given a tube $V$, there is a line segment $\ell(V)$ (that is, a path straight with respect to the affine structure) associated to $V$, which is characterized up to radial rescaling in $U^\trop$ by the following properties:
  \begin{enumerate}
  \item $\ell(V)$ lies in the sector of $U^\trop$ corresponding to the interval $\pi(V) = [s_0,s_1] \subset S$, touching each ray sector exactly once.
  \item The tangent vector to $\ell(V)$ is positively proportional to the homology class $\rho$ of the loop that is moving in $V$.
  \end{enumerate}
\end{proposition}

\begin{proof}
The second property requires us to check a consistency condition, since the notion of the ``loop that is moving'' depends up to sign on the orientation of the interval $[s_0,s_1]$. Assuming this, the construction of the segment is simple. The class $\rho$ determines the slope, and we position the line segment in the appropriate sector. The result is well-defined up to radial rescaling.

For concreteness, let us regard the tube as right-moving, from $s_0$ to $s_1$, so that $\rho = \rho_1$. Since we want the line segment to move from the ray for $s_0$ to the ray for $s_1$, a consistency issue arises of whether this tangent vector actually points in the right direction. This can be resolved using the inequality \eqref{eq:monotonicity-right}. At any point $s \in [s_0,s_1]$, the vector $\rho$ can be decomposed uniquely as
\begin{equation}
  \rho = \rho_\perp + \rho_\parallel
\end{equation}
where $\langle A(s),\rho_\perp \rangle = 0$, and $\langle A'(s),\rho_\parallel \rangle = 0$. Let us clarify that, at the chosen point $s$, we fix the decomposition of $\rho$, so the components $\rho_\perp$ and $\rho_\parallel$ are not varying when we differentiate with respect to $s$ (other than by parallel transport between the fibers). The inequality \eqref{eq:monotonicity-right}, taken in the limit $s_0 \to s_1 = s$, implies that $\frac{d}{ds}\langle A(s),\rho \rangle \geq 0$. On the other hand,
\begin{equation}
  \frac{d}{ds}\langle A(s),\rho_{\parallel} \rangle = \langle A'(s),\rho_\parallel \rangle = 0
\end{equation}
So we obtain $\frac{d}{ds}\langle A(s),\rho_\perp \rangle = \langle A'(s),\rho_\perp \rangle \geq 0$. 

The local convexity condition on $A(s)$ implies that $A(s)$ and $A'(s)$ rotate in the same direction with increasing $s$, namely clockwise in our conventions. Thus the rays $\R_{\geq 0}\cdot v_\Reeb(s)$ are rotating clockwise, and the condition $\langle A'(s), \rho_\perp\rangle \geq 0$ implies that $\rho_\perp$ lies on the clockwise side of the ray $\R_{\geq 0} \cdot v_\Reeb(s)$, as needed for the picture to be consistent. See Figure \ref{fig:rotation}.
\end{proof}

We can also say when the line segment will have an infinite direction. If it happens that $\langle A'(s), \rho_\perp \rangle = 0$, then actually $\rho_\perp = 0$, and $\rho = \rho_\parallel$ is a vector parallel to the ray. It is not possible for this to occur at an interior point $s \in (s_0,s_1)$, since Lemma \ref{lem:convex-maximum-property} says that the functional $I_\rho(s) = \langle A(s),\rho\rangle$ has a non-degenerate maximum or minimum at such a point, and so the inequality \eqref{eq:monotonicity-right} will be violated on one side of such a point or the other. It is possible for $\rho = \rho_\parallel$ to occur at a boundary point. If this occurs at $s_1 \in [s_0,s_1]$, then the inequality \eqref{eq:monotonicity-right} implies that $I_\rho(s)$ has maximum at $s_1$, so $\rho$ is positively proportional to $v_\Reeb(s_1)$. In this case, the line segment will actually be infinite and parallel to the ray $\R_{\geq 0}\cdot v_\Reeb(s_1)$. If $\rho = \rho_\parallel$ occurs at $s_0 \in [s_0,s_1]$, then the situation is reversed: $\rho$ is negatively proportional to $v_\Reeb(s_0)$, and the segment is infinite and anti-parallel to $\R_{\geq 0}\cdot v_\Reeb(s_0)$.

\begin{figure}
  \centering
  \includegraphics[width=2.5in]{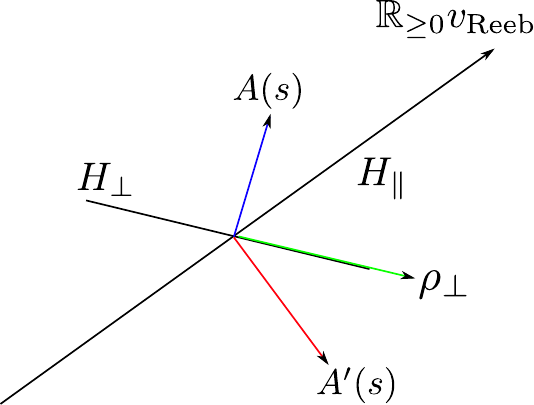}
  \caption{The relationships between $H_\parallel$, $H_\perp$, $A(s)$, $A'(s)$ and $\rho_\perp$.}
  \label{fig:rotation}
\end{figure}

The line segment also has a \emph{weight}, which is the divisibility of the tangent vector $\rho$ in the lattice $H_1(\pi^{-1}(s);\Z)$.

Note that the line segment $\ell(V)$ associated to a tube $V$ is only well-defined up to a rescaling centered at the origin of $U^\trop$ (which should not be interpreted as rescaling the tangent vector $\rho$). When constructing more complicated tropical curves, we can make use of this freedom to make different segments meet each other.

The general prescription for associating a tropical curve to a holomorphic curve $u: T\to \Sigma$ is to decompose the domain $T$ into a collection of tubes. These meet each other over the critical values of the map $\pi \circ u: T \to S$ (we regard the images of the punctures as critical values). At each critical value, several of the tubes interact, and we draw the corresponding segments as having a common endpoint. If necessary we can rescale some of the segments so that they all meet at the same point. In principle, this sort of arbitrary choice can create a consistency issue if the resulting graph contains a cycle, but when considering genus zero operations such as the product, the graph will be a tree.

 There is a \emph{balancing condition} enforced by the topology of $T$ itself. When there is no puncture involved, the balancing condition says that the integral tangent vectors of the segments meeting at the vertex, taken in the direction away from the vertex, with the appropriate weights, sum to zero. This is because the integer tangent vectors multiplied by the weights correspond to the classes of the loops in the various tubes, and the preimage of a small interval around the critical value yields a two-chain bounded by the sum of these loops. The situation is slightly different if there are punctures involved. If the critical value is at $s \in S$, then the punctures must be asymptotic to Reeb orbits at this value of $s$. The homology class of such a loop is proportional to $v_\Reeb(s)$, which is the radial direction. Therefore, the balancing condition holds if we add radial segments emanating from the vertex in question, pointing into the origin in the case of negative punctures, and away from the origin in the case of positive punctures. This gives us the \emph{broken line diagram of $u: T \to V$}. 
 
\begin{remark}
  In this paper, we will only use broken line diagrams associated genus zero curves appearing in the operations involved in symplectic cohomology, but one could ask if such a diagram could be associated to a more general class of curves. Indeed, the broken line diagram, as we have defined it, is a topological approximation to the tropicalization of the curve, which should exist more generally. In the case of a truly arbitrary curve, say without a bound on its energy, one would have to worry about the possibility that the diagram could have infinitely many edges. Also, as we have said, in the case of higher genus curves one would need to worry about whether the various segments associated to the tubes can really be made to meet inside $U^{\trop}$. We will set aside these questions for now since they are not necessary for this paper.
\end{remark}

The prescription described above gives us a diagrammatic way to represent curves contained in the cylindrical end $\Sigma \times \R$ of $U$, and it is possible to read off the homology class of such a curve from this diagram. 

\begin{proposition}
  \label{prop:broken-line-to-homology-class}
  The homology class associated to a broken line diagram may be obtained as a sum of local contributions computed as follows. At each point where the broken line diagram crosses a ray corresponding to $v_{\Reeb}(s_{i})$, take the tangent vector $\rho$ to the diagram and project it into the rank-one lattice $TU^{\trop}(\Z)/v_{\Reeb}(s_{i})$. If $m \geq 1$ denotes the divisibility of the image of $\rho$ in this lattice, then the local contribution is $m D_{i}$. In the case where the intersection of the diagram with the ray corresponding to $v_{\Reeb}(s_{i})$ is singular, then we compute the same quantity slightly to either side of the ray (and the result is independent of this choice).
\end{proposition}

\begin{proof}
  Lemma \ref{lem:symplectization-level} shows that a curve in $\Sigma \times \R$ represents a homology class that is a non-negative combination of the $D_i$. The proof furthermore shows that the coefficients in this decomposition can be obtained projecting the curve onto the divisor $D_i$ near the region where the Reeb orbits are the normal circles to $D_i$. The point is that this degree can also be computed by looking at the homology class $\rho \in H_1(\pi^{-1}(s);\Z)$ carried by the tubes entering this region. Let $s_i \in S$ be the point where the Reeb orbits are normal circles to $D_i$; they represent the homology class $v_\Reeb(s_i)$; in particular, $v_\Reeb(s_i)$ is an integral class. In the portion of the fibration over $(s_i -\epsilon, s_i+\epsilon)$, choose an integral basis of sections of $H_1(\pi^{-1}(s);\Z)$ consisting of $v_\Reeb(s_i)$ and another class $v_b$. Of course, in the choice of $v_b$ there is freedom to add an integer multiple of $v_\Reeb(s_i)$ or reverse the sign. Now decompose the homology class $\rho$ as $\rho = nv_\Reeb(s_i) + mv_b$. The projection $r$ of $\pi^{-1}(s_i-\epsilon, s_i+\epsilon)$ onto $D_i$ collapses the circle $v_\Reeb(s_i)$, and maps $v_b$ to $r(v_b)$, the homology class of a loop on the divisor. Thus it maps $\rho = nv_\Reeb(s_i) + mv_b$ to $mr(v_b)$. The arguments of Lemma \ref{lem:symplectization-level} show that the projection of $\rho$ to $r(v_b)$ has the same degree at every point, either $1$ or $-1$ depending on orientations, so that the total degree is either $m$ or $-m$. Since we already know that the degree is non-negative, we find that it equals $|m|$, and tube carrying the homology class $\rho$ contributes $|m|D_i$ to the total homology class of curve. It is clear that $|m|$ is also the divisibility of $\rho$ in the rank one lattice $TU^{\trop}(\Z)/v_{\Reeb}(s_{i})$.

  If the intersection of the graph with the ray $v_\Reeb(s_i)$ is singular (if it coincides with at vertex of the graph), then we compute the same quantity slightly to either side of the ray, and the balancing condition guarantees that the result is well-defined.
\end{proof}

\begin{example}
  Consider the broken line diagram shown in Figure \ref{fig:broken-line}. The only ray that it crosses is the one corresponding to $D_i$. When it crosses this ray, its tangent vector is $(1,-1)$, which projects to $1$ in $\Z^2/(0,1)$. Therefore $m = 1$, and the homology class associated to this curve is $[D_i]$.
\end{example}

\subsection{Degeneration to the vertex}
\label{sec:vertex}

Let $n$ denote the number of irreducible components of $D$. The \emph{vertex} $\mathbb{V}_n$ is the singular algebraic surface consisting of $n$ copies of $\Af^2$ intersecting along coordinates axes, forming a cycle. Thus if $x_1,\dots, x_n$ are variables, and $\Af^2_{x_i,x_j} = \Spec \K[x_i,x_j]$, we have
\begin{equation}
  \mathbb{V}_n = \Af^2_{x_1,x_2}\cup \Af^2_{x_2,x_3} \cup \cdots \cup \Af^2_{x_n,x_1}
\end{equation}
The ring of functions on $\mathbb{V}_n$ is generated by the variables $x_1,\dots,x_n$, subject to the conditions that $x_i$ and $x_{i+1}$ generate a polynomial algebra (with indices taken modulo $n$), and that two non-consecutive variables (such as $x_1$ and $x_3$) multiply to zero. 

As shown in \cite{ghk}, there is a relationship between a log Calabi--Yau pair $(Y,D)$ such that $D$ has $n$ irreducible components and a deformation of the vertex $\mathbb{V}_n$. In the context of this paper, we may state the relationship as follows. Recall that, since $P \subset H_2(Y;\Z)$ is a strictly convex cone, the monoid ring $\K[P]$ has a maximal ideal $\mathfrak{m}_P$ generated by all nonzero elements of $P$. Being a $\K[P]$--algebra, the degree-zero symplectic cohomology defines a family $\Spec SH^0(U;\K[P]) \to \Spec \K[P]$.

Now we can state and prove the result alluded to in the introduction. It states that the spectrum of the symplectic cohomology ring $SH^0(U;\K[P])$ is a deformation of the vertex $\mathbb{V}_n$. The analogous result with the ring of theta functions instead of symplectic cohomology was proven in \cite{ghk}. Thus, this result can be regarded as further evidence for the correspondence between the ring of theta functions and the symplectic cohomology of $U$.

\begin{theorem}
  \label{thm:degeneration-to-vertex}
  The fiber at $\mathfrak{m}_P \in \Spec \K[P]$ of $\Spec SH^0(U;\K[P])$ is isomorphic to $\mathbb{V}_n$. The variables $x_i$ correspond to periodic orbits encircling once each of the irreducible components of $D$.
\end{theorem}

Since $SH^0(U;\K[P])$ is a free $\K[P]$--module, taking the fiber at $\mathfrak{m}_P$ yields a free $\K$--vector space on the same basis. The difference is in the product structure, since in the fiber at $\mathfrak{m}_P$ we must set to zero any contributions to the product that involve non-zero elements of $P$.

In considering such products, we begin again with an analysis of the possible limit configurations under the neck--stretching process. Lemmas \ref{lem:interior-level} and \ref{lem:symplectization-level} imply that each level individually represents a class in $P$. Since $P$ is strictly convex, the only way that the total homology class can be zero is if each level individually represents the zero homology class. In fact, Lemma \ref{lem:interior-level} shows that the interior level actually lives in $\mathfrak{m}_P$. 

On the other hand, there are non-zero contributions to the product in the zero homology class. Near each node of $D$, there is a collection of periodic orbits corresponding to the integral points in a quadrant. Just as in the case of $(\C^\times)^2$, there are pairs-of-pants in a neighborhood of a node, and these give rise to a product that corresponds to addition of integral points. In $U$ itself, there may be other contributions to the product, and the point is to show that these others have nonzero homology classes. What is needed is an enhancement of Lemma \ref{lem:symplectization-level}, dealing with a pair-of-pants representing the zero homology class, saying that such curves cannot connect periodic orbits near different nodes of $D$.

Recall (as in the proof of Lemma \ref{lem:symplectization-level}), that for each divisor $D_i$, there is a point $s_i \in S$, such that $\pi^{-1}(s_i) \subset \Sigma$ is a torus in which the Reeb orbit is the normal circle to $D_i$. The next Lemma says that these points act as barriers that curves in the symplectization cannot cross unless they have a non-zero homology class in $Y$.

\begin{lemma}
  \label{lem:zero-homology-class}
  Let $u : T \to \Sigma \times \R$ be a inhomogeneous pseudo-holomorphic curve in the symplectization, with connected genus zero domain $T$, such that $\phi(u) = 0\in H_2(Y;\Z)$. Then there is some $i$ such that $u$ is localized near the node $D_i \cap D_{i+1}$: precisely, that $\pi(u(T))$ is contained in the interval $[s_i,s_{i+1}]$.
\end{lemma}

\begin{proof}
  We use the broken line diagrams described in Section \ref{sec:drawing}. The homology class being zero implies that the broken-line diagram of such a $u$ cannot cross any of the rays $v_\Reeb(s_i)$ corresponding to the points where the Reeb orbit is a normal circle to $D_i$. Thus the broken line diagram lies entirely inside one quadrant of $U^\trop$, and the conclusion follows.
\end{proof}

\begin{remark}
  Alternatively, and what amounts to the same thing, we can argue that, if the curve touches the torus $\pi^{-1}(s_i)$, then in order for there to be no term of $D_i$ in $\phi(u)$, the intersection of $u(T)$ with $\pi^{-1}(s)$ must be homologous to a multiple of $v_\Reeb(s_i)$, for $s$ near $s_i$. The inequalities \eqref{eq:monotonicity-left} or \eqref{eq:monotonicity-right} applied to the functional $I_{v_\Reeb(s_i)}$ say that this functional is strictly monotonic along the curve. On the other hand $I_{v_\Reeb(s_i)}$ has maximum at $s_i$, so the curve cannot pass through this point.
\end{remark}

To finish the proof of Theorem \ref{thm:degeneration-to-vertex}, we use that the contributions to the symplectic cohomology product that lie entirely in the neighborhood of the one of the nodes $D_i \cap D_{i+1}$ correspond to the products in $SH^0((\C^\times)^2;\K)$, where we know that the two elements $\theta_i$ and $\theta_{i+1}$ corresponding to the loops around the two divisors generate a polynomial algebra. Lemma \ref{lem:zero-homology-class} shows that all other products vanish, so we recover the ring of functions on $\mathbb{V}_n$ from the symplectic cohomology by setting all products with nonzero classes in $P$ to zero.

\begin{remark}
  The results of this section apply to the case where $Y$ is a toric surface, and $D$ is the toric boundary divisor. In that case $U \cong (\C^\times)^{2}$, and $SH^{0}(U;\K)$ is isomorphic to the ring of Laurent polynomials. Even in this case, the deformation described above is not trivial, and it depends on the toric variety $Y$. This is due to the fact that the curves contributing to products between generators living near different nodes of $D$ carry nontrivial classes in $P$. For instance, if $Y = \PP^{1}\times \PP^{1}$ and $D$ is the toric boundary, then $\K[P]= \K[q^{a_{1}},q^{a_{2}}]$, where $a_{i}$ is the homology class of the $i$-th factor. Then $SH^{0}(U;\K[P])$ is generated by four orbits $x,y,x',y'$, linking each of the four components of $D$, and the relations are
\begin{equation}
  xx' = q^{a_{1}},\quad yy' = q^{a_{2}}.
\end{equation}
Setting $q^{a_{1}}=q^{a_{2}}=1$, one obtains the ring of Laurent polynomials, but setting $q^{a_{1}}=q^{a_{2}}=0$, one obtains the ring of functions on
\begin{equation}
  \mathbb{V}_{4} = \Af^{2}_{x,y} \cup \Af^{2}_{y,x'}\cup \Af^{2}_{x',y'} \cup \Af^{2}_{y',x}.
\end{equation}
\end{remark}

\bibliographystyle{amsplain}
\bibliography{symplectic}

\end{document}